\documentclass{article}

\usepackage{amsmath, amsthm, amssymb, dsfont, mathtools}
\usepackage{graphicx}
\usepackage{verbatim}
\usepackage{natbib}
\usepackage{caption}
\usepackage{subcaption}
\usepackage{fancyvrb}
\usepackage{enumerate}
\usepackage[inline]{enumitem}
\usepackage{relsize}
\usepackage{hyperref}
\usepackage[margin=1.45in]{geometry}
\hypersetup{colorlinks,citecolor=blue,urlcolor=blue,linkcolor=blue}
\usepackage{diagbox}
\usepackage{nikos_tex}
\usepackage{float}
\usepackage{microtype}

\usepackage{multirow}

\usepackage[utf8]{inputenc}
\usepackage[english]{babel}

\usepackage[skins,theorems]{tcolorbox}
\tcbset{highlight math style={enhanced,
  colframe=red,colback=white,arc=0pt,boxrule=1pt}}

\graphicspath{{figures/}}

\theoremstyle{plain}
\newtheorem{prop}{Proposition}

\newtheorem{coro}[prop]{Corollary}
\newtheorem{lemm}[prop]{Lemma}
\newtheorem{theo}[prop]{Theorem}

\newtheorem{manualtheoreminner}{Theorem}
\newenvironment{manualtheorem}[1]{%
  \renewcommand\themanualtheoreminner{#1}%
  \manualtheoreminner
}{\endmanualtheoreminner}

\newtheorem{manualpropinner}{Proposition}
\newenvironment{manualprop}[1]{%
  \renewcommand\themanualpropinner{#1}%
  \manualpropinner
}{\endmanualpropinner}

\theoremstyle{remark}
\newtheorem{rema}[prop]{Remark}
\newtheorem{assu}[prop]{Assumption}

\newtheorem{exam}[prop]{Example}

 \usepackage{epigraph}

 \usepackage{etoolbox}
 
 \makeatletter
 \patchcmd{\epigraph}{\@epitext{#1}}{\itshape\@epitext{#1}}{}{}
 \makeatother

\usepackage{algorithm2e}
\RestyleAlgo{boxruled}
\LinesNumbered
\DontPrintSemicolon
\SetAlFnt{\normalsize}
\SetKwInOut{Input}{Input}

\date{September, 2024}

\title{Empirical partially Bayes multiple testing \\ and compound $\chi^2$ decisions}

\author{
Nikolaos Ignatiadis\\
 \texttt{ignat@uchicago.edu}
 \and
 Bodhisattva Sen \\
 \texttt{bodhi@stat.columbia.edu}
}

\begin{document}

\maketitle

\begin{abstract}
A common task in high-throughput biology is to screen for associations across thousands of units of interest, e.g., genes or proteins. Often, the data for each unit are modeled as Gaussian measurements with unknown mean and variance and are summarized as per-unit sample averages and sample variances. The downstream goal is multiple testing for the means. In this domain, it is routine to ``moderate'' (that is, to shrink) the sample variances through parametric empirical Bayes methods before computing p-values for the means. Such an approach is asymmetric in that a prior is posited and estimated for the nuisance parameters (variances) but not the primary parameters (means). Our work initiates the formal study of this paradigm, which we term ``empirical partially Bayes multiple testing.'' In this framework, if the prior for the variances were known, one could proceed by computing p-values conditional on the sample variances---a strategy called partially Bayes inference by Sir David Cox. We show that these conditional p-values satisfy an Eddington/Tweedie-type formula and are approximated at nearly-parametric rates when the prior is estimated by nonparametric maximum likelihood. The estimated p-values can be used with the Benjamini-Hochberg procedure to guarantee asymptotic control of the false discovery rate. Even in the compound setting, wherein the variances are fixed, the approach retains asymptotic type-I error guarantees.
\\ 


\end{abstract}

\section{Introduction} 
\label{sec:introduction}
We study an empirical Bayes method 
for testing $n$ null hypotheses,
$H_i: \mu_i=0$, in the normal means problem with estimated variances.
To be concrete, we assume that for the $i$-th unit we observe a Gaussian measurement
$Z_i \sim \nn(\mu_i,\; \sigma_i^2)$ centered around $\mu_i$ with variance $\sigma_i^2$, as well as
an independent unbiased measurement $S_i^2$ of $\sigma_i^2$ whose law is the scaled $\chi^2$ distribution
with $\nu \geq 2$ degrees of freedom:
\begin{equation}
    \label{eq:full_sampling}
     (Z_i, S_i^2) \, \simindep \, \mathcal{N}(\mu_i, \sigma_i^2) \otimes \frac{\sigma_i^2}{\nu} \chi^2_{\nu}\;\;\,\text{ for } i=1,\dotsc,n.
\end{equation}
Here $\sigma_i^2 > 0$ and $\mu_i \in \RR$ are unknown. $Z_i$ and $S_i^2$ arise from data summarization
in Gaussian linear models.
As one concrete example,
suppose that for the $i$-th unit
we observe \smash{$Y_{ij} \mid \mu_i, \tilde{\sigma}_i^2 \simiid \nn(\mu_i, \tilde{\sigma}_i^2)$} for $j=1,\dotsc,K$.
Then $Z_i \coloneqq \sum_{j=1}^K Y_{ij}/K$ and $S_i^2 \coloneqq \sum_{j=1}^K (Y_{ij}- Z_i)^2/(K^2-K)$
satisfy~\eqref{eq:full_sampling} with $\mu_i=\mu_i, \sigma_i^2 =  \tilde{\sigma}_i^2/K$ and $\nu=K-1$.
Data of the form~\eqref{eq:full_sampling} are ubiquitous
and arise naturally in high-throughput biological studies, 
for example in transcriptomics~\citep{Lonnstedt166664, smyth2004linear,ritchie2015limma}, 
proteomics~\citep{kammers2015detecting, terkelsen2021high,zheng2021mixtwice}, and methylomics~\citep{zhang2013genomewide,Maksimovic2017f1000};
cf. Sections~\ref{subsec:limma},~\ref{subsec:applicability}, ~\ref{subsec:methylation}, and~\ref{sec:applications}.

The crux of a standard empirical Bayesian~\citep{robbins1956empirical, efron2010largescale} analysis with 
data from~\eqref{eq:full_sampling} is to 
further posit that $(\mu_i, \sigma_i^2)$ are $\text{iid}$ draws from an unknown distribution $\Pi$ 
and to then make statistical decisions by imitating
the actions of an oracle Bayesian who knows $\Pi$.
The imitation step can be accomplished by estimating $\Pi$ based on $(Z_1, S_1^2), \dotsc, (Z_n, S_n^2)$
as \smash{$\widehat{\Pi}$} and then pretending that \smash{$(\mu_i, \sigma_i^2) \simiid \widehat{\Pi}$}. 

This is \emph{not} the empirical Bayes approach we pursue in this paper. 
Instead, we imitate an oracle Bayesian that is a \emph{``partial Bayesian''}~\citep{cox1975note, mccullagh1990note} 
who posits a distribution $G$ for the nuisance parameters $\sigma_i^2$,
\begin{equation}
    \label{eq:random_sigma}
\sigma_i^2 \simiid G,\,\,\, \text{for } i=1,\dotsc,n,
\end{equation}
but treats the parameters of primary interest, $\mu_i$, as fixed and unknown.\footnote{
\label{footnote:independence}
If \smash{$\mu_i$} is treated as random, then the statement in~\eqref{eq:random_sigma} should be interpreted as \smash{$\sigma_i^2 \mid \mu_i \sim G$}, which implies that \smash{$\sigma_i^2$} and \smash{$\mu_i$}  are independent. 
This assumption will play an important role in what follows. We will show that our proposed procedure is robust to violations of the independence assumption, especially with respect to type-I error guarantees 
(see, e.g.,~Section~\ref{sec:joint_hierarchical}).}
Moreover, for the partial Bayesian that we seek to imitate the \emph{conditionality
principle} is in full force  (see explanation in Section~\ref{sec:conditional_pvalues}): 
their inferences for $\mu_i$ are based on the conditional distribution of $Z_i \mid S_i^2$,
for example, a two-sided p-value for testing $H_i: \mu_i=0$ is given by $P_i = P_G(Z_i, S_i^2)$,
\begin{equation}
    \label{eq:conditional_pvalue}
P_G(z, s^2) \coloneqq \PP[G]{ \abs{Z^{H_0}} \geq \abs{z}\; \mid \;  S^2=s} = \EE[G]{2(1-\Phi(|z|/\sigma)) \mid S^2=s^2}.
\end{equation}
Above $\Phi$ is the standard normal distribution function and we write \smash{$(\sigma^2, Z^{H_0}, S^2)$} for 
the triple generated as follows: $\sigma^2$ is generated according to $G$, as assumed in~\eqref{eq:random_sigma}, and \smash{$(Z^{H_0}, S^2)$} is generated as assumed in~\eqref{eq:full_sampling} 
conditional on $\sigma^2$ and $\mu=0$. In a purely frequentist analysis
with $\sigma^2$ treated as fixed, the distribution of $Z^{H_0}$ conditional on $S^2$
is \emph{not pivotal}, since $Z^{H_0} \mid S^2, \sigma^2 \sim \nn(0, \sigma^2)$. 
However, the partial Bayesian can marginalize over $\sigma^2 \mid S^2$, and so 
the distribution of $Z^{H_0} \mid S^2$ becomes pivotal. 
The conditional p-value $P_G(Z, S^2)$ has the following elegant interpretation in view of the 
right hand side of~\eqref{eq:conditional_pvalue}:  
it is a weighted average of standard z-test based p-values in the case of known variance $\sigma^2$ 
with weighting determined by the posterior of $\sigma^2$ conditional on $S^2$.

Our proposal is as follows: in accordance to the empirical Bayes
principle, we estimate $G$ by \smash{$\hG$} based on \smash{$S_1^2,\dotsc,S_n^2$} using the nonparametric maximum likelihood estimator~\citep[NPMLE]{kiefer1956consistency}. Then we
conduct partially Bayes inference \emph{conditional} on \smash{$S_1^2,\dotsc,S_n^2$}
with the prior \smash{$\hG$}. That is, we form the conditional p-values \smash{$P_i = P_{\hG}(Z_i, S_i^2)$} in~\eqref{eq:conditional_pvalue}
for $i=1,\dotsc,n$. Finally, we apply the method of~\citet[BH]{benjamini1995controlling} 
on these p-values to control the false discovery rate.
We call the resulting approach
to testing \emph{``empirical partially Bayes multiple hypothesis testing''}---such empirical Bayes developments had been explicitly 
anticipated by Sir David Cox~\citeyearpar{cox1975note} 
(more on this historical connection in Section~\ref{sec:conditional_pvalues}).
Algorithm~\ref{algo:np_limma} provides a high-level description of our approach.

In this paper, we demonstrate, that the empirical partially Bayes multiple testing approach has 
substantial advantages compared to alternative approaches for the same problem. 

\textbf{Comparison with standard t-tests.} A standard solution to the multiple testing task with data from~\eqref{eq:full_sampling}
is to compute p-values $P_i^{\text{ttest}}$ based on the t-statistics $T_i = Z_i/S_i$. The latter are pivotal 
under the null, that is, $T_i \sim t_{\nu}$ for $\mu=0$, where $t_{\nu}$ is the t-distribution
with $\nu$ degrees of freedom.  When the degrees of freedom $\nu$ are small, 
as is common in applications (we give concrete examples later), the t-test may have almost no power 
(due to the heavy tails of the t-distribution with few degrees of freedom). In such cases,
the empirical partially Bayes p-values $P_i$ (see~\eqref{eq:conditional_pvalue})  often lead to substantially more powerful tests.
For example, if $G$ in~\eqref{eq:random_sigma} is very concentrated around a fixed $\bar{\sigma}^2$,
the $P_i$ will be approximately equal to the p-value of the standard z-test with \emph{known} variance 
(as formalized in Theorem~\ref{theo:pvalue_quality} and Example~\ref{exam:same_variance} below). Furthermore, in Proposition~\ref{prop:ttest_failure} we show that
the conditional type-I error of the standard t-test for a small observed value of $S_i^2$ can be arbitrarily inflated. 
Such a property is undesirable in view of the conditionality principle; we \emph{ do not } want to prioritize
a hypothesis that had a small sample variance by chance.
The empirical partially Bayes approach avoids the aforementioned caveat since the p-values are computed \emph{conditional}
on $S_i^2$ (see~\eqref{eq:conditional_pvalue}). 
The upshots of the empirical partially Bayes approach compared
to standard t-tests, however, come at a cost: if we apply the Benjamini-Hochberg (BH) procedure with the t-test p-values $P_i^{\text{ttest}}$,
then the false discovery rate is controlled at level $\alpha$ for any finite $n$. In contrast, if we apply the BH procedure with the empirical partially Bayes p-values $P_i$, then we are only guaranteed control of the false discovery rate under the conditions of Theorems~\ref{theo:limma_bh_controls_the_FDR} and~\ref{manualtheorem:limma_bh_controls_the_FDR} as $n \to \infty$. 

\textbf{Comparison with a full empirical Bayes treatment.}
As mentioned in the beginning of this paper, an alternative approach for multiple testing with data from~\eqref{eq:full_sampling} 
treats the problem in a fully empirical Bayes fashion by
positing \smash{$(\mu_i,\;\sigma_i^2) \simiid \Pi$} and 
estimating \smash{$\widehat{\Pi}$}~\citep{lonnstedt2002replicated, ploner2006multidimensional, hwang2010optimal, lu2019empirical, zheng2021mixtwice}.
Asymptotically optimal testing procedures~\citep{sun2007cai} can then be derived by thresholding the (estimated)
local false discovery rate $\PP[\widehat{\Pi}]{\mu_i=0 \mid Z_i, S_i^2}$. 
By contrast, the approach we propose has three main advantages:\footnote{A fourth advantage was pointed to us by Roger Koenker, who suggested calling our approach ``impartially Bayes.'' By test inversion, our proposed p-values can be used to form intervals for the $\mu_i$s. These intervals may be shorter than standard t-test based intervals (due to the benefits of shrinkage), however, they are still centered at $Z_i$ unlike fully empirical Bayes intervals that would also shrink the $Z_i$. Such centering may be desirable in applications in which one is worried about introducing bias to the primary parameters $\mu_i$. 
} first, it is computationally 
streamlined---the most intensive step is the solution of a nonparametric maximum likelihood task for a one-dimensional
distribution; there are established software solutions for this task~\citep{koenker2017rebayes, kim2020fast}.
Given such software solutions, Algorithm~\ref{algo:np_limma} can be implemented in a few lines of code. 
On the other hand, solving a nonparametric maximum likelihood problem over two dimensions is more difficult~\citep{gu2017unobserved}---albeit still tractable---and
may require the practitioner
to take a stance e.g., on possible independence of $\sigma_i^2$ and $\mu_i$ as assumed e.g., by~\citet{zheng2021mixtwice}.
Second, in Theorem~\ref{manualtheorem:limma_bh_controls_the_FDR} we prove that Algorithm~\ref{algo:np_limma} asymptotically
controls the false discovery rate even if all the parameters $\mu_1,\dotsc,\mu_n$ and $\sigma_1^2,\dotsc,\sigma_n^2$ are in fact fixed 
and not random. We are not aware of an analogous result for empirical Bayes approaches that 
threshold $\PP[\widehat{\Pi}]{\mu_i=0 \mid Z_i, S_i^2}$.\footnote{In recent breakthrough work,~\citet{castillo2020spike} derive such guarantees 
when $\sigma_i^2$ is known to the analyst and equal to $1$ for all $i$. It is plausible---but far from obvious---that such 
results could be extended to the setting we consider wherein the variances are heteroscedastic and unknown.} Third, the empirical partially Bayes approach 
leads to an intermediate computation
of p-values. This can be an advantage for practitioners who may be more comfortable reporting results
that are accompanied by p-values~\citep{jeffery2006comparison}.

\begin{algorithm}[t]
    \caption{\textbf{Empirical partially Bayes multiple hypothesis testing.} 
    }
    \label{algo:np_limma}
  \Input{Pairs $(Z_i, S_i^2), \; i=1,\dotsc, n$.\\
         Nominal false discovery rate control level $\alpha \in (0,1)$.} 
    Let $\hG$ be an estimate of the distribution of $\sigma_i^2$ based on $(S_1^2,\dotsc,S_n^2)$ only, e.g.,~\eqref{eq:npmle_opt}.\; 
    Let $P_i = P_{\hG}(Z_i, S_i^2)$, where $P_G(z,s^2) = \EE[G]{2(1-\Phi(|z|/\sigma)) \mid S^2=s^2}$ is defined in~\eqref{eq:conditional_pvalue} and $\Phi$ is the standard normal distribution function.\; 
    Apply the Benjamini-Hochberg (BH) procedure with p-values $P_1,\dotsc,P_n$ at level $\alpha$ as in~\eqref{eq:BH} to declare discoveries.
\end{algorithm}

\subsection{Limma: Parametric empirical partially Bayes multiple testing}
\label{subsec:limma}
The empirical partially Bayes multiple testing approach is not 
new---a parametric instantiation of the framework has found broad usage and success in genomics applications.
\citet{Lonnstedt166664} and \citet{lonnstedt2002replicated} proposed a
model for detecting differentially expressed genes with replicated microarray data. In turn, \citet{smyth2004linear} generalized this model and proposed (in our terminology) a parametric empirical partially Bayes multiple testing approach implemented as the R package \texttt{limma}. This package has become 
the de-facto testing approach for the analysis of microarray data 
(\citet{smyth2004linear} has over 13,000 citations to date), and is also commonly used 
for the analysis of RNA-Seq data with
over 27,000 citations in~\citet{ritchie2015limma}.
Limma has also been successfully applied in many other areas, e.g., for the analysis of methylation data~\citep{Maksimovic2017f1000}, 
in proteomics~\citep{kammers2015detecting, terkelsen2021high}, and
genetic interaction screening~\citep{laufer2013mapping}.

The model underlying limma specifies $G$ in~\eqref{eq:random_sigma} parametrically as
\begin{equation}
\label{eq:limma_shrinkage_assumption}
\frac{1}{\sigma_i^2} \simiid \frac{1}{\nu_0 s_0^2} \chi^2_{\nu_0},\,\,\, \text{for } i=1,\dotsc,n,
\end{equation}
where $s_0^2, \nu_0 >0$ are estimated in an empirical Bayes fashion as $\hat{s}_0^2, \hat{\nu}_0$,
e.g., by the method of moments~\citep{smyth2004linear} or maximum marginal likelihood~\citep{wright2003random, menezes2004microarray}. 
The p-value for the $i$-th unit is defined as $P_i = P^{\text{limma}}_{\hat{\nu}_0, \hat{s}_0^2}(Z_i, S_i)$, where
\begin{equation}
\label{eq:moderated_t}
P^{\text{limma}}_{\nu_0, s_0^2}(z, s^2) \coloneqq 2\bar{F}_{t,\nu_0 + \nu}(|\tilde{t}|),\,\;\,\; \tilde{t} = \frac{z}{\tilde{s}}, \;\,\,\; \tilde{s}^2 = \frac{\nu_0 s_0^2 + \nu s^2}{\nu_0 + \nu},
\end{equation}
and $\bar{F}_{t,\nu_0 + \nu}$ is the survival function of the t-distribution with $\nu_0 + \nu$ degrees of freedom.
The formal argument justifying the limma p-values is the following: when $(Z_i, S_i^2)$ are sampled 
according to~\eqref{eq:full_sampling} with $\mu_i=0$ and we marginalize over $\sigma_i^2 \sim G$ in~\eqref{eq:limma_shrinkage_assumption}, 
then \smash{$\widetilde{T}_i$}, i.e., the random variable with realization 
$\tilde{t}$ in~\eqref{eq:moderated_t}, follows 
the t-distribution with $\nu_0 + \nu$ degrees of freedom 
(see Proposition~\ref{prop:limma_as_conditional} below).

In Section~\ref{sec:conditional_pvalues} we will see that limma is a special case
of the general approach set forth in Algorithm~\ref{algo:np_limma}.
The test statistic underlying limma, \smash{$\widetilde{T}_i$}, is called a ``moderated'' (that is, shrunken) t-statistic.
While the standard t-statistic $T_i = Z_i/S_i$  is 
studentized by $S_i^2$, the moderated t-statistic is studentized by $\tilde{S_i}^2$, a convex
combination of $S_i^2$ and $s_0^2$. At a heuristic level, the moderated t-statistic leads to an 
increase in power compared to the standard t-statistic because p-values are computed with 
respect to the t-distribution with $\nu + \nu_0$ degrees of freedom (rather than $\nu$, as in the standard
t-test). The increase in degrees of freedom may be attributed to the additional information we retrieve for each hypothesis
by ``borrowing'' information across units.
Furthermore, excessively small values of $S_i^2$ are pulled upward toward $s_0^2$, and so a hypothesis
with very small $S_i^2$ is less likely to be rejected (compared to the standard t-test).

In this paper we go beyond the parametric assumption for the distribution of $\sigma_i^2$ in~\eqref{eq:limma_shrinkage_assumption}
and instead propose to estimate $G$ (defined in~\eqref{eq:random_sigma}) by the nonparametric maximum likelihood estimator (NPMLE, see Section~\ref{sec:empirical-bayes}). Methodologically we thus provide an alternative
in cases wherein the parametric assumption is unlikely to hold. From a theoretical perspective, in Theorem~\ref{theo:hellinger_dist_convergence} 
we study the statistical properties of the NPMLE with scaled $\chi^2$-data and derive parametric rates (up to logarithmic factors) of Hellinger consistency for the marginal density
that are analogous to related results for the NPMLE with Gaussian data~\citep{ghosal2001entropies,zhang2009generalized, soloff2021multivariate}, 
or Poisson data~\citep{jana2022optimal}.
Furthermore, our results put forth a theoretical
foundation for understanding empirical partially Bayes multiple testing methods and their success in applications.

In Proposition~\ref{prop:pvalue} we demonstrate that the oracle partially Bayes p-values satisfy an Eddington/Tweedie-type formula~\citep{dyson1926method, efron2011tweedie}
which along with Theorem~\ref{theo:hellinger_dist_convergence} enables us to prove (Theorem~\ref{theo:pvalue_quality}) that the oracle partially Bayes p-values
can be approximated via empirical Bayes at parametric rates up to logarithmic factors, even when
$G$ in~\eqref{eq:random_sigma} is specified fully nonparametrically. In Theorem~\ref{manualtheorem:limma_bh_controls_the_FDR}, 
we prove the remarkable property
that our proposed method can asymptotically control the false discovery rate even when the $\sigma_i^2$ are in fact fixed (and not random). 
Our formal results are corroborated by the data analyses in Sections~\ref{subsec:methylation} and~\ref{sec:applications}, as well as the 
simulation study in Section~\ref{sec:simulations}.

\subsection{Applicability and relevance of the proposed framework to high-throughput biology}
\label{subsec:applicability}
The data requirements for applying the proposed empirical partially Bayes methods are very particular: the data for each unit is summarized as $(Z_i, S_i^2)$ that follows the sampling distribution in~\eqref{eq:full_sampling}, the number of units $n$ needs to be large, and the approach is most beneficial when the degrees of freedom $\nu$ are small. We briefly explain why these data requirements are approximately met in many applications in high-throughput biology. We first elaborate on the large-$n$, small-$\nu$ provenance. 

A principal characteristic of high-throughput biology is the large-scale generation of data. In fields such as genomics, proteomics, epigenomics, and metabolomics, thousands of genes, proteins, methylation probes, or metabolites can be analyzed simultaneously (thus $n$ is large). On the other hand, high-throughput measurements are often only made for a small number of biological replicates (e.g., distinct mice, distinct human cell lines); thus $\nu$ is small. There are many reasons for the small number of biological replicates. For one, it is common for a high-throughput experiment to serve as an initial, potentially cost-effective, screen to be followed by more targetted and comprehensive experimental validation~\citep{yu2013shrinkage}. The availability of biological specimen may be limited. Finally, the cost of even a single run of a high-throughput technology can be high, and so the number of replicates is limited by economic constraints~\citep{kooperberg2005significance}.\footnote{Technical advances can reduce cost over time, however, scientists are also inclined to ask more fine-grained questions that in turn require newer and more expensive genomic technology~\citep{stark2019rna}.}

We next elaborate on the normality assumption underlying~\eqref{eq:full_sampling}. The following observation, explicated by~\citet{smyth2004linear}, lies at the core of the wide applicability of the empirical partially Bayes approach:  $(Z_i, S_i^2)$ distributed as in~\eqref{eq:full_sampling} arise from data summarization in Gaussian linear models with repeated design where the goal is to test for a certain contrast. (We omit the proof of the following proposition.)

\begin{prop}[Testing for contrasts in linear regression with repeated design]
\label{rema:contrasts}
    Suppose our data for the $i$-th unit is modelled by a linear regression
    with $K$ samples and $p$ covariates ($K > p$), i.e.,
    \begin{equation}
    \label{eq:linear_model}
    \mathbf{Y}_i \cond \mathbf{X} \sim \nn\p{\mathbf{X} \boldsymbol{\beta}_i,  \, \tilde{\sigma}_i^2  \mathbf{I}_K},
    \end{equation}
    where $\mathbf{Y}_i  \in \RR^K$ and $\mathbf{X} \in \RR^{K \times p}$ is full rank. Here $\mathbf{I}_K$ is the $K \times K$ identity
    matrix and $\boldsymbol{\beta}_i \in \RR^p$, $\tilde{\sigma}_i^2 > 0$ are the unknown parameters.
    Furthermore, suppose the parameter of interest is $\mu_i = \mathbf{c}^\top \boldsymbol{\beta}_i$,
    where $\mathbf{c} \in \RR^p$ is a known contrast. Let $\widehat{\boldsymbol{\beta}}_i \coloneqq \p{\mathbf{X}^\top \mathbf{X}}^{-1} \mathbf{X}^\top \mathbf{Y}_i$
    be the ordinary least squares coefficient estimate, and let
    $$Z_i \coloneqq \mathbf{c}^\top \widehat{\mathbf{\beta}}_i, \;\;\;\;\, S_i^2 \coloneqq  \mathbf{c}^{\top} \p{\mathbf{X}^\top \mathbf{X}}^{-1} \mathbf{c} \cdot \Norm{\mathbf{X}\widehat{\boldsymbol{\beta}}_i - \mathbf{Y}_i}_2^2\bigg/(K-p).$$
    Then $(Z_i, S_i^2)$ satisfy~\eqref{eq:full_sampling} with $\mu_i =  \mathbf{c}^\top \boldsymbol{\beta}_i$, $\sigma_i^2 = \tilde{\sigma}_i^2 \mathbf{c}^{\top} \p{\mathbf{X}^\top \mathbf{X}}^{-1} \mathbf{c}$,
    and $\nu = K-p$. 
\end{prop}
By Proposition~\ref{rema:contrasts}, the proposed methods can be applied as soon as the normal linear model in~\eqref{eq:linear_model} is appropriate. This is a strong modeling assumption, which however, is widely considered to be a reasonable approximation for several genomic technologies with quantitative output.\footnote{In all three applications that we consider in this paper (Sections~\ref{subsec:methylation}, \ref{subsec:ibrutinib}, and~\ref{subsec:proteomics}), the authors of the original scientific publications posited the sampling model in~\eqref{eq:full_sampling}, in addition to the parametric assumption in equation~\eqref{eq:limma_shrinkage_assumption} on the distribution of $\sigma_i^2$ (which we relax). } 
A necessary prerequisite is the availability of a
strong scientific and statistical understanding of the given genomic technology.
We quote \citet{Smyth2022}: ``The challenge with each new technology is to preprocess and normalize the data to make the linear model assumptions as good as possible. [...] Figuring out the best preprocessing for each new technology is one of the key concerns of statistical bioinformatics research and is the topic of many published journal articles.''

\subsection{A demonstration with high-throughput methylation data}
\label{subsec:methylation}

\begin{figure}
    \centering
    \includegraphics[width=1.1\linewidth]{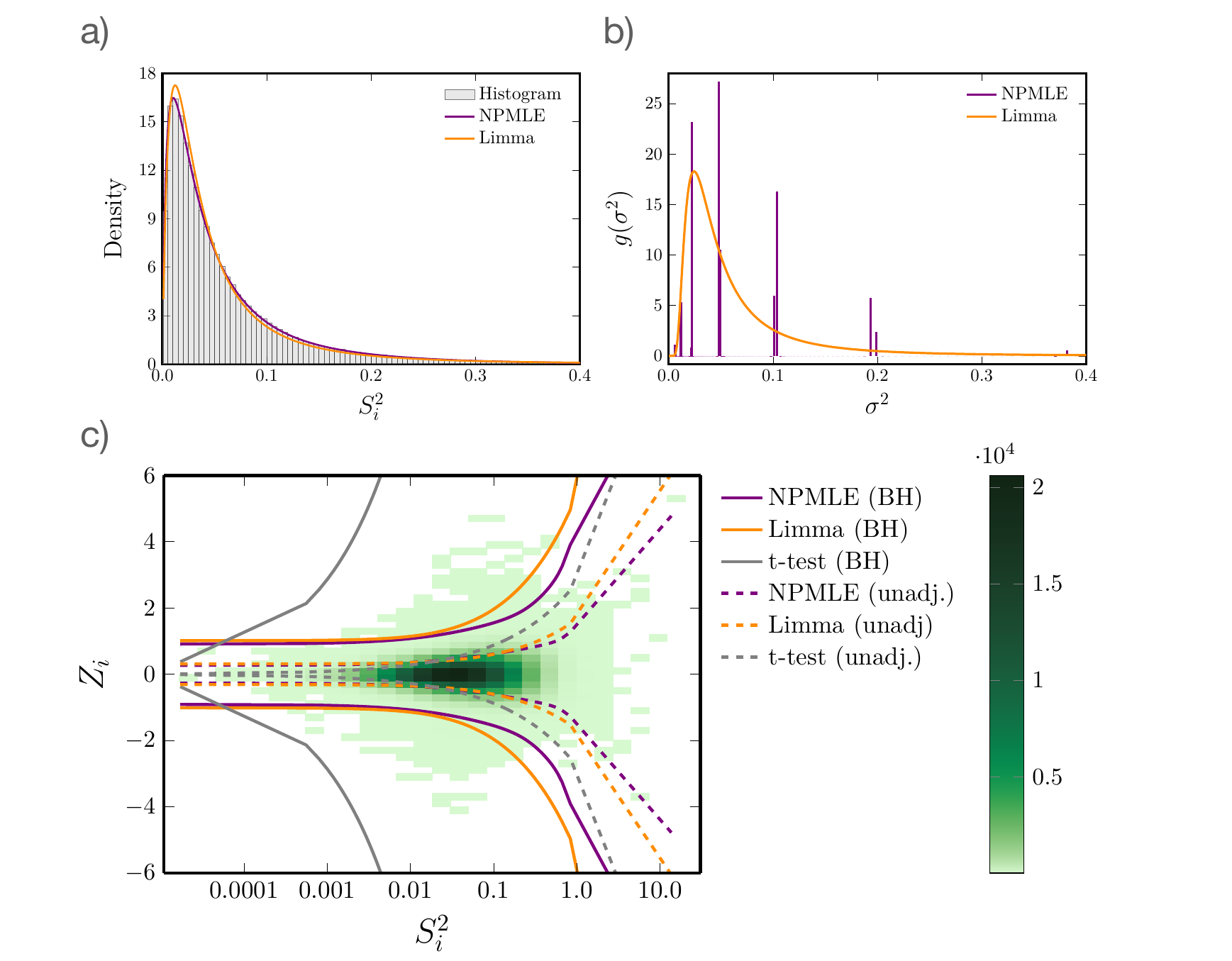}
    \caption{\textbf{Empirical partially  Bayes analysis of methylation data of \citet{zhang2013genomewide}}: {\small 
    Each unit $i$ corresponds to a distinct methylation probe ($n=439,918$) and its data is summarized as in~\eqref{eq:full_sampling}
    with $\nu = 4$ degrees of freedom.
    \textbf{a)} Histogram of the sample variances $S_i^2$, as well as the
    implied marginal densities of $S_i^2$ (marginalizing over~\eqref{eq:random_sigma}) for two choices
    of \smash{$G=\hG$}: the nonparametric maximum likelihood estimate 
    (NPMLE), and limma's
    parametric prior. \textbf{b)} The two estimated priors \smash{$\hG$}. For limma we plot
    the Lebesgue density of the prior. The NPMLE is a discrete measure;
    we have rescaled the mass of 
    the support points by a multiplicative constant for visualization purposes. 
    \textbf{c)} Two dimensional histogram of $(S_i^2, Z_i)$ pairs with the $x$-axis logarithmically scaled. 
    The number of hypotheses in each histogram bin is indicated by 
    color (color legend on the right) wherein the denser regions have darker green color. 
    The range of histogram counts is wide and any non-empty bin 
    (even with a single hypothesis) is colored in green. 
    The plot also shows the rejection thresholds of $3 \times 2$ different methods; in each case any hypothesis with 
    $Z_i$ above the upper line or below the lower line will be rejected. The different methods correspond to three constructions
    of p-values (empirical partially Bayes p-values with NPMLE, with limma, and standard t-test p-values) combined with two decision rules: 
    reject any hypothesis with p-value $\leq 0.05$ (dashed lines, unadjusted for multiple testing) and 
    reject any hypothesis with Benjamini-Hochberg (BH) adjusted p-value $\leq 0.05$ (solid lines).}}
    \label{fig:methylation}
\end{figure}

Before delving into a more formal treatment, we provide a high-level overview
of the methodological developments in this paper through 
a concrete data analysis. \citet{zhang2013genomewide} compared two types of immune cells,
naïve T-cells and antigen-activated naïve T-cells, in terms of DNA methylation, a heritable epigenetic mark.  For each of $n=439,918$ methylation probes, the outcome variable is the $\log_2$ ratio between the methylated intensity to the unmethylated intensity measured on the Illumina HumanMethylation450 microarray. The dataset, encompassing $10$ samples from $3$ human subjects and $4$ immune cell types, is analyzed using a linear model (as described in Proposition~\ref{rema:contrasts}) that adjusts for subject and cell type variations, yielding $\nu = 10 - 1 - (3-1) - (4-1) = 4$ degrees of freedom. The contrast between naïve T-cells and antigen-activated naïve T-cells is of primary interest.
We refer to \citet{zhang2013genomewide} and~\citet{Maksimovic2017f1000}
for further details on the dataset; what is important for our analysis is the following:
after data preprocessing and summarization,
the statistical task amounts to a large scale multiple testing problem based on data that may be assumed to be drawn from~\eqref{eq:full_sampling} with $\nu=4$ degrees 
of freedom.\footnote{Indeed, the analysis in \citet{zhang2013genomewide} and~\citet{Maksimovic2017f1000} uses limma,
and posits the sampling model in~\eqref{eq:full_sampling}, as well as the parametric assumption in~\eqref{eq:limma_shrinkage_assumption}
on the distribution of $\sigma_i^2$.}

Fig.~\ref{fig:methylation}a) shows the histogram of the sample variances $S_i^2$, as well as the
implied marginal densities of $S_i^2$ (marginalizing over~\eqref{eq:random_sigma}) for two choices
of \smash{$G=\hG$}: the nonparametric maximum likelihood estimate 
(NPMLE; more of which in Section~\ref{subsec:npmle}), and limma's
parametric prior ($\hat{\nu}_0 \approx 3.96$ and $\hat{s}_0^2 \approx 0.037$ in~\eqref{eq:limma_shrinkage_assumption}).
The NPMLE leads to a slightly better fit at the mode of the histogram. 
Fig.~\ref{fig:methylation}b) shows the two estimated priors \smash{$\hG$}; the NPMLE is discrete. 
While the two priors are qualitatively quite different,  the implied marginal densities
for $S_i^2$ in Fig.~\ref{fig:methylation}a) are quite similar.

Finally, Fig.~\ref{fig:methylation}c) demonstrates the heart of the matter:
it shows a 2d histogram of all the $(Z_i, S_i^2)$ pairs as well as the rejection thresholds 
of different methods.
There are three dashed lines; one corresponding to standard t-test p-values, and the other
two to the NPMLE and parametric (limma) partially Bayes p-values. For each value of $S_i^2$, 
the dashed rejection threshold demarcates the value of $z$ which would be such that the corresponding
p-value would be equal to $0.05$; in each case, hypotheses with 
$Z_i > 0$  (resp. $Z_i <0$) above (resp. below) the corresponding rejection threshold are rejected. The solid lines demarcate the rejection thresholds 
of the different methods after adjusting for multiple testing with the Benjamini-Hochberg procedure
with a nominal false discovery rate control of $\alpha = 0.05$. 

With the Benjamini-Hochberg correction, the NPMLE-based method leads to $969$ discoveries, compared to 627 with the
parametric limma prior, and to only $1$ discovery with the standard t-test. 
In this case the empirical partially Bayes methods lead to a substantial increase
in discoveries.
Fig.~\ref{fig:methylation}c) provides some insight into the mechanism. The (solid) rejection threshold
corresponding to the t-test is a lot more liberal for small values of $S_i^2$ (that is: closer to $0$) compared
to the empirical partially Bayes methods, but quickly becomes a lot 
more conservative as $S_i^2$ increases.\footnote{The reader may at this point wonder why the difference is less striking
for the dashed thresholds that correspond to p-values $\leq 0.05$ compared to the solid thresholds 
(Benjamini-Hochberg rejection thresholds). 
One reason is that the Benjamini-Hochberg procedure adapts to the signal in the p-values, and so,
the implied rejection threshold in terms of p-values varies for the different methods.
To be concrete, the data-driven p-value thresholds (for the different methods)
determined by the Benjamini-Hochberg procedure are equal to: \smash{$P_i^{\text{NPMLE}} \lessapprox 10^{-4}$}, \smash{$P_i^{\text{limma}} \lessapprox 0.7\cdot 10^{-4}$}, and \smash{$P_i^{\text{ttest}} \lessapprox 10^{-7}$}.}
The single discovery made with the t-test has sample variance $S_i^2=4\cdot10^{-5}$, which is
the second smallest sample variance in the whole dataset, and this hypothesis is \emph{not} rejected
by the empirical partially Bayes methods. The empirical partially Bayes approach shrinks 
this outlying sample variance upwards, so that this hypothesis is not rejected. On the other hand,
for the bulk of the units, the empirical Bayes approach leads to more accurate estimate of the sample variance,
which in turn leads to a more liberal rejection threshold and thus to more power overall. In this problem the NPMLE makes more discoveries than limma. In general this will not be true,
and the main benefit of the NPMLE is that it has theoretical type-I error guarantees also when
the parametric model for $G$ in~\eqref{eq:limma_shrinkage_assumption} does not hold.

We now discuss the conditional properties of procedures that we alluded to earlier. 
Fig.~\ref{fig:methylation_histogram} shows the histogram of the t-test p-values, as well as the NPMLE-based empirical partially Bayes p-values.
Three histograms are shown for each method: a histogram of all units, a histogram of units with sample variance $S_i^2$
in the bottom 10\% percentile, as well as a histogram of units with sample variance $S_i^2$ in the top 10\% percentile.
We observe that for large sample variances, the t-test leads to a histogram that is superuniform with almost no p-values $\leq 0.05$.
The t-test p-values are enriched for smaller p-values when only retaining units with small sample variance. In contrast,
the empirical partially Bayes p-values appear to have better conditional properties. We explain these empirical findings theoretically
in Proposition~\ref{prop:ttest_failure} which shows that the type-I error of the t-test may be arbitrarily inflated when the sample variance
is small, while Proposition~\ref{prop:asymptotically_uniform} demonstrates that the NPMLE p-values are asymptotically uniform 
conditional on $S_i^2$.

\begin{figure}
    \centering 
    \includegraphics[width=0.99\linewidth]{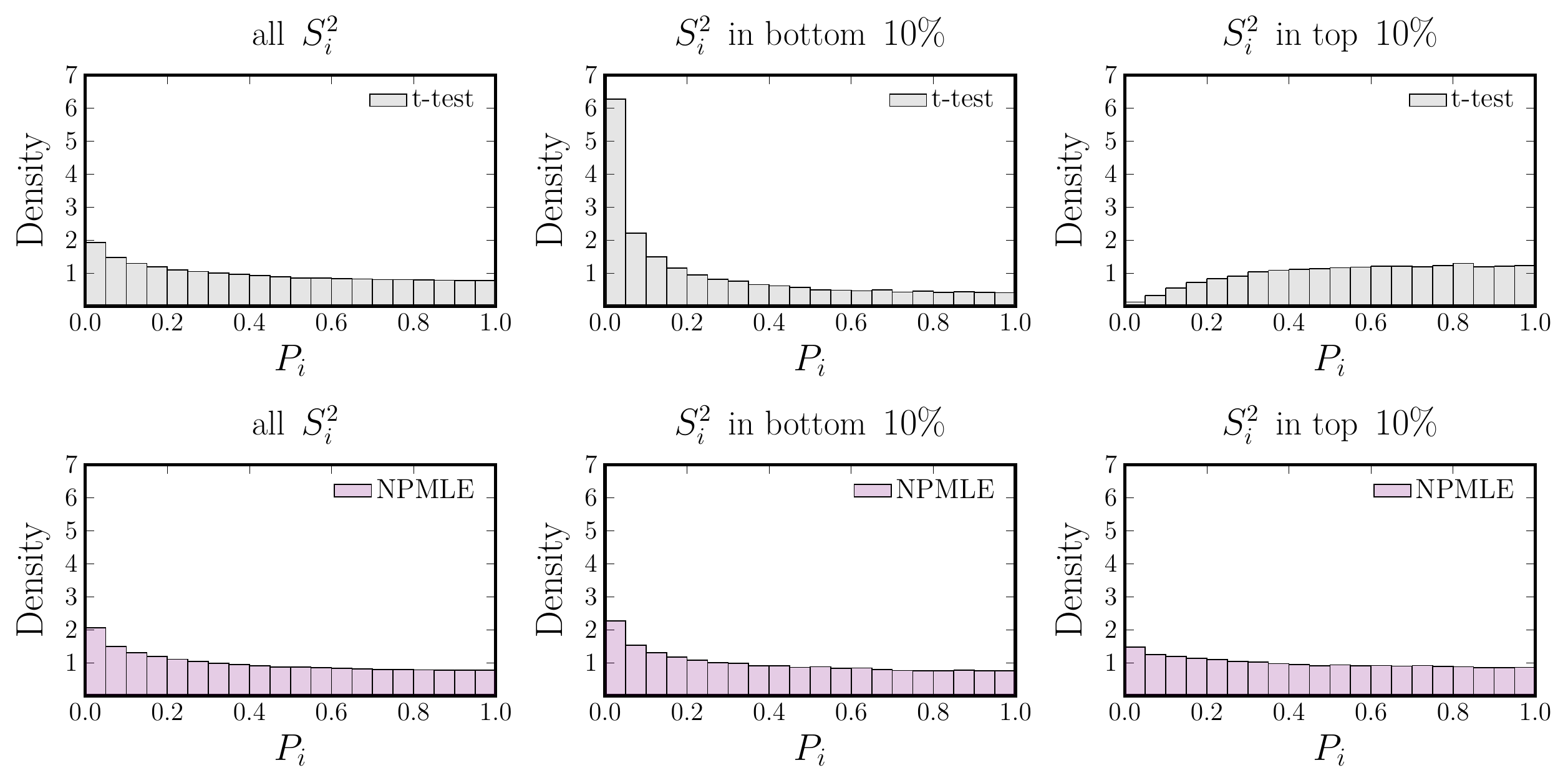}
    \caption{\textbf{P-value histograms in the analysis of methylation data of \citet{zhang2013genomewide}}: {\small The 
    first row shows t-test p-values, while the second rows shows the empirical partially Bayes p-values based on the NPMLE. The columns correspond to three different subsets of units: the first column shows all units, the second column only
    retains units with sample variance $S_i^2$ in the bottom 10\% percentile, and the third column retains units with sample variance $S_i^2$ in the top 10\% percentile.} } 
    \label{fig:methylation_histogram}
\end{figure}

\subsection{Related work}
\label{subsec:related_work}
Our work builds on a rich literature on multiple testing that was motivated by the microarray 
technology~\citep{tusher2001significance, efron2001empiricala, baldi2001bayesian, newton2001differential, cui2003statistical, wright2003random, menezes2004microarray, newton2004detecting, delmar2005mixture, cui2005improved, ploner2006multidimensional, witten2007comparison}.
As mentioned earlier, the specific model we elaborate upon was originally developed by~\citet{Lonnstedt166664, lonnstedt2002replicated} and 
turned into a widely used ``empirical partially Bayes'' testing methodology by~\citet{smyth2004linear}.
\citet{ignatiadis2022evalues} consider testing in the limma model of~\citep{smyth2004linear} with e-values instead of p-values. 

Closest to our paper, \citet{lu2016variance} consider a nonparametric generalization of the limma method 
based on a class of unimodal priors. We consider an even broader nonparametric assumption and lay 
out the theoretical foundations and framework that justify our proposal. 

When testing for $\mu_i=0$, uncertainty in $\sigma_i$ is often ignored by assuming that $\sigma_i$ is
known exactly~\citep{sun2012multiple, stephens2017false, gu2018oracle, fu2022heteroscedasticityadjusted}. This can be a good approximation in some cases,
 e.g., when $\sigma_i^2$ is estimated by $S_i^2$ and the degrees of freedom $\nu$ are large. 
In this paper we are mostly interested in the situation wherein the degrees of freedom $\nu$ are small;
such settings benefit the most from the empirical partially Bayes methodology that we outline
in this paper.

Going beyond the goal of multiple testing, several authors have studied 
the empirical Bayes problem with sample variances $S_i^2$ 
or joint draws of $(Z_i,\, S_i^2)$ as in~\eqref{eq:full_sampling}. 
\citet{yu2018adaptive} develop confidence intervals for $\mu_i$ with finite-sample frequentist coverage that
borrow strength 
from $(Z_j, S_j^2)_{j \neq i}$. A further strand of the literature considers shrinkage estimation of e.g., $\mu_i$, 
$\sigma_i^2$, or $1/\sigma_i^2$~\citep*{robbins1982estimating, muralidharan2012high, gu2017empirical, gu2017unobserved, banerjee2021nonparametric, kwon2022fmodellingbased}.
When $\nu=2$, the mixture model we consider for $S_i^2$ is a mixture of exponential distributions;
 see e.g.,~\citet{jewell1982mixtures} and \citet{polyanskiy2020selfregularizing} for some results for exponential mixtures.

Finally, the empirical partially Bayes approach was introduced in the setting
wherein there is a single primary parameter of interest and many nuisance parameters~\citep{lindsay1985using, zhao2019powerful};
a prior is posited (and estimated by empirical Bayes) only for the nuisance parameters but not the primary parameter.
Here we use the same terminology and study the empirical partially Bayes approach when we are facing a simultaneous inference
task with many parameters of primary interest, as well as many nuisance parameters.

\section{Oracle partially Bayes hypothesis testing}
\label{sec:conditional_pvalues}

\epigraph{``It would be possible to develop this investigation in much more detail, but in the absence of a specific application this will not be done''}{--- \textup{Sir David Roxbee Cox}, \citeyear{cox1975note}}

\noindent We first consider the setting in which the random variance assumption~\eqref{eq:random_sigma} holds
for a known prior $G$. To be concrete, we assume that we observe $(Z_i, S_i^2)$ drawn from the following model,
\begin{equation}
\label{eq:full_hierarchical}
\sigma_i^2 \simiid G,\;\;\;\; (Z_i, S_i^2) \mid \sigma_i^2 \, \simindep \, \mathcal{N}(\mu_i, \sigma_i^2) \otimes \frac{\sigma_i^2}{\nu} \chi^2_{\nu} \,\,\, \text{ for } i=1,\dotsc,n,
\end{equation}
where $G$ is known and $\mu_i$ is unknown (and with $\nu \geq 2$). Later,
we will drop the assumptions that $G$ is known (Section~\ref{sec:empirical-bayes}) and that
$\sigma_i^2$ are random (Section~\ref{sec:compound}). We study the model in~\eqref{eq:full_hierarchical} to answer
the question: \emph{how} would an oracle analyst with knowledge of the data generating mechanism in~\eqref{eq:full_hierarchical}
test the hypothesis $H_i: \mu_i=0$? Our nonparametric empirical Bayes approach will
then seek to mimic that oracle analyst. In the rest of this section
 we drop the subscript $i$ since we study only a single draw from~\eqref{eq:full_hierarchical}.

We start with the following characterization of $(Z, S^2)$ in~\eqref{eq:full_hierarchical}, 
see Supplement C.1 for a proof of the result.
\begin{prop}
\label{prop:minimal_sufficient}
Suppose that $G$ in~\eqref{eq:full_hierarchical} is not degenerate, that is, the support
of $G$ consists of 2 points or more and $G((0,\infty))=1$. Then $(Z, S^2)$ is the minimal sufficient statistic\footnote{However, $(Z,S^2)$ is \emph{not} complete for $\mu$.}
for $\mu$ in~\eqref{eq:full_hierarchical}. Furthermore, $S^2$ is ancillary for $\mu$, that is,
its distribution does not depend on $\mu$.
\end{prop}
The implication of the proposition is the following: when conducting inference for $\mu$, we need to use
both $Z$ and $S^2$ even though $S^2$ does not contain any information about $\mu$. The conditionality principle
is in full force~\citep[Chapter 2.3]{cox1974theoretical} and any inference for $\mu$ should be based on the conditional
distribution of $Z$ given 
$S^2$.\footnote{Suppose $G$ is degenerate, i.e., $G(\cb{\bar{\sigma}})=1$ for some $\bar{\sigma} >0$.
Then the conclusion of Proposition~\ref{prop:minimal_sufficient} does not hold
and $Z$ (rather than $(Z,S^2)$) is minimal sufficient for $\mu$. In that case $Z$ 
is independent of $S^2$, and so inference based on the conditional
distribution of $Z$ given $S^2$ is identical to inference based on $Z$.
}
Hence, an oracle analyst that knows~\eqref{eq:full_hierarchical} may use the conditional p-value $P = P_G(Z, S^2)$ 
defined in~\eqref{eq:conditional_pvalue} to test $H: \mu=0$.

Now, what if $G$ has the parametric form in~\eqref{eq:limma_shrinkage_assumption}? It turns out that then the limma p-value \smash{$P^{\text{limma}}_{\nu_0, s_0^2}(z, s^2)$} is identical
to the conditional p-value in~\eqref{eq:conditional_pvalue}. This is a consequence of the following result proved in Supplement C.2.
\begin{prop}
\label{prop:limma_as_conditional}
Let $(Z,S^2)$ be drawn from~\eqref{eq:full_hierarchical} with $G$ equal to
the distribution of $\sigma^2$ in~\eqref{eq:limma_shrinkage_assumption}, 
i.e., the inverse scaled chi-square distribution with parameters $\nu_0, s_0^2$. Then:
\begin{equation}
    \label{eq:limma_conditional}
Z \mid (S^2=s^2) \;\; \stackrel{\mathcal{D}}{=}\;\; \p{\mu \,+\, \tilde{s} T_{\nu_0+\nu}},
\end{equation}
where $T_{\nu_0+\nu}$ is a random variable following the t-distribution with $\nu_0+\nu$ 
degrees of freedom and $\tilde{s}$ is as in~\eqref{eq:moderated_t}.
It follows that for such $G$, $P_G(z, s^2) = P^{\text{limma}}_{\nu_0, s_0^2}(z, s^2)$.
\end{prop}
The above result gives an alternative interpretation of the limma p-value showing that it is exactly the same 
as the conditional p-value, which is the right thing to look at given the conditionality principle.

The claim of Proposition~\ref{prop:limma_as_conditional} is known, e.g., it is mentioned in~\citet{mccarthy2009testing} and~\citet{lu2016variance}, yet its 
implications are far-reaching and of historical importance: Sir David Cox~\citeyearpar{cox1975note} considered the model in~\eqref{eq:full_hierarchical} with $G$ 
as an inverse scaled chi-square distribution
and concluded that estimation/inference for $\mu$ 
should be based on~\eqref{eq:limma_conditional}, that is, the conditional distribution
of $Z$ given $S^2$. He called the resulting inference
``\emph{partially Bayes};'' the nuisance parameter is randomly distributed according
to a prior, while the primary parameter of interest is fixed. Moreover, 
Sir David Cox anticipated that the above model would be fruitful in an empirical Bayes 
analysis~\citep[``If the parameters of the prior distribution are estimated from a large amount of data'']{cox1975note}.
We do not believe that this contribution of Sir David Cox is well-known.\footnote{\citet{cox1975note} has less than 40 citations to date. 
But this is not for lack of importance: this work is included in the collection of selected statistical papers of Sir David Cox~\citep{cox2005selected} and it appears to have
been one source of inspiration in the development of predictive p-values, e.g., by~\citet{meng1994posterior}.}
We find it a source of great inspiration that yet another idea of Sir David Cox 
had ground-breaking impact 
in applied statistical work, in this case through the idea's wide adoption
in high-throughput biology via the limma software~\citep{smyth2004linear}.

The following proposition states three properties of the conditional p-values $P_G(Z, S^2)$; see Supplement C.3 for a proof.
\begin{prop}
    \label{prop:monotonicity}
    Let $(Z,S^2)$ be drawn from model~\eqref{eq:full_hierarchical} with $\mu=0$. Then, $P=P_G(Z, S^2)$ has the following properties.
    \begin{enumerate}[wide, noitemsep, label=\alph*.]
        \item Conditional uniformity: $\PP[G]{P \leq t  \mid S^2 } = t \text{ for all } t \in [0,\,1]$ almost surely.
        \item Unconditional uniformity:  $\PP[G]{P \leq t} = t \text{ for all } t \in [0,\,1]$.
        \item Monotonicity: $P_G(z, s^2)$ is non-increasing in $\abs{z}$ for any fixed value of $s^2>0$, 
        and non-decreasing in $s^2 >0$ for any fixed value of $\abs{z}$.
    \end{enumerate}
\end{prop}
The first property follows from the definition in~\eqref{eq:conditional_pvalue}
and the probability integral transform applied conditional on $S^2$. 
Such conditional validity is desirable (cf. \citet[Chapter 2.3]{cox1974theoretical} and \citet[Chapter 10]{lehmann2005testing}), 
and may not be taken for granted. For example,
this property does not hold for the standard t-test as we will prove formally in Proposition~\ref{prop:ttest_failure} below;
we also refer the reader to the data analysis of Section~\ref{subsec:methylation} (in particular, Figure~\ref{fig:methylation_histogram})
that showcases empirical ramifications of the standard t-test's lack of conditional type-I error control.
The second property (unconditional uniformity) follows from the conditional uniformity
result (and iterated expectation). Finally, the monotonicity posits
a ``natural'' requirement for a testing procedure based on $(Z, S^2)$: 
the larger the value of $\abs{Z}$ is, the stronger the evidence, and for 
a fixed value of $\abs{Z}$ the evidence is stronger when $S^2$ is smaller.
Empirical Bayes procedure often satisfy such natural monotonicity requirements, 
e.g.,~\citet{vanhouwelingen1976monotone, koenker2014convex}, but not always
~\citep{gu2022invidious}.

We contrast the results from Proposition~\ref{prop:monotonicity}
with the following result for the standard t-test, see Supplement C.4 for a proof:

\begin{prop}
\label{prop:ttest_failure}
Consider the t-test p-value $P=P^{\text{ttest}}(Z, S^2)$ with $P^{\text{ttest}}(z, s^2) = 2\bar{F}_{t,\nu}(\abs{z}/s)$ and $\bar{F}_{t,\nu}$ is the
survival function of the t-distribution with $\nu$-degrees of freedom. Let $(Z,S^2)$ be drawn from model~\eqref{eq:full_hierarchical} with $\mu=0$. 
Then, $P$ is unconditionally uniform and satisfies the monotonicity property 
in Proposition~\ref{prop:monotonicity}. However, if $\EE[G]{\sigma^{-\nu}} \in (0,\infty)$, then:
$$ \lim_{\delta \searrow 0} \PP[G]{P^{\text{ttest}}(Z, S^2) \leq t \cond S^2 \leq \delta} = 1 \text{ for all }t \in (0,\,1].$$
\end{prop}
Hence the standard t-test p-value does not satisfy the conditional uniformity property.\footnote{By contrast, according to Proposition~\ref{prop:monotonicity},
the oracle partially Bayes p-value $P$ satisfies $\mathbb P_G[P  \leq t \cond S^2 \leq \delta] = t$ for all $t\in[0,1]$ and $\delta>0$.}
If we observe $S^2$ that is unusually small, then Proposition~\ref{prop:ttest_failure}
implies that the conditional level of the t-test will be inflated and 
the resulting inference untrustworthy. 
The concern about 
poor conditional behavior, however, is not actionable
with a single hypothesis and without prior information for $\sigma^2$: without
prior knowledge about plausible values of $\sigma^2$, it is impossible to know whether $S^2$
is unusually small. The partially Bayes model~\eqref{eq:full_hierarchical} provides knowledge
of plausible values of $\sigma^2$, since $\sigma^2 \sim G$. Such knowledge
is also provided in the empirical Bayes 
setting of the following section wherein we have access to measurements $S_1^2,\dotsc,S_n^2$ of many related variances $\sigma_1^2,\dotsc,\sigma_n^2$.
Hence the setting we consider provides a further 
demonstration of the value of large scale data~\citep{efron2010largescale} with new opportunities
that were not available previously.

\section{Empirical partially Bayes multiple hypothesis testing}
\label{sec:empirical-bayes}
\epigraph{``This formula for the mean error depends only on the observed distribution''}{--- \textup{Sir Frank Dyson}, \citeyear{dyson1926method}}

We now turn to the practically relevant case wherein we have $n$ independent draws of 
$(Z_i, S_i^2)$ from~\eqref{eq:full_hierarchical} with $G$ and $\mu_i$ unknown. Our goal
is to mimic the oracle analyst who would test $H_i: \mu_i =0$ with
the conditional p-value defined in~\eqref{eq:conditional_pvalue}.
The strategy we pursue for this task is a nonparametric generalization
of the limma approach. First, we estimate \smash{$\widehat{G}$} with 
the nonparametric maximum likelihood estimator (NPMLE) of~\citet{robbins1950generalization} and~\citet{kiefer1956consistency} based on $(S_1^2,\dotsc,S_n^2)$, which are distributed as
\begin{equation}
    \label{eq:EB}
\sigma_i^2 \simiid G,\;\;\,  S_i^2 \mid \sigma_i^2 \; \simindep \frac{\sigma_i^2}{\nu} \chi^2_{\nu} \,\;\text{ for }\; i=1,\dotsc,n,
\end{equation}
and then we use the plug-in p-values $P_{\widehat{G}}(Z_i, S_i^2)$ in place 
of $P_G(Z_i, S_i^2)$. We call this approach \emph{empirical partially Bayes hypothesis testing};
a similar nonparametric extension of limma was also considered by~\citet{lu2016variance} 
based on a nonparametric class of unimodal priors.

\subsection{Nonparametric maximum likelihood estimation for \texorpdfstring{$\chi^2$}{chi-squared} data}
\label{subsec:npmle}

Our proposed estimator of $G$ based on $(S_1^2,\dotsc,S_n^2)$ 
is defined as a solution to the following optimization problem:
\begin{equation}
\label{eq:npmle_opt}
\hG \in \argmax\cb{ \sum_{i=1}^n \log\p{ f_G(S_i^2)}\,:\,G \text{ distribution supported on } (0,\, + \infty)}.
\end{equation}
Above, we use the notation $f_G(\cdot)$ to denote the marginal
density of $S_i^2$ when $\sigma_i^2 \sim G$ (for any distribution $G$ supported on $(0,\, \infty)$),
and $S_i^2 \mid \sigma_i^2$ follows~\eqref{eq:EB}.\footnote{\label{footnote:fmodeling}
Herein we follow the notation in~\citet{efron2014two} who uses the letters ``$F$,'' respectively ``$G$,'' to denote the marginal, respectively, prior distribution in empirical Bayes problems. We use lowercase $f$ to denote the marginal density of $S_i^2$, and make the dependence of $f$ on the prior distribution $G$ explicit by writing $f=f_G$.}
More formally, for $s^2 > 0$:
\begin{equation}
    \label{eq:marginal_density}
    f_G(s^2) \equiv f_G(s^2;\; \nu) \coloneqq \int_0^{\infty} p(s^2 \mid \sigma^2, \nu) \, dG(\sigma^2),
\end{equation}
where  $p(s^2 \mid \sigma^2, \nu)$ is the (scaled) $\chi^2_{\nu}$-density which is equal to,
\begin{equation}
\label{eq:chisq_likelihood}
p(s^2 \mid \sigma^2) \equiv p(s^2 \mid \sigma^2, \nu) \coloneqq \frac{\nu^{\nu/2}}{\p{\sigma^2}^{\nu/2} 2^{\nu/2}\Gamma(\nu/2)}  \p{s^2}^{\nu/2-1} \exp\p{-\frac{\nu s^2}{2\sigma^2}}.
\end{equation}
As suggested in~\eqref{eq:marginal_density} and~\eqref{eq:chisq_likelihood},
we typically suppress the dependence on $\nu \geq 2$. The optimization problem in~\eqref{eq:npmle_opt} seeks 
the maximizer of the log marginal likelihood over all possible choices of the prior $G$.

We note a few standard properties of the
optimization problem in~\eqref{eq:npmle_opt}~\citep{jewell1982mixtures, lindsay1983geometrya}
that follow from the arguments in~\citet{lindsay1993uniqueness} upon reparametrizing the likelihood in~\eqref{eq:chisq_likelihood} through the natural parameter $1/\sigma^2$.\footnote{We also provide a self-contained proof in Supplement D.1.}
\begin{prop}[Properties of NPMLE optimization]
\label{prop:npmle_opt}
Suppose that $S_i^2 \in (0, \infty)$ for all $i$.
Then there exists a unique maximizer \smash{$\hG$} of \eqref{eq:npmle_opt}. Further, \smash{$\hG$}
is a discrete probability measure with at most $n$ points of support,\footnote{More precisely: the number of support points is bounded by the number of unique values in $\cb{S_1,\dotsc,S_n}$.} all lying in the interval
 $[\min_i\cb{S_i^2}, \max_i \cb{S_i^2}]$.
\end{prop}

As a consequence of the proposition, the NPMLE may be written as \smash{$\hG = \sum_{j=1}^{\hat{K}}\hat{g}_j \delta_{\hat{\sigma}_j^2}$},
where \smash{$\hat{K} \leq n$} is the number of components, \smash{$\delta_{\sigma^2}$} is a Dirac point mass at \smash{$\sigma^2 >0$}, and \smash{$\hat{g}_j \geq 0$} are such that \smash{$\sum_{j=1}^{\hat{K}}\hat{g}_j =1$}. All of \smash{$\hat{K}$}, \smash{$\hat{\sigma}_j^2$}, and $\hat{g}_j$ 
are determined in a fully data-driven way.
It is worthwhile to compare this result to the method of~\citet{delmar2005mixture} who assumed a clustering
structure for the variances and estimated a prior of the form \smash{$\widehat{G} = \sum_{j=1}^K \hat{g}_j \delta_{\hat{\sigma}_j^2}$}
for a pre-determined value of $K$. For the NPMLE, the number of components is automatically learned, and furthermore, as we will see below,
the NPMLE performs well for the downstream task of empirical partially Bayes testing even if no such clustering structure is present.

\begin{rema}[Computation]
\label{rema:computation}
To solve~\eqref{eq:npmle_opt} we use the interior point convex programming solver 
MOSEK~\citep{aps2020mosek}
after discretization of~\eqref{eq:npmle_opt} as proposed by~\citet{koenker2014convex}. To be concrete, we discretize
as follows: we let $a$, resp. $b$, be the $1\%$ quantile,\footnote{ \label{footnote:lower_bound}
The choice of the $1\%$ quantile for the lower end of the support of the estimated prior
is motivated by the following considerations: 
on one hand, it is a conservative choice that protects us from spuriously small p-values. The choice
implies that the conditional p-values we compute in~\eqref{eq:conditional_pvalue} will be at least
as large as \smash{$2(1-\Phi(\abs{Z}/S_{0.01}))$}, where \smash{$S_{0.01}^2$} is the $1\%$ quantile of the sample variances.
Second, \citet{lonnstedt2002replicated} consider filtering out hypotheses with variances in the bottom $1\%$
as a simple solution to avoid an excessive number of rejections of hypotheses with small sample variance.  We also assess the impact of the choice of the lower end by conducting a simulation study and sensitivity analysis in Supplement H.1.
} resp. maximum, of $\cb{S_1^2,\dotsc,S_n^2}$ and 
then we consider a logarithmically equispaced grid\footnote{\label{footnote:spacing}
The logarithmic spacing is motivated by the following heuristic considerations: first, as seen in, e.g., Fig.~\ref{fig:methylation}c), the sample variances in common applications can range over multiple orders of magnitude. Second, taking logarithms in~\eqref{eq:EB}, we find that \smash{$\log(\nu S_i^2) = \log(\sigma_i^2) + \xi_i$} with $\exp(\xi_i) \sim \chi^2_{\nu}$ independent of $\sigma_i^2$. That is, we arrive at a location model with parameters \smash{$\log(\sigma_i^2)$}. Since uniformly spaced grids perform well empirically for location models~\citep{koenker2014convex}, we choose a uniformly spaced grid for \smash{$\log(\sigma_i^2)$}, which implies a logarithmically spaced grid for $\sigma_i^2$.
} between $a$ and $b$ with $300$ points. We maximize
the marginal likelihood over all distributions supported on the aforementioned grid.
Our theoretical results below do not take into account the discretization; we refer to~\citet{dicker2016highdimensional}
and~\citet{soloff2021multivariate} for analyses of the NPMLE that also consider the discretization error.
\end{rema}

For all our theoretical results, we assume that $\nu\geq 2$ and the following.\footnote{\label{footnote:no_nu_1}We do not consider the case $\nu = 1$ throughout because it is extreme (with a Cauchy reference distribution for the t-statistic) and not as common in applications. The case $\nu=1$ would also require a different analytical approach from our treatment for $\nu \geq 2$ because the conditional density in~\eqref{eq:chisq_likelihood} is unbounded as $s^2 \searrow 0$  for $\nu=1$.} 
\begin{assu}[Compact support bounded away from $0$ and $\infty$]
\label{assu:compact_support_and_nu_EB}
Let $0< \ubar{L} \leq L \leq U \leq \bar{U} < \infty$. We assume that $\sigma_i^2 \in [L,\, U]$ for all $i$ (almost surely)
and that we compute the NPMLE \smash{$\hG$} in~\eqref{eq:npmle_opt}
under the additional constraint $G([\ubar{L}, \bar{U}])=1$. 
\end{assu}
As we explained in Footnote~\ref{footnote:lower_bound}, we view the lower bound ($L$,$\ubar{L}$) as a substantive requirement: 
it is a conservative  choice that protects us from spuriously small p-values for units with sample variance $S_i^2 \approx 0$.
The upper bound ($U$, $\bar{U}$) is a technical assumption that streamlines our proofs; it could be relaxed 
to a light tail assumption with a more careful analysis as in~\citet{zhang2009generalized}.

As is common in the literature~\citep{ghosal2001entropies, zhang2009generalized},
we study the statistical properties of \smash{$\hG$} in terms of
the Hellinger distance between \smash{$f_G$} and \smash{$f_{\hG}$}. The squared Hellinger distance
between two densities $f$ and $h$ on $(0,\, \infty)$ is defined as follows:
\begin{equation}
    \label{eq:hellinger_dist}
\Dhel^2(f, h) = \frac{1}{2} \int_0^{\infty} \p{ \sqrt{f(t)} - \sqrt{h(t)}}^2 dt.
\end{equation}

In Supplement D.2 we prove the next theorem using standard entropy arguments following~\cite{ghosal2001entropies} and~\citet{zhang2009generalized}.
\begin{theo}[Convergence in Hellinger distance]
    \label{theo:hellinger_dist_convergence}
Under Assumption~\ref{assu:compact_support_and_nu_EB}, 
there exist constants $C=C(\nu, \ubar{L}, \bar{U}) >0$ and $c=c(\nu, \ubar{L}, \bar{U}) >0 $ such that:
$$\PP[G]{\Dhel(f_{\hG},\, f_G) \geq  C\frac{\log n}{\sqrt{n}}} \leq  \exp\p{- c (\log n)^2} \text{ for all } n \in \mathbb N_{\geq 1}.$$
\end{theo}
According to Theorem~\ref{theo:hellinger_dist_convergence}, estimation of $f_G$ in Hellinger
distance is a relatively easy statistical task that may be achieved at the 
parametric rate $1/\sqrt{n}$ up to a logarithmic factor.

With Theorem~\ref{theo:hellinger_dist_convergence} in hand and identifiability
results of~\citet{teicher1961identifiability},
we also prove the following in Supplement D.4.
\begin{coro}
\label{coro:weak_convergence}
Under Assumption~\ref{assu:compact_support_and_nu_EB}, it holds with probability $1$ that:
$$ \hG \cd G \; \text{ as } \; n \to \infty.$$
Here $\cd$ denotes weak convergence.
\end{coro}
We do not pursue 
the rate of the above convergence for two reasons: first, 
as in related deconvolution problems, e.g., with additive
Gaussian noise, minimax convergence rates can be very slow, that is,
polynomial in $1/\log n$~\citep{zhang1990fourier, fan1991optimal}. 
\citet[paragraph after Remark 5]{belomestny2020nonparametric} suggest that such slow
rates also hold in the deconvolution problem with \emph{multiplicative} noise
that we are studying.
Second, in what follows 
we establish that plug-in estimation of the partially Bayes p-values inherits the fast rates in
Theorem~\ref{theo:hellinger_dist_convergence} and
is not impacted by the slow rate in the convergence of \smash{$\hG$} to $G$.

\subsection{An Eddington/Tweedie-type formula for conditional p-values}
\label{subsec:tweedie}

As mentioned above, in empirical Bayes problems it is common that the prior $G$
can only be estimated at very slow rates, while the marginal density $f_G$ can be estimated
at faster rates. The statistical difficulty of an empirical Bayes analysis, that is, the difficulty
in imitating the oracle Bayesian, can vary substantially across possible models and statistical tasks, see
e.g.,~\citet{efron2014two, ignatiadis2022confidence}. In the worst case, the empirical Bayes problem
may be as difficult as the task of estimating $G$.

However, for certain empirical Bayes problems, the action of the oracle Bayesian may be described
as a benign functional of $f_G(\cdot)$; in that case the task of the empirical Bayesian is  
easier. A concrete example of such a situation occurs when the empirical Bayesian seeks
mean squared error optimal estimation of the natural parameter in an 
exponential family. For example,
the distribution of the sample variances $S_i^2$ in~\eqref{eq:EB} forms an exponential family with 
natural parameter $1/\sigma_i^2$. When $\sigma_i^2 \sim G$ as in~\eqref{eq:random_sigma},
the optimal estimator of $1/\sigma_i^2$ under squared error loss,
is  the posterior mean $\EE[G]{1/\sigma_i^2 \mid S_i^2}$ which may 
be expressed in terms of the marginal density $f_G$ of $S_i^2$ as follows:
\begin{equation}
\label{eq:tweedie}
\EE[G]{\frac{1}{\sigma_i^2} \cond S_i^2} = \frac{\nu-2}{\nu}\frac{1}{S_i^2} - \frac{2}{\nu} \frac{\partial f_G(S_i^2)}{\partial (S_i^2)} \frac{1}{f_G(S_i^2)}.
\end{equation}
The above formula is called Tweedie's formula~\citep{efron2011tweedie} and extends to other exponential families,
e.g., in the homoskedastic Gaussian case an analogous formula was credited by~\citet{dyson1926method} to Sir Arthur Eddington.

There are further settings where analogous ``$F$-formulas'' hold, see e.g., \citet{robbins1982estimating, cressie1982useful, du2022empirical, kwon2022fmodellingbased}.
Our next result demonstrates that such a formula also exists for the conditional p-values in~\eqref{eq:conditional_pvalue}; 
see Supplement E.1 for a proof.
\tcbset{fgp1/.style={no shadow,shrink tight,extrude by=1mm,colframe=blue,
  boxrule=0pt,frame style={opacity=0.3},interior style={opacity=0}}}
\tcbset{fg/.style={no shadow,shrink tight,extrude by=1mm,colframe=orange,
  boxrule=0pt,frame style={opacity=0.3},interior style={opacity=0}}}
\begin{prop}
\label{prop:pvalue}
The conditional p-value in~\eqref{eq:conditional_pvalue} may be expressed as
$$P_G(z, s^2) = C(\nu)\frac{\p{s^2}^{\nu/2-1}}{\tcbhighmath[fg]{f_G(s^2;\nu)}}\int_0^{\infty} \ind\p{t^2 \geq \frac{\nu s^2 + z^2}{\nu+1}} \frac{ (t^2)^{-(\nu-1)/2}}{\sqrt{ (\nu+1)t^2 - \nu s^2}}\; \tcbhighmath[fgp1]{f_G(t^2; \nu+1)} \,d(t^2),$$
where $C(\nu) \coloneqq \cb{(1+1/\nu)^{-\nu/2}\Gamma((\nu+1)/2)}\big/\cb{\sqrt{\pi}(\nu+1)^{-1/2}\Gamma(\nu/2)}$.
\end{prop}
In words, the p-value function $P_G(\cdot)$ only depends on the unknown $G$
through $\tcbhighmath[fg]{f_G(\cdot;\nu)}$ and $\tcbhighmath[fgp1]{f_G(\cdot;\nu+1)}$. 
$f_G(\cdot;\nu)$ is easily handled, since we have access to direct measurements from it.
As shown in Theorem~\ref{theo:hellinger_dist_convergence}, $f_G(\cdot;\nu)$ may
be estimated at a rate that is parametric (up to logarithmic factors) with the NPMLE under nonparametric specification of $G$.

The term $f_G(\cdot;\nu+1)$ is slightly unusual since it corresponds
to the density of sample variances with one additional degree of freedom 
compared to our actual observations ($\nu+1$ vs. $\nu$). Hence we do not have
access to direct measurements from $f_G(\cdot;\nu+1)$. 

Nevertheless, the following
lemma establishes that if we have \smash{$\hG$} such that 
$f_G(\cdot;\nu) \approx f_{\hG}(\cdot;\nu)$, then also
$f_G(\cdot;\nu+1) \approx f_{\hG}(\cdot;\nu+1)$. The proof in Supplement E.2
builds on a truncation argument in the transformation domain of the Mellin transform~\citep{epstein1948applications, butzer1999selfcontained, brennermiguel2021spectral}.
\begin{lemm}
    \label{lemm:closeness_of_marginals_plus_1}
Consider two distributions $G,G'$ as in~\eqref{eq:random_sigma} with $\EEInline[G]{\sigma^{-1}}, \EEInline[G']{\sigma^{-1}} < M,$
for $M < \infty$. Then there exists a positive constant $C=C(\nu, M)$ that only depends on $\nu \geq 2, M$ such that:
$$\Norm{f_G(\cdot;\nu+1) -  f_{G'}(\cdot;\nu+1)}^2_{L^2} \,\leq \, C  \psi\p{\Norm{f_G(\cdot;\nu) -  f_{G'}(\cdot;\nu)}^2_{L^2}},$$
where $\Norm{f}_{L^2}^2 := \int_{(0,\infty)}f^2(t)\,dt$ and $\psi(u) := \p{1 + \abs{\log(u)}}u$ for $u>0$.
\end{lemm}

We next prove that the empirical Bayes p-values $P_{\hG}(Z_i, S_i^2)$ converge to the oracle p-values $P_G(Z_i, S_i^2)$
at almost parametric rates. 
At a high-level, the proof in Supplement E.3 proceeds as follows: let $\mathfrak{F}$ be the mapping $f_G \mapsto P_G$ (which is well-defined by identifiability of $G$). Then the integral representation of $P_G$ in Proposition~\ref{prop:pvalue}, along with Lemma~\ref{lemm:closeness_of_marginals_plus_1}, can be used to show a continuity property such that $\Dhel(f_{\hG},\, f_G)  \approx 0$ implies that $\mathfrak{F}(f_{\hG}) \approx \mathfrak{F}(f_{G})$ (in a suitable sense). We conclude by applying the convergence result for $\Dhel(f_{\hG},\, f_G)$ in Theorem~\ref{theo:hellinger_dist_convergence}. We pattern our argument following~\citet{jiang2009general}.\footnote{\label{footnote:jiangzhang}\citet{jiang2009general} consider an empirical Bayes estimand that does not have an integral representation (such as the one in Proposition~\ref{prop:pvalue}) and instead directly depends on the derivative of the marginal density, very similar to the expression in~\eqref{eq:tweedie}. The ``continuity'' argument of~\citet{jiang2009general} is a technical feat that proceeds via induction, bounding lower-order derivatives using higher-order derivatives.
}

\begin{theo}
\label{theo:pvalue_quality}
Under Assumption~\ref{assu:compact_support_and_nu_EB}, 
for any $\zeta \in (1/2,\, 1)$, there exists a constant $C = C(\nu, \ubar{L}, \bar{U}, \zeta)$ that depends only 
on $\nu, \ubar{L}, \bar{U}, \zeta$ such that:
$$ \max_{1 \leq i \leq n} \EE[G]{\abs{P_G(Z_i, S_i^2)\land \zeta - P_{\hG}(Z_i, S_i^2)\land \zeta}} \leq C \frac{(\log n)^{5/2}}{\sqrt{n}} \text{ for all } n \in \mathbb N_{\geq 2}.$$
Above we write $a \land b \coloneqq \min \cb{a, b}$ for two numbers $a, b \in \RR$. 
\end{theo}
The truncation of the p-values to $\zeta \in (1/2,\, 1)$ is helpful 
in controlling the tail behavior of $P_G(z, s^2)$ when $z \approx 0$. 
In the next section we will study 
the Benjamini-Hochberg procedure; the above guarantee suffices as long as $\alpha < \zeta$, where
$\alpha$ is the nominal level at which we seek to control the false discovery rate.

\begin{exam}[Identical variances]
    \label{exam:same_variance}
Suppose that all the variances $\sigma_i^2$ are almost surely equal to $\bar{\sigma}^2 > 0$, that is, \smash{$G = \delta_{\bar{\sigma}^2}$}, where $\delta_u$ denotes the point 
mass at $u$.  In this case, the oracle procedure for testing $H_i: \mu_i=0$ reduces to the usual z-test with p-value \smash{$2\p{1-\Phi(\abs{Z_i}/\bar{\sigma})}$}.
Theorem~\ref{theo:pvalue_quality} then implies that the empirical partially Bayes p-values using the NPMLE in~\eqref{eq:npmle_opt}
automatically adapt and are close to the oracle z-test p-values.

The assumption that $G$ is a point mass at $\bar{\sigma}$, may appear to be
contrived.  In the statistical literature, related assumptions are often considered
~\citep{robbins1956empirical, ignatiadis2021empirical, soloff2021multivariate} to
demonstrate the adaptation of nonparametric empirical Bayes procedures. 
In the present setting, the assumption is substantive: the assumption
that all $\sigma_i$ are equal appears regularly in scientific publications~\citep{kerr2000analysisa,houde2011utility, weis2019comment}. 
The motivation is to stabilize the variance estimate
 when the individual sample variances are very noisy, 
(e.g.,model~\eqref{eq:EB} with small degrees of freedom $\nu$).
Our nonparametric strategy 
 automatically adapts to the equal variance
assumption, if it is indeed warranted, and can thus match the power of the oracle
z-test with known variance. The gains in that case, e.g.,
compared to the t-test, can be substantial, 
see the classical treatment of~\citet{walsh1949information} and the simulation study in Section~\ref{sec:simulations}.
\end{exam}

We conclude with an important consequence of Theorem~\ref{theo:pvalue_quality}: 
the empirical partially Bayes p-values $P_{\hG}(Z_i, S_i^2)$ 
asymptotically satisfy the uniformity properties of the oracle p-values $P_G(Z_i, S_i^2)$ that we laid
out in Proposition~\ref{prop:monotonicity}.
\begin{prop}
    \label{prop:asymptotically_uniform}
    Suppose Assumption~\ref{assu:compact_support_and_nu_EB} holds.
    Let $\Hnull \coloneqq \cb{i \in \cb{1,\dotsc,n}\,:\, \mu_i=0}$ be the subset of null hypotheses and let $\zeta \in (1/2,\, 1)$, $C = C(\nu, \ubar{L}, \bar{U}, \zeta)$ be as in Theorem~\ref{theo:pvalue_quality}.
    Then, the following hold for $P_i = P_{\hG}(Z_i, S_i^2)$:
    \begin{enumerate}[wide, noitemsep]
        \item Asymptotic conditional uniformity: $$\max_{i \in \Hnull}\cb{\EE[G]{\sup_{t \in [0,\zeta]} \abs{\PP[G]{P_i \leq t  \mid S_1^2,\dotsc,S_n^2 } - t}}}\, \leq\,  C \frac{(\log n)^{5/2}}{\sqrt{n}} \text{ for all } n \in \mathbb N_{\geq 2}.$$
        \item Asymptotic uniformity:  
         $$\max_{i \in \Hnull}\sup_{t \in [0,\zeta]} \abs{\PP[G]{P_i \leq t} - t}\, \leq\,  C \frac{(\log n)^{5/2}}{\sqrt{n}} \text{ for all } n \in \mathbb N_{\geq 2}.$$
    \end{enumerate}
\end{prop}

\begin{rema}[Avoiding double-use of the data]
    \label{rema:double_use}
The empirical partially Bayes approach  proceeds in two steps: first, construct the tail functions \smash{$z \mapsto P_{\widehat{G}}(z, S_i^2)$} using all the sample variances $(S_1^2,\dotsc, S_n^2)$, and then use the remaining randomness, that is, the randomness in $(Z_1,\dotsc,Z_n)$ conditional on $(S_1^2,\dotsc,S_n^2)$, for the testing task. 
The separation of roles of $(S_1^2,\dotsc,S_n^2)$ and $(Z_1,\dotsc,Z_n)$ is conceptually appealing as it prevents the double-use of data, a point discussed by \citet{bayarri2000values} in the context of conditional predictive p-values.
The separation of roles also facilitates the  proof of the fast convergence rate (\smash{$1/\sqrt{n}$} up to logarithmic factors) in Proposition~\ref{prop:asymptotically_uniform}. At a high-level, the argument proceeds as follows.
Let \smash{$\Delta_i : = \sup_{z} \abs{ P_{\hG}(z, S_i^2) - P_G(z, S_i^2)}$}. For any $t \in [0,1]$ and $i \in \Hnull$, it holds that
$$\PP[G]{ P_{\hG}(Z_i, S_i^2) \leq t \cond S_1^2,\dotsc,S_n^2} \leq \PP[G]{P_G(Z_i, S_i^2) \leq t + \Delta_i \cond S_1^2,\dotsc,S_n^2} = (t + \Delta_i)\land 1.$$
The last equality holds because $\Delta_i$ is a function of $S_1^2,\dotsc,S_n^2$ only (since \smash{$\widehat{G}$} is a function thereof) and because the distribution of the oracle p-value \smash{$P_G(Z_i, S_i^2)$}  is uniform conditional on \smash{$S_1^2,\dotsc,S_n^2$} by Proposition~\ref{prop:pvalue} and independence of the pairs \smash{$(Z_i, S_i^2)$}, \smash{$i=1,\dotsc,n$}. 
Analogously it follows that \smash{$\mathbb P_G[P_i \leq t \cond S_1^2,\dotsc,S_n^2] \geq \max\cb{t - \Delta_i,\, 0}$}, and so the conclusion of part 1 of the proposition can be verified by controlling $\EE[G]{\Delta_i}$. This is accomplished in Supplement E.4 (with an argument that uses a slightly different definition for $\Delta_i$).

\end{rema}

\subsection{Benjamini-Hochberg with empirical partially Bayes p-values}
\label{subsec:BH}

In this subsection we describe the last step 
of Algorithm~\ref{algo:np_limma} in more detail.
Suppose we are testing $n$ hypotheses $H_1, \dotsc, H_n$ based on p-values $P_1,\dotsc,P_n$.
Given $\alpha \in (0,1)$, the~\citet{benjamini1995controlling} (BH)
procedure produces a subset $\mathcal{D}$ of $\cb{1,\dotsc,n}$
representing the indices of rejected hypotheses 
through the following rule:
\begin{equation}
\label{eq:BH}
\text{Reject } H_k \, \text{ if }\, P_k \leq P_{(k^*)}, \text{ where } k^* \coloneqq \max\cb{\ell \in \cb{0,\dotsc,n} \,:\, P_{(\ell)} \leq \frac{\alpha \cdot \ell}{n}},
\end{equation}
and $P_{(1)} \leq P_{(2)} \leq \dotsc \leq P_{(n)}$ are the order statistics of $P_1,\dotsc,P_n$
and $P_{(0)}=0$. 

The most basic type-I error result for the BH procedure is as follows.
Let $V_n = \# \mathcal{D}\cap \Hnull$ be the total number of 
false discoveries of BH, $R_n = \#\mathcal{D}$ the total number of discoveries, and 
define the false discovery rate
as the expectation of the false discovery proportion,
that is, $\mathrm{FDR}_n := \EE{ \mathrm{FDP}_n}$, $\mathrm{FDP}_n := V_n / (R_n\lor 1)$, where we write $a \lor b = \max\cb{a, b}$ for $a, b \in \RR$. 
Then, if all p-values are independent and $P_i \sim U[0,\,1]$ for $i \in \Hnull$, 
the BH procedure satisfies $\textrm{FDR}_n = \alpha n_0/n$ where $n_0 = \# \Hnull$ is the number of null hypotheses.
In words: the BH procedure controls the false
discovery rate at the slightly conservative level $\alpha n_0/n$.

In our setting, we only have approximate p-values $P_i = P_{\hG}(Z_i, S_i^2)$,
which in general will not be exactly uniformly distributed for $i \in \Hnull$
and finite $n$. Our next results builds on Proposition~\ref{prop:asymptotically_uniform} 
and demonstrates that BH applied to $P_i = P_{\hG}(Z_i, S_i^2)$ (Algorithm~\ref{algo:np_limma})
asymptotically controls the false discovery rate at level $\alpha n_0/n$ as $n \to \infty$.\footnote{\label{footnote:storey} The conservative FDR control 
of the BH procedure has motivated the development of null proportion adaptive methods, e.g., Storey's procedure~\citep{storey2004strong}. 
At a high level, null proportion adaptive methods proceed by estimating $n_0/n$ as \smash{$\widehat{\pi_0}$} and then applying the BH procedure
at level \smash{$\alpha/\widehat{\pi_0}$}. Under assumptions and a setting similar to those in Theorem~\ref{theo:limma_bh_controls_the_FDR}, it can be proven that using Storey's procedure instead of BH in Algorithm~\ref{algo:np_limma} would still asymptotically control the false discovery rate at level $\alpha$.
}
The crux of our proof, presented in Supplement F.2, employs a leave-one-out strategy inspired by \citet{roquain2022false} among others~\citep{ferreira2006benjamini,li2019multiple,ignatiadis2021covariate}. Concretely, when controlling the contribution to the false discovery rate of $i\in \Hnull$, we condition on $Z_1,\dotsc,Z_{i-1},Z_{i+1},\dotsc,Z_n$ and $S_1^2,\dotsc,S_n^2$ (including $S_i^2$). Hence we only need to consider the randomness in $Z_i$ (conditional on everything else), similarly to the proof of Proposition~\ref{prop:asymptotically_uniform}.

\begin{theo}
    \label{theo:limma_bh_controls_the_FDR}
Suppose that Assumption~\ref{assu:compact_support_and_nu_EB} holds and that there exist sequences $\kappa_n, \eta_n > 0$ such that $\PP[G,\boldmu]{R_n < n \kappa_n} \leq \eta_n$, where $R_n$ is the number of discoveries of Algorithm~\ref{algo:np_limma} applied at level $\alpha \in (0,1)$. Then there exists a constant $C=C(\ubar{L}, \bar{U},\alpha)$ such that for all $n \in \mathbb N_{\geq 2}$:
$$ \mathrm{FDR}_n -\frac{n_0}{n}\alpha\; \leq \;  \EE[G,\boldmu]{ \p{\EE[G,\boldmu]{\mathrm{FDP}_n \cond S_1^2,\dotsc,S_n^2} - \frac{n_0}{n}\alpha}\lor 0}\; \leq \;  C \frac{(\log n)^{5/2}}{n^{1/2} \kappa_n } \,+\, \eta_n.$$
\end{theo}
The statement of the theorem is inspired by~\citet[Theorem 2]{li2019multiple}, whose FDR bound is also a function of $\kappa_n$ and $\eta_n$ and requires that the number of rejections grows asymptotically. Below, in Section~\ref{subsec:denseregime}, we show that in a commonly studied dense asymptotic regime, we may choose $\kappa_n = \Theta(1)$ and \smash{$\eta_n = \widetilde{O}(n^{-1/2})$}
in the statement of the theorem.\footnote{We write $\widetilde{O}(\psi_n)$ for a sequence $\psi_n$ as shorthand for $O(\psi_n \psi_n')$ where $\psi_n'$ is polylogarithmic in $n$.}
The theorem also permits for asymptotic $\mathrm{FDR}$ control under sparsity with $n_0/n \to 1$ as $n \to \infty$. As one example, if \smash{$\#\{i: \abs{\mu_i} \geq \sqrt{2.1 \, \bar{U}\log n}\}\;\geq\; n^{1-\gamma}$ for $\gamma \in (0,1/2)$}, then  \smash{$\mathrm{FDR}_n \leq (n_0/n)\alpha + O\p{n^{\gamma-1/2}(\log n)^{5/2}}$} by choosing $\kappa_n = \Theta(n^{-\gamma})$ and $\eta_n = O(n^{\gamma-1/2})$.\footnote{However, Theorem~\ref{theo:limma_bh_controls_the_FDR} does not permit for sparsity $n-n_0 \leq n^{1/2}$ and beyond  and so is weaker compared to the results in~\citet{roquain2022false}, who also do not require a lower bound on the number of rejections. The primary reason is that in this work, we are able to bound the additive error between data-driven and oracle p-values, while~\citet{roquain2022false} are able to also control the multiplicative error (in a suitable sense). We leave a more precise study of the multiplicative error of the empirical partially Bayes p-values to future work.}
Theorem~\ref{theo:limma_bh_controls_the_FDR} also establishes a guarantee on the expected false discovery proportion conditional on \smash{$S_1^2,\dotsc,S_n^2$}: if the rightmost side of the inequalities of Theorem~\ref{theo:limma_bh_controls_the_FDR} converges to $0$ as $n \to \infty$, then $\PP[  G,\boldmu]{\EE[G,\boldmu]{\mathrm{FDP}_n \cond S_1^2,\dotsc,S_n^2}\geq \varepsilon + n_0\alpha/n} \to 0$ as $n \to \infty$ for any $\varepsilon >0$.

\subsubsection{False discovery rate and power analysis in the dense regime} 
\label{subsec:denseregime}
In this section, we study Algorithm~\ref{algo:np_limma} in an asymptotic regime with dense signals. Our motivation is two-fold.
First, we instantiate the false discovery rate guarantee of Theorem~\ref{theo:limma_bh_controls_the_FDR} in a commonly studied asymptotic regime. Second, we derive the asymptotic power of Algorithm~\ref{algo:np_limma} and compare it to the power of the following oracle procedure: the BH procedure conducted by a statistician with access to the oracle partially Bayes p-values $P_G(Z_i, S_i^2)$ in~\eqref{eq:conditional_pvalue}.\footnote{\label{footnote:reference_oracle} This is a natural reference procedure: the conditionality principle in Section~\ref{sec:conditional_pvalues} led us to the empirical partially Bayes p-values (which are only known to the oracle) and the Benjamini-Hochberg procedure is the most commonly used multiple testing procedure that operates on p-values and provides strong control of the false discovery rate. We also refer the reader to~\citet[Section 7.4]{roquain2022false} for a convincing argument that BH is a natural reference procedure for multiple testing with oracle p-values.}

The asymptotic regime is stated below in terms of assumptions placed on the null proportion and the distribution of the oracle partially Bayes p-values $P_G(Z_i, S_i^2)$:
\begin{assu}[Dense signal asymptotics]
    \label{assu:dense_signal}
    We assume that as $n \to \infty$, $\frac{n_0}{n} \to \pi_0^{\infty}$ and $\frac{1}{n}\sum_{i=1}^n \PP[G,\mu_i ]{P_G(Z_i, S_i^2) \leq t} \to H^{\infty}(t) \text{ for all } t\in [0,1],
    $
    where $\pi_0^{\infty} \in (0,1)$ and $H^{\infty}(\cdot)$ is a continuous distribution function on $[0,1]$. We assume that $t^{\infty} := \sup\cb{t \in [0,1]\, :\, t/H^{\infty}(t) \leq \alpha}$ satisfies $t^{\infty} \in (0,\alpha)$ and that $t \mapsto t/H^{\infty}(t)$ is strictly increasing in a neighborhood of $t^{\infty}$. 
\end{assu}
Asymptotic regimes like the above are commonly studied in the multiple testing literature with closely related assumptions made in~\citet{genovese2002operating, storey2004strong, ferreira2006benjamini, du2014singleindex}. At a high-level, the assumption implies that the fraction of null hypotheses remains constant (and bounded away from $0$ and $1$), and that the alternative signals $\mu_i$ behave as if ``iid'' in the limit (cf. Section~\ref{sec:joint_hierarchical}) with sufficient signal strength. Although this asymptotic regime precludes sparse settings, it is conceptually appealing. Under this regime, (as shown by the aforementioned work) the oracle BH procedure is asymptotically equivalent to a single-threshold procedure with threshold $t^{\infty}$ that rejects $H_i$ if $P_G(Z_i, S_i^2) \leq t^{\infty}$. Furthermore, as noted by~\citet{stephens2017false}, this regime clarifies that the burden associated with multiple testing is not necessarily related to the number of tests. Under Assumption~\ref{assu:dense_signal}, as the number of tests increases, the true positives and false discoveries increase linearly, and the FDR remains approximately constant (cf. Proposition~\ref{prop:mimicking_oracle} below and the paragraph following it).

Our first result evaluates the false discovery rate $\mathrm{FDR}_n$ of the empirical partially Bayes procedure (Algorithm~\ref{algo:np_limma} for $\alpha \in (0,1)$). The proof is presented in Supplement F.3 and builds upon Theorem~\ref{theo:limma_bh_controls_the_FDR} with $\kappa_n = \Theta(1)$ and \smash{$\eta_n = \widetilde{O}(n^{-1/2})$}.
\begin{prop}
    \label{prop:dense_fdr_control}
    Under Assumptions~\ref{assu:compact_support_and_nu_EB} and~\ref{assu:dense_signal}, the following holds for any  $\xi > 5/2$: 
    $$
    \limsup_{n \to \infty} \cb{n^{1/2}(\log n)^{-\xi} \p{\mathrm{FDR}_n - \frac{n_0}{n}\alpha}} \leq 0.$$
\end{prop}
Hence the FDR is controlled up to an error that decays at a nearly-parametric rate. By comparison, the oracle procedure  has exact finite-sample control of the false discovery rate, that is, $\mathrm{FDR}^{\text{or}}_n = (n_0/n)\alpha$ for all $n$. 

We next turn to the important issue of power. We consider two commonly used criteria to assess the type-II error performance of a multiple testing procedure:
\begin{equation}
    \label{eq:power_criteria}
\mathrm{Pow}_n := \EE{\frac{R_n - V_n}{n-n_0}},\;\;\;\;\, \mathrm{FNDR}_n := \EE{ \frac{\sum_{i=1}^n \ind(\mu_i \neq 0,\; i \notin \mathcal{D}) }{(n-R_n)\lor 1}}.
\end{equation}
Above, $\mathrm{Pow}_n$ is the expected proportion of correctly rejected hypotheses among all non-null hypotheses (see e.g,~\citet{ferreira2006benjamini}) and $\mathrm{FNDR}_n$ is the false nondiscovery rate introduced by~\citet{genovese2002operating}. The above quantities refer to Algorithm~\ref{algo:np_limma}; we  also define $\mathrm{Pow}^{\text{or}}_n$, and $\mathrm{FNDR}^{\text{or}}_n$ for the analogous quantities for the oracle procedure. The proof of the following result is presented in Supplement F.4 and calls upon the results in~\citet[Section 3]{ferreira2006benjamini}.
\begin{prop}
\label{prop:mimicking_oracle}
Under Assumptions~\ref{assu:compact_support_and_nu_EB} and~\ref{assu:dense_signal}, both the oracle procedure and Algorithm~\ref{algo:np_limma} have asymptotically the same power in terms of the criteria in~\eqref{eq:power_criteria}, i.e.,  
$$
\limsup_{n \to \infty} \abs{\mathrm{Pow}_n - \mathrm{Pow}^{\mathrm{or}}_n} =0,\;\;\;\, \limsup_{n \to \infty}\abs{\mathrm{FNDR}_n -\mathrm{FNDR}^{\mathrm{or}}_n} = 0.
$$
Both criteria have asymptotic limits that may be expressed in terms of $t^{\infty}$, $\pi_0^{\infty}$, and $\alpha$:
$$
\lim_{n \to \infty} \mathrm{Pow}_n = \frac{t^{\infty}}{\alpha}\frac{1-\alpha \pi_0^{\infty}}{1-\pi_0^{\infty}},\;\;\;\, \lim_{n \to \infty}\p{1- \mathrm{FNDR}_n} = \pi_0^{\infty}\frac{\alpha (1-t^{\infty})}{\alpha - t^{\infty}}.
$$
\end{prop}
The result of the proposition can be intuitively understood as follows. Both Algorithm~\ref{algo:np_limma} and the oracle procedure behave asymptotically as the single-threshold procedure that rejects $H_i$ if $P_G(Z_i, S_i^2) \leq t^{\infty}$.
The number of total discoveries divided by $n$ concentrates around $H^{\infty}(t^{\infty})$, while the number of false discoveries divided by $n$ concentrates around $\pi_0^{\infty} t^{\infty}$. Furthermore, by Proposition~\ref{prop:dense_fdr_control}, the false discovery rate concentrates around $\pi_0^{\infty}\alpha$.  The asymptotic power and false nondiscovery rate can be derived from the former quantities.

\section{Multiple hypothesis testing in the compound setting}
\epigraph{``Let us use a mixed model, even if it might not be appropriate"}{--- \textup{Hans C. van Houwelingen}, \citeyear{vanhouwelingen2014role}}
\label{sec:compound}

\noindent 

In this section we study our procedure in the purely
frequentist compound setting wherein $\sigma_i^2$ are fixed---there may 
be strong conceptual reasons to prefer such a viewpoint, see, e.g.,~\citet[Supplementary Material]{ebrahimpoor2021inflated} and~\citet*[Appendix A]{battey2022partial}.
Formally, 
we observe independent (but not identically distributed) draws of $(Z_i, S_i^2)$ from~\eqref{eq:full_sampling}
with $\boldsymbol{\sigma} = \boldsymbol{\sigma}_n  = (\sigma_1,\dotsc,\sigma_n)$, $\boldsymbol{\mu} = \boldsymbol{\mu}_n = (\mu_1,\dotsc,\mu_n)$ deterministic 
but unknown. We nevertheless proceed as before and
pretend that the data generating process is given by~\eqref{eq:full_hierarchical}.
Our main result in this section is that the conclusion of Theorem~\ref{theo:limma_bh_controls_the_FDR} 
on the asymptotic false discovery rate control continues to hold in the 
compound setting. 

It is instructive to consider a single hypothesis, say $n=i=1$,
as in Section~\ref{sec:conditional_pvalues}. Suppose we posit 
model~\eqref{eq:full_hierarchical} for a pre-specified $G$
and compute $P_i = P_G(Z_i, S_i^2)$ as in~\eqref{eq:conditional_pvalue}. However,
we seek inferential validity under~\eqref{eq:full_sampling},
that is, for fixed $\sigma_1$. Then, we may interpret $P_i$ 
as a conditional predictive p-value, an inferential device introduced by~\citet*{bayarri2000values} for integrating out nuisance parameters with a Bayesian predictive distribution that does not entail double-use of the data (cf. Remark~\ref{rema:double_use}). Let $\sigma_1$ be fixed, $\mu_1=0$ and take $\nu \to \infty$. 
Then (under regularity assumptions on $G$) it may be verified (directly or with the results of~\citet{robins2000asymptotic}) 
that $P_1$ is asymptotically uniform. 

In contrast, our asymptotic results are attained as $n \to \infty$ and $\nu$ is \emph{fixed}.
For example, Theorem~\ref{manualtheorem:limma_bh_controls_the_FDR} below establishes asymptotic control
of the false discovery rate with Algorithm~\ref{algo:np_limma} as $n \to \infty$ and $\nu$ is fixed.
We find this conclusion striking. For example, if $\mu_1=0$, then 
$P_G(Z_1, S_1^2)$ will not be uniform for \emph{any} choice of $G$, unless
$G = \delta_{\sigma_1}$. This is in sharp contrast to e.g., the result in 
Proposition~\ref{prop:monotonicity}, wherein $P_G(Z_1, S_1^2)$ is uniform 
conditional on $S_1^2$. In other words, in the compound setting $P_{\hG}(Z_i, S_i^2)$
are not asymptotically ``valid'' p-values. Nevertheless,
the asymptotic false discovery rate guarantee holds (Theorem~\ref{manualtheorem:limma_bh_controls_the_FDR}). 
The reason is that while type-I error for individual hypotheses is not correctly controlled, 
the \emph{average} type-I error across hypotheses \emph{is controlled}---such notion of average significance
control was introduced by~\citet{armstrong2022false}. We make this argument rigorous below.

To explain the core ideas, we will first provide a summary
of compound decision theory in the Gaussian sequence model
in Section~\ref{subsec:recap_compound} and then explain
how compound decision theory applies to our setting in
Section~\ref{subsec:average_significance}.

\subsection{A modicum of compound decision theory}
\label{subsec:recap_compound}

We provide a bird's eye overview of compound decision theory 
in the Gaussian sequence model, and refer the
reader for more details to \citet{robbins1951asymptotically, copas1969compounda, zhang2003compound, jiang2009general, weinstein2021permutation}
and references therein. Consider the following two Gaussian location sequence models (where $M$ is a distribution on $\RR$):
\begin{subequations}
    \label{eq:gaussian_sequence}
\begin{align}
 &  \boldsymbol{\mu} = (\mu_1,\dotsc,\mu_n) \text{ fixed},\;\;  && Z_i \phantom{ \cond \mu_i } \simindep \nn(\mu_i, \,1) \text{ for }i=1,\dotsc,n \label{eq:compound_gaussian_sequence} \\
 & \mu_i \simiid M \text{ for }i=1,\dotsc,n,\;\; & &  Z_i \cond \mu_i \simindep \nn(\mu_i, \,1) \text{ for }i=1,\dotsc,n. \label{eq:eb_gaussian_sequence}
\end{align}
\end{subequations}
In the first model, the parameters of interest $\mu_i$ are fixed, while in the second 
model they are randomly distributed according to $M$. The crux of compound decision theory is that 
for certain statistical tasks, models~\eqref{eq:compound_gaussian_sequence} and 
\eqref{eq:eb_gaussian_sequence} behave very ``similarly'' under the choice of prior $M = M(\boldsymbol{\mu}) = \frac{1}{n}\sum_{i=1}^n \delta_{\mu_i}$,
that is, when $M$ is equal to the empirical measure of $\mu_1,\dotsc,\mu_n$.

Perhaps the most famous result in compound decision theory is the fundamental
theorem of compound decisions. Suppose we seek to estimate $\mu_1,\dotsc,\mu_n$ in either of the models
in~\eqref{eq:gaussian_sequence} with $\hat{\mu}_i = \eta(Z_i)$ where $\eta:\RR \to \RR$ is a fixed 
function. Then:
\begin{equation}
    \label{eq:fundamental_compound_decisions}
\EE[\boldsymbol{\mu}]{\frac{1}{n} \sum_{i=1}^n \p{\eta(Z_i) - \mu_i}^2} =\frac{1}{n}\sum_{i=1}^n \EE[\mu_i]{ \p{\eta(Z_i) - \mu_i}^2} = \EE[\mu \sim M(\boldsymbol{\mu})]{\big ( \eta(Z) - \mu \big )^2}.
\end{equation}
The critical observation is that the LHS is the mean squared error in estimating $\mu_i$ with $\hat{\mu}_i$
under the compound model~\eqref{eq:compound_gaussian_sequence} while the RHS is the mean squared 
under the hierarchical model~\eqref{eq:eb_gaussian_sequence}. Under~\eqref{eq:eb_gaussian_sequence}, $(\mu_i, Z_i)$
are $\text{iid}$ and so we drop the subscript $i$ in the RHS of~\eqref{eq:fundamental_compound_decisions}.
The equality in~\eqref{eq:fundamental_compound_decisions} is merely formal;
however a deep (and challenging to prove) consequence is that 
an empirical Bayes approach that estimates $M$ as \smash{$\widehat{M}$} based on $Z_1,\dotsc,Z_n$
and then uses \smash{$\hat{\eta}(z) = \EE[\widehat{M}]{\mu \mid Z_i=z}$} will perform well in 
mean squared error under both models in~\eqref{eq:gaussian_sequence}~\citep{jiang2009general}.

\subsection{Average significance controlling tests}
\label{subsec:average_significance}

We are ready to study our empirical partially Bayes multiple testing method in the compound setting.
In Section~\ref{sec:empirical-bayes} we treated
$\boldmu = (\mu_1,\dotsc,\mu_n)$ as fixed, but assumed that \smash{$\sigma_i^2 \sim G$}. Now  we also treat $\boldsigma = (\sigma_1,\dotsc,\sigma_n)$
as fixed. As in our recap of compound decision theory in Section~\ref{subsec:recap_compound}, 
the crux of our argument will be that the setting with $\boldsigma$ fixed 
is quantitatively similar to the hierarchical model~\eqref{eq:full_hierarchical} 
with a specific choice of prior $G$ for $\sigma_i^2$, namely the empirical distribution
of the $\sigma_i^2$:
\begin{equation}
\label{eq:empirical_sigma_distribution}
G \equiv G(\boldsigma) \coloneqq \frac{1}{n} \sum_{i=1}^n \delta_{\sigma_i^2}.
\end{equation}  
The compound decision argument suggests considering the following
compound ``p-value'' in lieu of the conditional p-values in~\eqref{eq:conditional_pvalue}:
\begin{equation} 
    \label{eq:compound_pvalue}
P_{\boldsigma}(z, s^2) \coloneqq P_{G(\boldsigma)}(z, s^2) = \frac{ \sum_{j=1}^n 2(1-\Phi(\abs{z}/\sigma_j)) p(s^2 \mid \sigma_j^2)}{\sum_{j=1}^n  p(s^2 \mid \sigma_j^2)}.
\end{equation}
Suppose we knew $\boldsigma$ and could use $P_{\boldsigma}(Z_i, S_i^2)$, 
$i=1,\dotsc,n,$ to test the null hypotheses $H_i: \mu_i=0$. Would we have 
any frequentist guarantees? The following theorem
proves that this is indeed the case and the $P_{\boldsigma}(Z_i, S_i^2)$
correspond to average significance controlling tests introduced by \citet{armstrong2022false}.
\begin{theo}
    \label{theo:avg_significance_controlling}
Consider $n$ independent draws from model~\eqref{eq:full_sampling} with $\boldmu$ and $\boldsigma$ fixed. Denote 
the null indices by $\Hnull \coloneqq \cb{1 \leq i \leq n\,:\, \mu_i=0}$. The
oracle compound p-values $P_{\boldsigma}(Z_i, S_i^2)$ in~\eqref{eq:compound_pvalue} satisfy 
the following guarantee:
\begin{equation}
    \label{eq:avg_significance_control}
\frac{1}{n}\sum_{i \in \Hnull} \PP[\boldsigma]{ P_{\boldsigma}(Z_i, S_i^2) \leq t } = \frac{1}{n}\sum_{i \in \Hnull} \PP[\sigma_i]{ P_{\boldsigma}(Z_i, S_i^2) \leq t }  \leq t \text{ for all } t \in [0,1].
\end{equation}
\end{theo}
\begin{proof}
Let $Z_i' = Z_i - \mu_i$ so that $Z_i'=Z_i$ for $i \in \mathcal{H}_0$. By definition, $P_{\boldsigma}(Z_i, S_i^2) = P_{G(\boldsigma)}(Z_i, S_i^2)$
with $G(\boldsigma)$ defined in~\eqref{eq:empirical_sigma_distribution}. Then:
$$
\begin{aligned}
  \frac{1}{n}\sum_{i \in \mathcal{H}_0} \PP[\sigma_i]{ P_{G(\boldsigma)}(Z_i,S_i^2) \leq t }  
\leq \frac{1}{n}\sum_{i =1}^n \PP[\sigma_i]{ P_{G(\boldsigma)}(Z_i',S_i^2) \leq t } = \PP[G(\boldsigma)]{P_{G(\boldsigma)}(Z',S^2) \leq t}.
\end{aligned}
$$
The last equality is only formal (cf. Section~\ref{subsec:recap_compound}).
Nevertheless, this formal equality enables us to conclude, since by Proposition~\ref{prop:monotonicity} it holds that 
$\PP[G(\boldsigma)]{P_{G(\boldsigma)}(Z',S^2) \leq t} = t$ for all $t \in [0,1]$.
\end{proof}
We comment on the average significance controlling property
in~\eqref{eq:avg_significance_control} as far as it pertains
to empirical partially Bayes multiple hypothesis testing and refer
the reader to \citet{armstrong2022false} and \citet*{armstrong2022robust} for elaboration in other settings.
Theorem~\ref{theo:avg_significance_controlling} has two weaknesses compared to 
Proposition~\ref{prop:monotonicity}b. First, control only holds on \emph{average} 
(rather than holding for each individual null hypothesis), and second 
the guarantee is inflated by $n/n_0$, where $n_0 = \#\Hnull$.\footnote{
In the setting of Proposition~\ref{prop:monotonicity}b: $\sum_{i \in \Hnull} \mathbb P_G[ P_G(Z_i, S_i^2) \leq t ] = t n_0$.}
The first weakness seems unavoidable in view of the nature of the compound 
setting. The second weakness is inherent to the 
empirical partially Bayes approach we have described (and which is predominant
in practice): by applying the NPMLE in~\eqref{eq:npmle_opt} to $S_i^2$, $i=1,\dotsc,n$, we estimate
the empirical distribution of all the $\sigma_i^2$, that is, $G(\boldsigma)$ defined in~\eqref{eq:empirical_sigma_distribution}.
However, the correct reference distribution is the empirical distribution
of the null $\sigma_i^2$, that is, $G((\sigma_i)_{i \in \Hnull})$.\footnote{ To estimate
\smash{$G((\sigma_i)_{i \in \Hnull})$}, we would also need to use information about the \smash{$\mu_i$}, 
which is contained in the $Z_i$, but not the \smash{$S_i^2$}. 
This hints at possible 
benefits of joint empirical Bayesian modeling of $\mu_i$ and \smash{$\sigma_i^2$} compared
to the empirical partially Bayes approach.}

For concreteness, suppose that $n=2$, $\nu=3$, and $i=1$ is a null hypothesis, while $i=2$
is an alternative hypothesis. Further suppose that $\sigma_1^2 = 1, \sigma_2^2 = 0.5$, that is, the alternative
has a smaller variance than the null. Then we can numerically compute for $t=0.01$,
that $\PP[\sigma_1]{ P_{(\sigma_1, \sigma_2)}(Z_1, S_1^2) \leq 0.01 } \approx 0.0199$,
i.e., the upper bound $n/n_0 \cdot t = 0.02$ in~\eqref{eq:avg_significance_control} is nearly tight.
The intuitive reason that type-I error control is lost by a factor of nearly $2$ is that the 
reference distribution for the null variances, that is $G = \delta_1$, is pushed 
stochastically downwards to $G = (\delta_{1} + \delta_{0.5})/2$ 
due to the presence of a smaller variance under the alternative.

To recap, even under adversarial arrangement of
the variances $\sigma_i^2$, the oracle compound p-values have the type-I error guarantee in~\eqref{eq:avg_significance_control}
which is only inflated by a factor $n/n_0$. In many applications there will only be few alternatives,
so that $n/n_0 \approx 1$. A further silver lining is that the average significance control 
guarantee in~\eqref{eq:avg_significance_control} suffices to control the false discovery rate asymptotically when applying
the Benjamini-Hochberg procedure. This result was shown by~\citet{armstrong2022false} and is consistent with an earlier observation by~\citet{efron2007size}, who noted that the control of the false discovery rate does not require identical null distributions, but rather that the average density of the nulls behaves like a uniform density.
Hence, below (Theorem~\ref{manualtheorem:limma_bh_controls_the_FDR}) we establish a false discovery control guarantee 
for Algorithm~\ref{algo:np_limma} in the compound setting that is analogous
to the result in Theorem~\ref{theo:limma_bh_controls_the_FDR}.

\subsection{Asymptotic results in the compound setting}

In this section we show that most of the asymptotic results of Section~\ref{sec:empirical-bayes} have an
analogue in the compound setting wherein $\boldsigma = (\sigma_1, \dotsc, \sigma_n)$ is fixed.  

To state our results, it will be helpful to introduce the averaged density $f_{\boldsigma}$ of the $S_i^2$:
\begin{equation}
\label{eq:average_marginal_density} 
f_{\boldsigma}(s^2) \equiv  f_{\boldsigma}(s^2;\; \nu) \coloneqq \frac{1}{n} \sum_{i=1}^n  p(s^2 \mid \sigma_i^2, \nu).
\end{equation} 
Note that, as in the general compound argument, $f_{\boldsymbol{\sigma}}(\cdot)$ is formally equal to
the marginal density $f_{G(\boldsigma)}(\cdot)$ in~\eqref{eq:marginal_density} with the prior
$G(\boldsigma)$ defined in~\eqref{eq:empirical_sigma_distribution}.
The probabilistic interpretation of these two objects however is distinct.

First we state the matching result to Theorem~\ref{theo:hellinger_dist_convergence}:

\begin{manualtheorem}{\getrefnumber{theo:hellinger_dist_convergence}*}
    \label{manualtheorem:hellinger_dist_convergence}
Under Assumption~\ref{assu:compact_support_and_nu_EB}, 
there exist constants $C=C(\nu, \ubar{L}, \bar{U}) >0$ and $c=c(\nu, \ubar{L}, \bar{U}) >0 $ such that:
$$\PP[\boldsigma]{\Dhel(f_{\hG},\, f_{\boldsigma}) \geq  C\log n/\sqrt{n}} \leq  \exp\p{- c (\log n)^2} \text{ for all } n \in \mathbb N_{\geq 1}.$$
\end{manualtheorem}
Compared to Theorem~\ref{theo:hellinger_dist_convergence}, Theorem~\ref{manualtheorem:hellinger_dist_convergence} only has the following differences:
first, the probability in the statement is taken with respect to the $n$ independent draws of $S_i^2$ in~\eqref{eq:EB} with $\boldsigma$ fixed
(rather than \smash{$\sigma_i^2 \simiid G$}), and second, the Hellinger distance bound pertains to the average
density $f_{\boldsigma}$ (rather than $f_G$). The argument for these two theorems are almost identical (cf. \citet{zhang2009generalized} 
in the Gaussian location setting) and so we also prove both in the Supplement concurrently.\footnote{While we do not explicitly
state it here, Corollary~\ref{coro:weak_convergence} also has an analogue in the compound setting: for any
bounded and continuous function $\psi: [0, \infty) \to \RR$ it holds that $\int \psi(u) (d\hG(u) - dG(\boldsigma)(u))$ converges to 
$0$ in probability. We refer the reader to~\citet[Theorem 4]{greenshtein2022generalized} for the rationale and rigorous
justification of such a result.}    

The following theorem, which parallels Theorem~\ref{theo:pvalue_quality}, also follows immediately:
\begin{manualtheorem}{\getrefnumber{theo:pvalue_quality}*}
    \label{manualtheorem:pvalue_quality}
Under Assumption~\ref{assu:compact_support_and_nu_EB}, for any $\zeta \in (1/2,\, 1)$, there exists a constant $C = C(\nu, \ubar{L}, \bar{U}, \zeta)$ that depends only 
on $\nu, \ubar{L}, \bar{U}, \zeta$ such that:
$$ \frac{1}{n}\sum_{i=1}^{n} \EE[\boldsigma]{\abs{P_{\boldsigma}(Z_i, S_i^2)\land \zeta - P_{\hG}(Z_i, S_i^2)\land \zeta}} \leq C \frac{(\log n)^{5/2}}{\sqrt{n}} \text{ for all } n \in \mathbb N_{\geq 2}.$$
\end{manualtheorem}
The difference between the conclusions of the above result with that of Theorem~\ref{theo:pvalue_quality} is that we only bound the expectation averaged over all $i=1,\dotsc,n$.

Building on Theorem~\ref{manualtheorem:pvalue_quality}, we can prove a variant of Proposition~\ref{prop:asymptotically_uniform} in the compound setting.
\begin{manualprop}{\getrefnumber{prop:asymptotically_uniform}*}
    \label{manualprop:asymptotically_uniform}
Fix $\zeta \in (1/2,\, 1)$. Under Assumption~\ref{assu:compact_support_and_nu_EB}, there exists a constant $C = C(\nu, \ubar{L}, \bar{U}, \zeta)$ such that for all $n \in \mathbb N_{\geq 2}$
$$\sup_{t \in [0,\zeta]} \cb{ \frac{1}{n}\sum_{i \in \Hnull}\PP[\boldsigma]{P_i \leq t} - t}\,\leq \, \EE[\boldsigma]{\sup_{t \in [0,\zeta]} \cb{ \frac{1}{n}\sum_{i \in \Hnull}\ind(P_i \leq t) - t}}\, \leq\,  C \frac{(\log n)^{5/4}}{n^{1/4}}.$$
\end{manualprop}
In the compound setting, the proof argument of Proposition~\ref{prop:asymptotically_uniform} no longer applies as the oracle null p-values are not uniform conditional on \smash{$S_1^2,\dotsc,S_n^2$}. Instead we apply a smoothing argument that is commonly used to translate bounds in the $1$-Wasserstein metric to bounds in the Kolmogorov-Smirnov metric (e.g., \citealp[Proposition 1.2.2]{ross2011fundamentalsa}). This smoothing argument leads to the slow down of the rate to $n^{-1/4}$ in the compound setting---we do not know if this is a proof artifact or a fundamental difference between the two settings. On the other hand, one appealing feature of Proposition~\ref{manualprop:asymptotically_uniform} compared to Proposition~\ref{prop:asymptotically_uniform} is that it also controls the supremum of the empirical process of the null p-values.

As already announced, the asymptotic false discovery rate control guarantee of Theorem~\ref{theo:limma_bh_controls_the_FDR}
also continues to hold in the compound setting.
Asymptotic false discovery rate control with p-values satisfying the average significance control property~\eqref{eq:avg_significance_control} 
was already shown by~\citet{armstrong2022false}. Here our argument also accounts for the fact that the average significance control property
is only approximately satisfied since the BH procedure uses $P_i = P_{\hG}(Z_i, S_i^2)$ rather than $P_{\boldsigma}(Z_i, S_i^2)$.
\begin{manualtheorem}{\getrefnumber{theo:limma_bh_controls_the_FDR}*}
        \label{manualtheorem:limma_bh_controls_the_FDR}
Suppose that Assumption~\ref{assu:compact_support_and_nu_EB} holds and that there exist sequences $\kappa_n, \eta_n > 0$ such that $\PP[\boldsigma,\boldmu]{R_n < \kappa_n n} \leq \eta_n$, where $R_n$ is the number of discoveries of Algorithm~\ref{algo:np_limma} applied at level $\alpha \in (0,1)$. Then there exists a constant $C=C(\ubar{L}, \bar{U},\alpha)$ such that
$$ \mathrm{FDR}_n \; = \; \EE[\boldsigma, \boldmu]{\mathrm{FDP}_n} \; \leq \; \alpha \,+\, C \frac{(\log n)^{5/4}}{n^{1/4}\kappa_n } \,+\, \eta_n \;\; \text{ for all }\, n \in \mathbb N_{\geq 2}.$$
\end{manualtheorem}
In a dense regime analogous to that of Section~\ref{subsec:denseregime}, we may choose $\kappa_n = \Theta(1)$ and \smash{$\eta_n = \widetilde{O}(n^{-1/2})$}. It is worth pointing out three differences in the statement of Theorem~\ref{manualtheorem:limma_bh_controls_the_FDR} compared to Theorem~\ref{theo:limma_bh_controls_the_FDR}.
First, in the latter, the FDR is controlled at level $(n_0/n)\alpha$, while here it is controlled only at level $\alpha$:
BH is conservative by a factor $n_0/n$~\citep[equation (2)]{benjamini1995controlling}, the oracle compound p-values may be anticonservative by a factor of at most $n/n_0$ 
(Theorem~\ref{theo:avg_significance_controlling}), and these two behaviors cancel out.\footnote{\label{footnote:storey_fails}Our conclusion in Footnote~\ref{footnote:storey} no longer holds and null-proportion adaptive methods such as Storey's procedure may not control the FDR in the compound setting (see
results of simulation study in Section~\ref{sec:simulations}).} Second, in Theorem~\ref{manualtheorem:limma_bh_controls_the_FDR} there is no guarantee on the expected false discovery proportion conditional on the sample variances. Finally, the rates are slowed down by the same factor as in Propositions~\ref{prop:asymptotically_uniform} versus~\ref{manualprop:asymptotically_uniform}.

\section{Results under a joint hierarchical model}
\label{sec:joint_hierarchical}
So far, we have treated the following two cases: $\boldmu=(\mu_1,\dotsc,\mu_n)$ is fixed and \smash{$\sigma_i^2 \simiid G$} (Section~\ref{sec:empirical-bayes})
and both $\boldmu$ and \smash{$\boldsigma^2 = (\sigma_1^2,\dotsc,\sigma_n^2)$} are fixed (Section~\ref{sec:compound}). It is instructive
to also study Algorithm~\ref{algo:np_limma} when \emph{both} $\mu_i$ and \smash{$\sigma_i^2$} are random; 
the following discussion will further illuminate the distinction between the formal results in Sections~\ref{sec:empirical-bayes}
and~\ref{sec:compound}. To be concrete, we suppose that there exists
an unknown distribution $\Pi$ supported on $\RR \times \RR_{>0}$ such that
\begin{equation}
    \label{eq:joint_EB}
(\mu_i,\,\sigma_i^2) \simiid \Pi\,\, \text{ for }\,\, i =1,\dotsc,n.
\end{equation}
The distribution $G(\cdot; \Pi)$ of $\sigma_i^2$ is given by marginalizing $\Pi$ over $\mu_i$, that is, 
$G$ is defined as follows for any Borel set $A \subset [0, \infty)$:
\begin{equation}
    \label{eq:marginalized_G}
     G(A; \Pi) \coloneqq \Pi(\RR \times A).
\end{equation}
We then have that $\sigma_i^2 \simiid G(\cdot; \Pi)$, i.e., the statement in~\eqref{eq:random_sigma} holds with prior $G(\cdot; \Pi)$.

We may directly apply the results of Section~\ref{sec:compound} by conditioning on the random vectors $\boldmu$ and $\boldsigma$,
in which case we are back to draws from model~\eqref{eq:full_sampling}. Below we also address the following question:
when are the results of Section~\ref{sec:empirical-bayes} applicable? For example, when can we argue that the empirical partially
Bayes p-values for null hypotheses are approximately uniform conditional on \smash{$(S_1^2,\dotsc,S_n^2)$} (Proposition~\ref{prop:asymptotically_uniform} in Section~\ref{sec:empirical-bayes})?
In the following we demonstrate that it suffices that $\ind(\mu_i=0)$ is 
independent of $\sigma_i^2$ under $\Pi$ in~\eqref{eq:joint_EB}, and we also contextualize this requirement.

As our initial attempt to applying the results of Section~\ref{sec:empirical-bayes}, we condition on $\boldmu$ (not on $\boldsigma$). 
Then,~\eqref{eq:full_hierarchical} specifies the following (where we make the conditioning
on $\boldmu$ explicit in the notation): there exists a distribution \smash{$\widetilde{G}$} such 
that\footnote{Here we are implicitly using the $\text{iid}$ structure in~\eqref{eq:joint_EB} which implies that
the conditional distribution of \smash{$\sigma_i^2$} given $\boldmu$ is equal to the conditional distribution of \smash{$\sigma_i^2$} given $\mu_i$.}
$$\sigma_i^2 \cond \mu_i \; \simiid \; \widetilde{G} \; \text{ for }\; i =1, \dotsc, n.$$
For the above to hold for almost all $\boldmu$ from~\eqref{eq:joint_EB}, we require that the marginals of $\Pi$ factorize, 
that is, we require that $\sigma_i^2$ is \emph{independent} of $\mu_i$ in which case $\widetilde{G} = G(\cdot; \Pi)$ defined in~\eqref{eq:marginalized_G}. 
Hence, in the above case, we may apply the results from Section~\ref{sec:empirical-bayes} by conditioning on $\boldmu$.
The assumption of independence is relatively common in the literature, see e.g., 
\citet[equation (2.2)]{ghosal2001entropies}, \citet{jiang2020general, zheng2021mixtwice}, and \citet{soloff2021multivariate}.
This assumption can be reasonable in some situations, namely when one believes that all primary parameters $\mu_i$
are of similar magnitude for all $i$~\citep[Section 8]{efron1973stein}. On the other hand, it is also well recognized
that the assumption may be violated in problems of 
practical interest~\citep{stephens2017false,  gu2017unobserved, weinstein2018grouplinear, chen2022gaussian}.

It turns out, however, that independence of $\mu_i$ and $\sigma_i^2$ is not necessary to e.g., prove
an analogous result to Proposition~\ref{prop:asymptotically_uniform} in Section~\ref{sec:empirical-bayes}:
in the empirical
partially Bayes approach we model and estimate the prior distribution of the nuisance parameters $\sigma_i^2$ based on $(S_1^2,\dotsc,S_n^2)$
only, i.e., a vector of statistics that are ancillary for $\mu_i$ (Proposition~\ref{prop:minimal_sufficient}).
Hence most of the technical results in Section~\ref{sec:empirical-bayes}, e.g., Theorem~\ref{theo:hellinger_dist_convergence} and Corollary~\ref{coro:weak_convergence}
only rely on the marginal distribution $G(\cdot; \Pi)$ of the $\sigma_i^2$; and not the joint distribution $\Pi$ in~\eqref{eq:joint_EB}.
The next proposition provides a sufficient condition, substantially more general than independence of $\mu_i$ and $\sigma_i^2$,
such that the oracle conditional ``p-values''
$P_{G(\cdot; \Pi)}(Z_i, S_i)$, i.e., the p-values defined in~\eqref{eq:conditional_pvalue} for the prior $G(\cdot; \Pi)$, are bona-fide 
conditional p-values that satisfy the properties in Proposition~\ref{prop:monotonicity}; see Supplement G for a proof.

\begin{prop}
    \label{prop:joint_hierarchical}
Suppose $(\mu_i, \sigma_i^2) \sim \Pi$ as in~\eqref{eq:joint_EB} and $(Z_i, S_i^2)$ are 
generated as in~\eqref{eq:full_sampling} conditional on \smash{$\mu_i, \sigma_i^2$}. Suppose further that, \begin{equation}
    \label{eq:indep_under_null}
\ind(\mu_i=0)\; \text{ is independent of }\,\sigma_i^2,
\end{equation}
then it holds $\Pi$-almost surely that:
$$ \PP[\Pi]{P_{ G(\cdot; \Pi)}(Z_i, S_i^2) \leq t \cond S_i^2, \mu_i=0  } = t \text{ for all } t \in [0,\,1].$$
\end{prop}
The assumption in~\eqref{eq:indep_under_null} seems unavoidable for the empirical partially Bayes methodology---it is the price
we pay for using all of \smash{$(S_1^2,\dotsc,S_n^2)$} to get a reference \emph{null} distribution of $Z_i$ conditional on $S_i^2$.
At the same time, \eqref{eq:indep_under_null} encapsulates a requirement that is substantially weaker than independence of $\mu_i$ and $\sigma_i^2$. For example, 
~\citet{stephens2017false} considers the following modeling assumption:
there exists $a \geq 0$ such that $\mu_i/\sigma_i^{a}$ is independent of $\sigma_i^2$ under $\Pi$.\footnote{By taking $a=0$,
this assumption posits independence of $\mu_i$ and $\sigma_i^2$.} This assumption implies the statement in~\eqref{eq:indep_under_null}.\footnote{The reason
is that $\ind(\mu_i =0) = \ind(\mu_i/\sigma_i^{a}=0)$.} \citet{lonnstedt2002replicated} posit that $\mu_i \mid \mu_i \neq 0, \sigma_i^2 \sim \mathcal{N}(0, \gamma \sigma_i^2)$ for some $\gamma > 0$, which also implies~\eqref{eq:indep_under_null} (as $\mu_i/\sigma_i$ is independent of $\sigma_i^2$ under $\Pi$).
From a practical perspective, we anticipate that~\eqref{eq:indep_under_null}
will approximately hold in many situations (and if it does not; then our methods still have the guarantees of Section~\ref{sec:compound}).

\section{Applications}
\label{sec:applications}

As we mentioned in the Introduction, the framework
we consider is widely applicable. 
Here we complement the differential methylation
analysis from Section~\ref{subsec:methylation}
by presenting two additional use-cases: detecting differentially expressed genes with microarray data
and detecting differentially abundant proteins with mass spectrometry data.
We further choose these datasets so as to get insights into empirical partially Bayes
analyses with different choices of the degrees of freedom $\nu$: $\nu=4$ for 
the methylation dataset, $\nu=11$ for the microarray dataset and $\nu=28$ for the proteomics
dataset.

\subsection{Differentially expressed genes after treatment with Ibrutinib}
\label{subsec:ibrutinib}

Here we consider a typical data analysis with microarray data.
As explained in the Introduction (also see~\citet{efron2010largescale} for an eloquent
exposition), microarrays were the driving force
of several developments in multiple testing and empirical Bayes methodology (e.g., limma)
in the beginning of the century. Nowadays the microarray technology has been largely
superseded by RNA-Seq, but nevertheless it finds use, e.g., 
for clinical applications~\citep{klaus2018end}.

\begin{figure}
    \centering
    \includegraphics[width=0.99\linewidth]{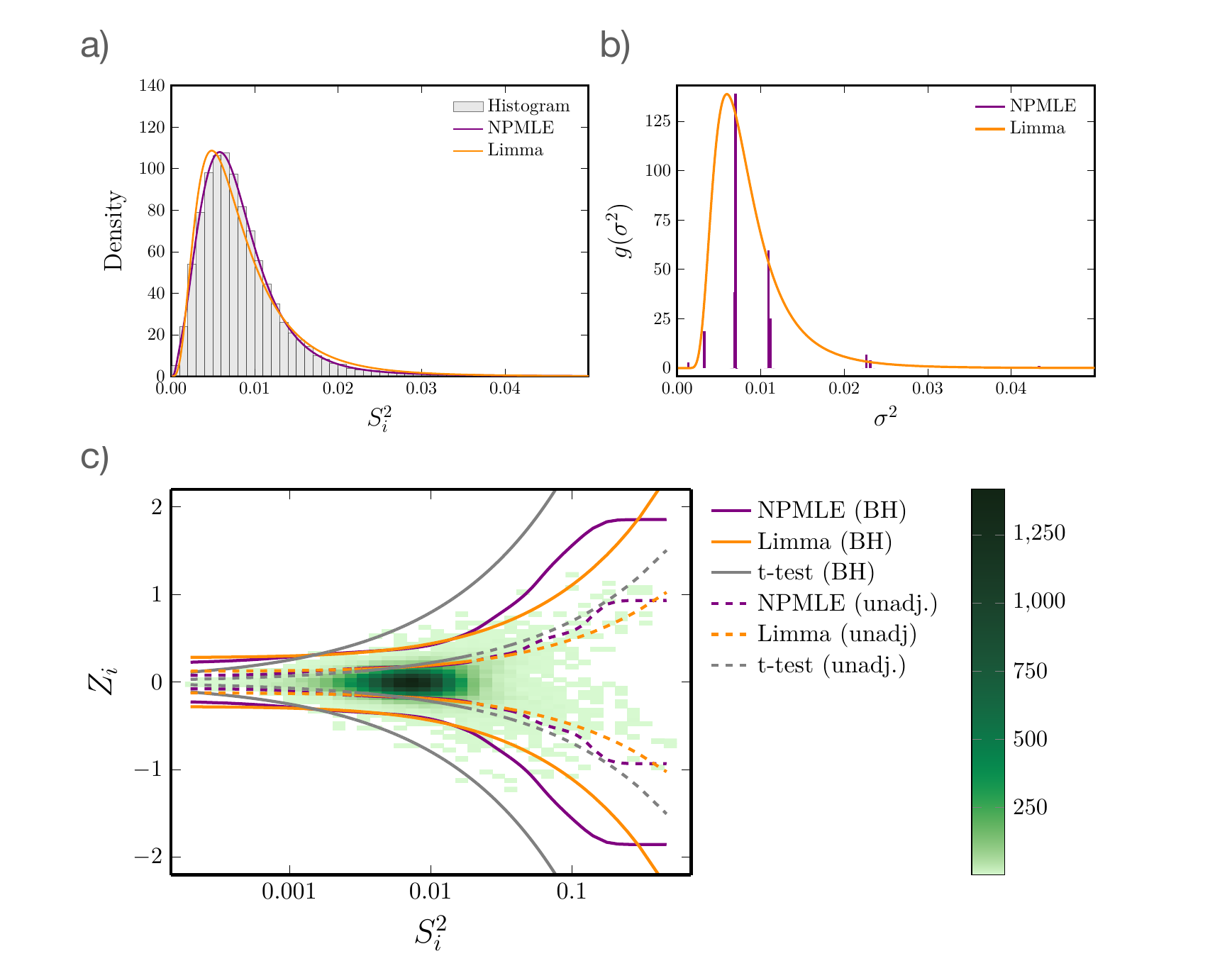}
    \caption{\textbf{Empirical partially Bayes analysis with microarray data of \citet{dietrich2017drugperturbationbased}}: {\small 
Each unit $i$ corresponds to a microarray probe ($n=48,100$) and its data is summarized as in~\eqref{eq:full_sampling}
with $\nu = 11$ degrees of freedom.
The three panels are analogous to the three panels of Fig.~\ref{fig:methylation}.}}
    \label{fig:microarray}
\end{figure}

The concrete application we consider is the following:
\citet{dietrich2017drugperturbationbased} collected samples from 12 
chronic lymphocytic leukemia (CLL) patients. Cells were incubated for 12 hours
with the drug Ibrutinib (which is used to treat CLL) and compared against
negative controls. Gene expression was measured with Illumina (incorporation, San Diego, California, USA) microarrays.
After performing a paired analysis for each probe
(pairing each patient's treated sample vs. the negative control),
the data may be summarized in the form~\eqref{eq:full_sampling} with $n=48,100$ probes
and $\nu = 12-1=11$ degrees of freedom.

Applying Benjamini-Hochberg to control the FDR at $\alpha=0.05$, we get
$92$ discoveries with the NPMLE-based p-values, $107$ discoveries with limma (the estimated prior~\eqref{eq:limma_shrinkage_assumption} 
has parameters $\hat{s}_0^2 \approx 0.007$, $\hat{\nu}_0 \approx 10.4$), and $14$ discoveries
with the standard t-test. Thus also in this problem, the empirical partially Bayes
approaches showcase substantial increase in discoveries compared to the t-test. 
The results are visualized in Fig.~\ref{fig:microarray}. We note
(Fig.~\ref{fig:microarray}a) that the NPMLE leads to a slightly better fit 
to the histogram of sample variances $S_i^2$ compared
to limma's inverse gamma model. The latter slightly overshoots for $S_i^2 \geq 0.02$,
and this in turn leads to less conservative rejection regions (Fig.~\ref{fig:microarray}c)
for large values of $S_i^2$ (which in turns explains the slight increase of 
discoveries with limma compared to the NPMLE).

\subsection{Differentially abundant proteins in breast cancer subtypes}
\label{subsec:proteomics}
 
Our next application pertains to a proteomics dataset analyzed by \citet{terkelsen2021high}, who conducted liquid chromatography tandem mass spectrometry (LC-MS/MS) on 34 tumor interstitial fluid samples from three breast cancer subtypes: luminal, Her2, and triple negative, processed in four experimental batches. For each of $n = 6,763$ proteins, the outcome variable is the $\log 2$-transformed intensity, measured through LC-MS/MS against a pooled internal standard. The investigators fit a linear model for each protein, with a design matrix including indicators for the $4$ batches and $3$ subtypes, estimating each variance with  $\nu = 34 - 1 - (4-1) - (3-1) = 28$ degrees of freedom. Here we seek to test for each protein whether the contrast for the difference of the coefficients for the luminal subtype
minus the coefficient for the Her2 subtype is equal to $0$. This problem can be cast into the framework 
of this paper with $\nu=28$ degrees of freedom (see Proposition~\ref{rema:contrasts}).

\begin{figure}
    \centering
    \includegraphics[width=0.99\linewidth]{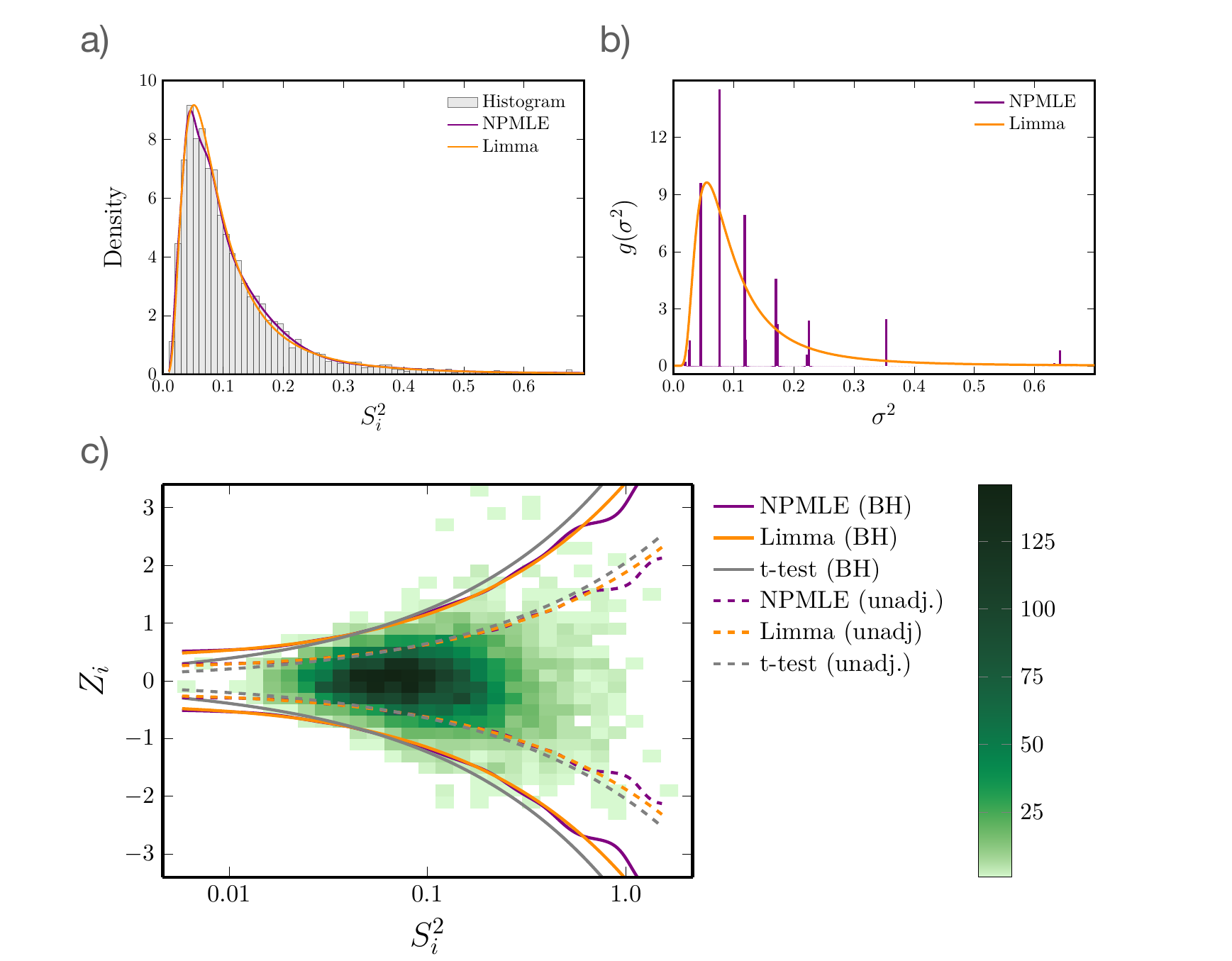}
    \caption{\textbf{Empirical partially Bayes analysis with proteomics data of \citet{terkelsen2021high}}: {\small 
    Each unit $i$ corresponds to a protein ($n=6,763$) and its data is summarized as in~\eqref{eq:full_sampling}
    with $\nu = 28$ degrees of freedom. The three panels are analogous to the three panels of Fig.~\ref{fig:methylation}.}}  
\label{fig:proteomics_bc}
\end{figure}

The Benjamini-Hochberg procedure (applied to control the FDR at $\alpha=0.05$) yields 
$93$ discoveries with the NPMLE-based p-values, $102$ discoveries with limma (the estimated prior~\eqref{eq:limma_shrinkage_assumption} 
has parameters $\hat{s}_0^2 \approx 0.077$, $\hat{\nu}_0 \approx 5.2$), and $74$ discoveries
with a standard t-test. As in our other applications, the empirical partially Bayes methods lead to more
discoveries compared to the standard t-test, however the difference is less pronounced in this case. The reason is that
the degrees of freedom are larger in this problem, so that there is less benefit in an empirical Bayes treatment of the sample variances $S_i^2$,
which are already reasonably accurate estimates of $\sigma_i^2$.\footnote{Indeed, all three p-values are asymptotically equivalent as $\nu \to \infty$.}
We visualize the results in Fig.~\ref{fig:proteomics_bc}. In Fig.~\ref{fig:proteomics_bc}c) we see that indeed the rejection thresholds
of the t-test are more closely aligned with the empirical partially Bayes testing thresholds compared to Fig.~\ref{fig:methylation}c) and Fig.~\ref{fig:microarray}c).
Nevertheless, it still holds true that the t-test is more liberal for small values of $S_i^2$, and more conservative for larger values of $S_i^2$.

\section{Simulations}
\label{sec:simulations}
In this section we conduct a simulation study to accompany our theoretical findings and the data analyses above.
Throughout our simulations, we generate data from model~\eqref{eq:full_hierarchical} with 
$n=10,000$. We vary the degrees of freedom $\nu \in \cb{2,4,8,16,32,64}$ and consider
three choices of prior $G$ for $\sigma_i^2$: 
\begin{enumerate}[wide, noitemsep, labelindent=0pt]
\item A point mass at $\sigma_i^2 =1$, i.e., $G= \delta_{1}$ as considered, e.g., in Example~\ref{exam:same_variance}.
\item The scaled inverse $\chi^2$ prior in~\eqref{eq:limma_shrinkage_assumption} with $\nu=6$ and $s_0^2=1$. 
\item The two point prior that assigns equal mass to $\sigma^2 = 10$ and $\sigma^2 =1$, that is, $G = ( \delta_{10} + \delta_1)/2$.
\end{enumerate}
The $\mu_i$ are generated as follows: first, we let $\mu_i=0$ for $9000$ of hypotheses 
(that is, the null proportion is equal to $0.9$) with the null indices picked at random. 
For the alternative hypotheses, we let $\mu_i \sim \nn(0,\, 16 \sigma_i^2)$. (In Supplement H.2, we also provide results for alternative signals drawn as $\mu_i \sim \nn(0,\, 16)$ and $\mu_i \sim \delta_4$.)

We compute p-values with the following four methods: 
\begin{enumerate}[wide, noitemsep, labelindent=0pt]
\item Standard \textbf{t-test}.
\item \textbf{Limma}.
\item \textbf{NPMLE}-based empirical partially Bayes.
\item \textbf{Oracle} partially Bayes that is provided with access to the ground truth $G$.
\end{enumerate}
We correct
for multiple testing with the Benjamini-Hochberg (\textbf{BH}) procedure
applied at a nominal false discovery rate level $\alpha = 0.1$.
For the oracle partially Bayes p-values, we also apply an alternative multiple testing
strategy based on Storey's null proportion adaptive procedure~\citep{storey2004strong}.\footnote{The rationale for 
considering Storey's procedure is to demonstrate that null proportion adaptive procedures may not control the $\mathrm{FDR}$
in settings wherein BH does; see the discussion in Footnote~\ref{footnote:storey_fails} following Theorem~\ref{manualtheorem:limma_bh_controls_the_FDR}.}
In total this leads to 5 methods under evaluation---t-test (BH), limma (BH), NPMLE (BH), oracle (BH),
and oracle (Storey). Finally, we evaluate methods in terms of their false discovery rate (FDR), and
power, which we define as $\mathrm{Pow}_n$ in~\eqref{eq:power_criteria}. We also consider a third metric that captures type-I error conditional on $S_1^2,\dotsc,S_n^2$. Concretely, this metric is defined as the probability that the p-value of the null hypothesis with the smallest sample variance is less than or equal to $0.2$:
\begin{equation}
\label{eq:minsamplevar}
\mathrm{MinSVarFP}(\leq 0.2) := \mathbb P[P_{\hat{i}(\mathrm{min})} \leq 0.2],\text{ where }\, \hat{i}(\mathrm{min}) := \argmin\cb{ S_i^2\,:\, i \in \Hnull}.
\end{equation}
If the null p-values are exactly uniform conditional on $S_1^2,\dotsc,S_n^2$ (as is the case for the oracle p-values in the setting of Proposition~\ref{prop:monotonicity}), then this probability should be equal to $0.2$. We note that the metric in~\eqref{eq:minsamplevar} only depends on the p-values, and not the multiple testing procedure. Thus, the results for oracle (BH) and oracle (Storey) are identical for this metric.
All three metrics are computed by averaging over $3,000$ Monte Carlo replicates of each simulation setting. 
\begin{figure}
    \centering
    \includegraphics[width=0.99\linewidth]{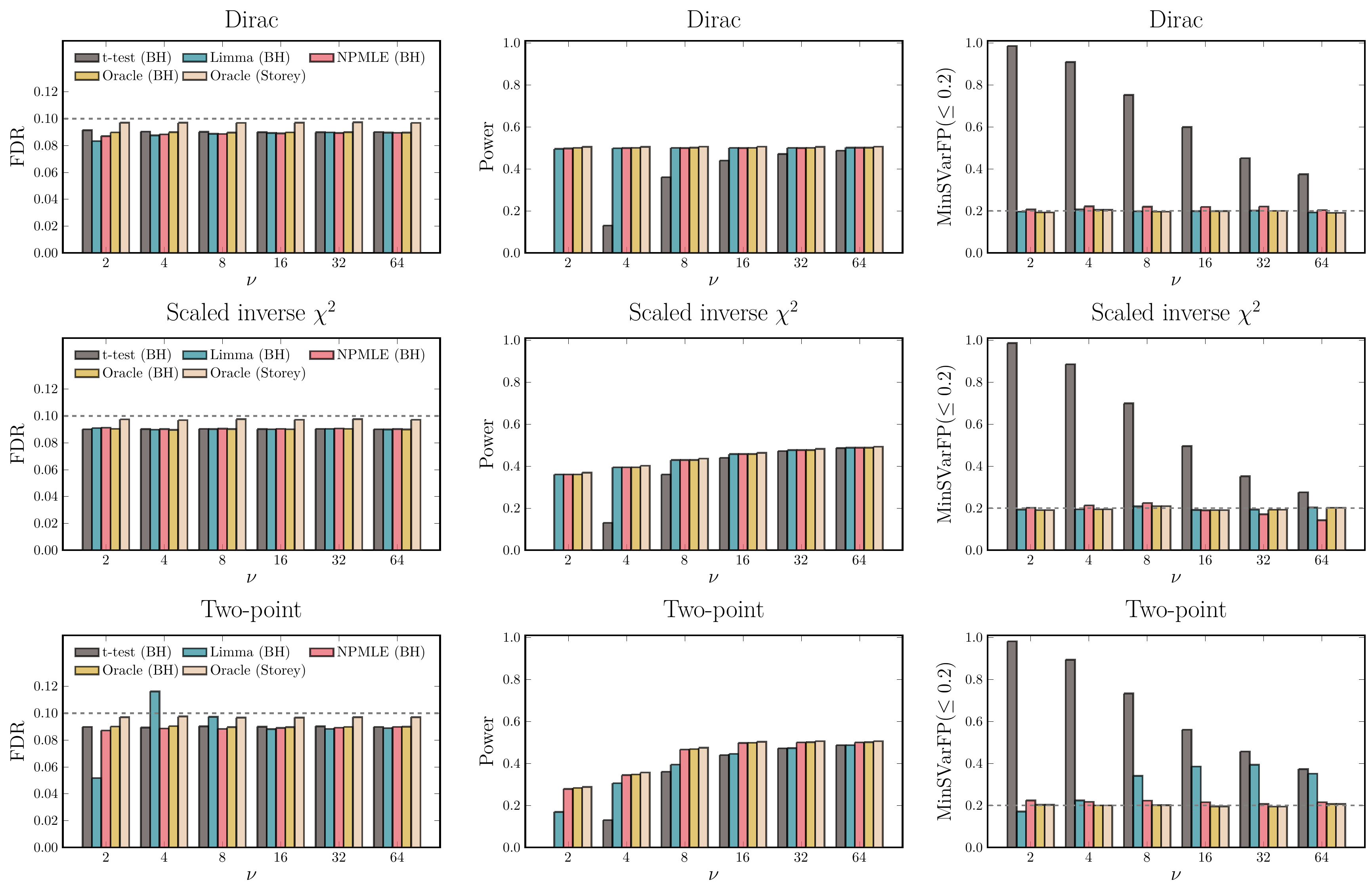}
    \caption{\textbf{Simulation results:} {\small 
    We compare $5$ methods---t-test (BH), limma (BH), NPMLE (BH), oracle (BH), and oracle (Storey)---across simulation settings
    that span three different choices of priors $G$ for the variances $\sigma_i^2$ (rows of the figure), as well as varying 
    degrees of freedom $\nu \in \cb{2,4,8,16,32,64}$ ($x$-axis of each plot). 
    The panels of the first column illustrate the false discovery rate of the different method, the second their power, and the third a metric for the type-I error of the null p-value with minimal sample variance (defined in~\eqref{eq:minsamplevar}).
    }}
    \label{fig:main_simulations}
\end{figure}
The results of the simulation are shown in Fig.~\ref{fig:main_simulations}. 
The first row shows the results for the simulation with the Dirac prior. 
All methods control the false discovery rate at $\alpha=0.1$.\footnote{In fact, we expect them to control the false discovery rate
at level $0.09 = n_0/n \cdot \alpha$, see e.g., Theorem~\ref{theo:limma_bh_controls_the_FDR} and the discussion preceding it.}
Limma and NPMLE are well-specified here\footnote{For limma this is not immediately obvious.
Note that the distribution of $\sigma_i^2$ in~\eqref{eq:limma_shrinkage_assumption}
converges weakly to $\delta_1$ as $\nu_0 \to \infty$ and $s_0^2=1$.}and have power similar to the oracle test (which here is just a z-test with known variance \smash{$\sigma^2=1$},
cf. Example~\ref{exam:same_variance}),
which is substantially greater than the power of the t-test. For example, for $\nu=2$, the t-test
has power almost equal to $0$, while the other methods have power $\approx 0.5$. The t-test becomes more competitive
as $\nu$ increases, while the performance of the other methods remains almost constant as $\nu$ increases (since here the empirical
partially Bayes p-values are determined by the point mass prior with the $\chi^2$-likelihood playing a minimal role). 
The second row of Fig.~\ref{fig:main_simulations} shows the scaled inverse $\chi^2$ simulation. Here the conclusions
are similar to the Dirac prior. However, the power of the (empirical) partially Bayes methods is not as 
large as in the Dirac case and it increases with $\nu$; the reason is that in this setting there is unexplained heterogeneity in the variances.
The results for the two point prior are shown in the last row of Fig.~\ref{fig:main_simulations}. 
In contrast to the previous two simulation settings, 
here the limma model is misspecified. We observe that this misspecification can lead to both conservative results (that is, FDR
substantially below nominal), as occurs when $\nu=2$, but also to anticonservative results (FDR above nominal)
as occurs when $\nu=4$. Finally, we note that throughout these simulations, Storey's multiple testing correction with oracle p-values
leads to FDR control (as follows from Proposition~\ref{prop:monotonicity}b and~\citet{storey2004strong}) and slightly higher power
than its BH counterpart. 
As regards the error metric in~\eqref{eq:minsamplevar}, we observe that the NPMLE empirical partially Bayes p-values have satisfying conditional uniformity properties in that $\mathrm{MinSVarFP}(\leq 0.2) \approx 0.2$ across settings with only small downwards deviations for the scaled inverse $\chi^2$ prior with $\nu \geq 32$; see Supplement H.1 for further discussion. In concordance with Proposition~\ref{prop:ttest_failure}, the t-test p-values strongly violate conditional uniformity. For the  two point prior and $\nu \geq 8$, limma also anti-conservatively violates conditional uniformity.\\

\noindent \textbf{Adversarial simulation:} The results of the preceding simulations are in accordance
to the theoretical results in Sections~\ref{sec:conditional_pvalues},~\ref{sec:empirical-bayes}, and~\ref{sec:joint_hierarchical}. 
We next consider a challenging setting for the empirical partially Bayes methods. We repeat the simulations from above
with the following tweak: we do not choose the null indices at random, 
but instead we sort the variances $\sigma_i^2$ and let $\mu_i=0$ for the $9,000$ hypotheses with the largest 
$\sigma_i^2$. Hence we are violating~\eqref{eq:indep_under_null} and the results of Section~\ref{sec:empirical-bayes}
are not applicable. Furthermore, as explained at the 
end of Section~\ref{subsec:average_significance}, 
the constellation in which the alternative hypotheses have smaller variances than the nulls
is adversarial/challenging for the empirical partially Bayes methods.
Yet, as explained in Section~\ref{sec:joint_hierarchical},
the theoretical results of the compound setting (Section~\ref{sec:compound}) are applicable. 

\begin{figure}
    \centering
    \includegraphics[width=0.99\linewidth]{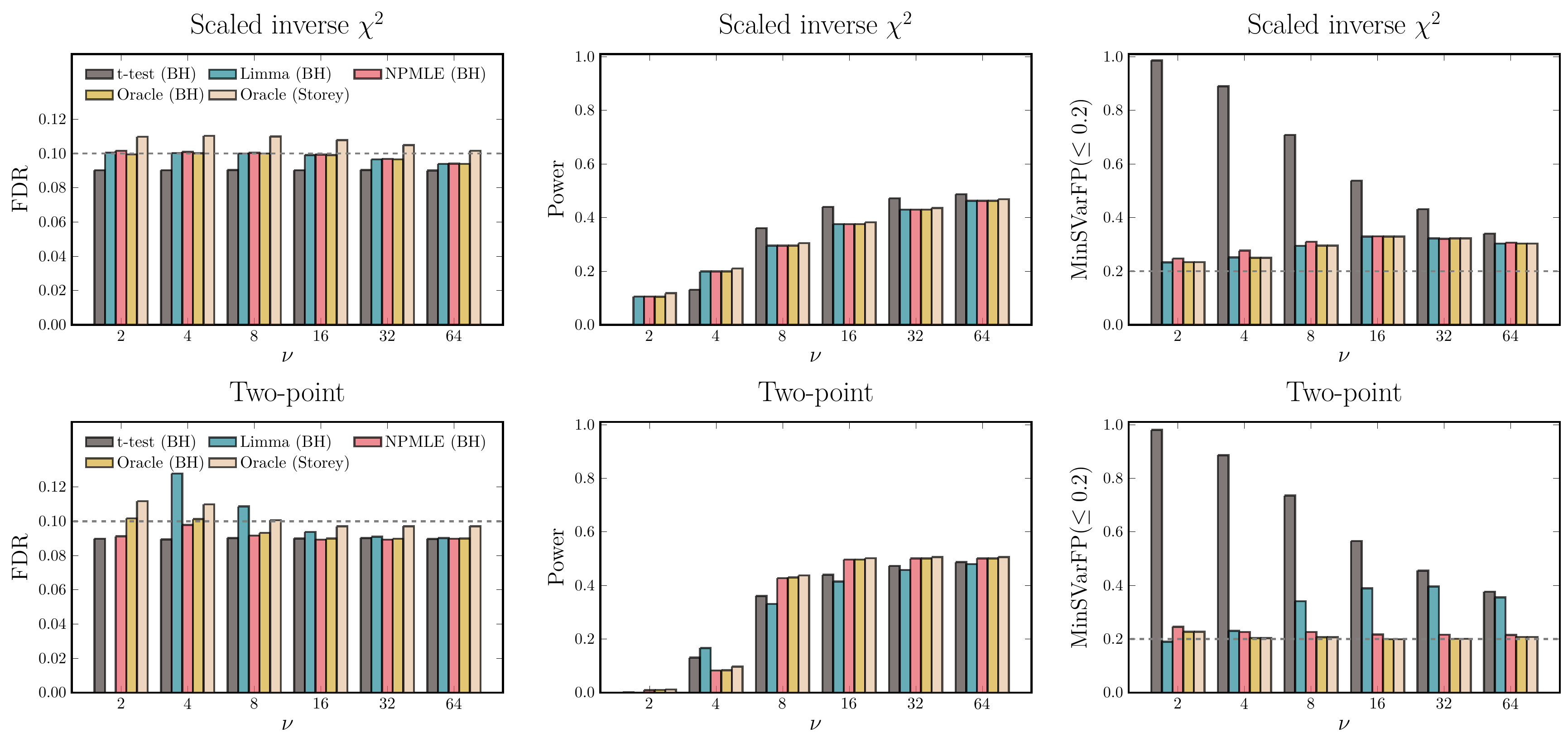}
    \caption{\textbf{Adversarial simulation results:} {\small 
    The panels in this figure are analogous to the panels of the last two rows of Fig.~\ref{fig:main_simulations}. 
    In contrast to Fig.~\ref{fig:main_simulations}, wherein the indices of null hypotheses were chosen at random, here
    null hypotheses are enriched for large variances $\sigma_i^2$.
    }}
    \label{fig:adv_simulations}
\end{figure}

Results of this simulation study are shown in Fig.~\ref{fig:adv_simulations}.\footnote{ Fig.~\ref{fig:adv_simulations} omits
the Dirac prior since then the adversarial and non-adversarial simulations are identical.}
We make the following observations: first, as explained in Section~\ref{sec:compound}, we observe that Storey's procedure
with the oracle partially Bayes p-values no longer controls the FDR in contrast to the BH procedure. Second, we note 
that under such adversarial arrangement of null hypotheses, in some cases the t-test can have more power than the empirical partially Bayes based 
approaches. For example, for the scaled inverse $\chi^2$ prior, the t-test has more power than the empirical partially Bayes approaches
for $\nu \geq 8$ degrees of freedom, but less power for $\nu \leq 4$ degrees of freedom. An explanation for this phenomenon is that under 
the adversarial constellation we consider, partially Bayes p-values are smaller than they should be for null hypotheses (so that Storey's procedure no longer controls FDR),
but larger than they should be for alternative hypotheses (that is, the smaller sample variances of the alternatives are shrunk upwards towards the bulk 
of larger sample variances of the nulls). When the degrees of freedom are very small, shrinkage (even in the wrong direction) is sufficiently
beneficial to lead to an increase in power compared to the t-test, while for larger degrees of freedom the t-test outperforms the partially Bayes p-values.
In terms of the metric for the conditional type-I error, $\mathrm{MinSVarFP}(\leq 0.2)$ as defined in~\eqref{eq:minsamplevar}, all methods are inflated for the scaled inverse $\chi^2$ prior. This does not contradict our results, since the conditional guarantees we derived do not hold in the compound setting. Nevertheless, the empirical partially Bayes methods still perform better than the t-test in terms of this metric.

All in all, we consider the results of this simulation to be encouraging: even under this adversarial setting (that would be unlikely
to occur in many real applications), Algorithm~\ref{algo:np_limma} with the NPMLE performs well, and in all cases
closely tracks the performance of the oracle partially Bayes p-values.

\section{Discussion}
\label{sec:discussion}

In this paper, we have made progress towards theoretically
understanding the problem
of empirical partially Bayes (multiple) hypothesis testing. The basic 
idea is to apply empirical Bayes techniques to ``shrink'' the nuisance
parameters, while treating the primary parameter of interest in standard
frequentist fashion. As we have mentioned throughout this work, the limma
approach and model is ubiquitous in applications of high-throughput biology. 
The more general idea of empirical partially Bayes testing comes up in even more
applications (albeit not under this name), e.g., in some of the most popular
approaches for analyzing RNA-Seq data~\citep{robinson2010edger, love2014moderateda}
based on models for count data.
We hope that our work will spur further research into empirical partially Bayes testing
and illuminate \emph{why} this approach has been so successful in practice.

We have also made methodological progress by deriving a fully nonparametric
generalization of the limma approach and by showing that even under 
nonparametric assumptions it is possible to approximate the ``oracle'' p-values
at a rate that is parametric (that is, of order $1/\sqrt{n}$) up to a logarithmic
factor. This suggests that our approach may have substantial benefits in practice
and at little cost when the parametric model is justified.
Further work could also consider finite-sample issues in more detail and 
modify the procedure to be more conservative, e.g., based on the ideas in
\citet{phipson2016robust, stephens2017false, ignatiadis2022confidence}, and \citet{xie2022discussion}.

Finally, our results also have implications for practitioners who use the 
parametric limma approach. They suggest the following:
\begin{enumerate}[wide, noitemsep, labelindent=0pt]
    \item \textbf{Sample variance histogram as a diagnostic:} The inverse chi-squared assumption for $G$ in~\eqref{eq:limma_shrinkage_assumption} 
    underlying the limma method may appear to be overly strong in applications. Nevertheless, the results of Proposition~\ref{prop:pvalue}
    and Lemma~\ref{lemm:closeness_of_marginals_plus_1} imply that limma can perform well even if \smash{$\hG^{\text{limma}}$} is a bad approximation of the
    ``true'' $G$ as long as the parametric model provides a good fit for the marginal density of the sample variances,
    that is, $f_{\hG^{\text{limma}}} \approx f_G$. The latter may be visually assessed by overlaying the implied marginal density
    with a histogram of the sample variances as e.g., in Fig.~\ref{fig:methylation}a) in this manuscript, 
    as well as \citet[Fig. 3 and 4]{wright2003random} and~\citet[Supplementary Fig. 1]{kammers2015detecting}.
    Hence we recommend accompanying a limma analysis,
    when appropriate, with this diagnostic plot.
    \item \textbf{Benjamini-Hochberg compared to Storey's adjustment:} A well-known disadvantage of the Benjamini-Hochberg (BH) method
    is that it is not null proportion adaptive, that is, it typically controls the false discovery rate at level 
    $\alpha \cdot \# \Hnull / n$ when applied at level $\alpha$ and so can be slightly conservative. 
    Adaptive methods such as Storey's procedure~\citep{storey2004strong} address this issue 
    by estimating the null proportion $\#\Hnull / n$ and then appropriately inflating the level at which the BH procedure is applied. 
    Our results in Sections~\ref{sec:compound},~\ref{sec:joint_hierarchical}, as well as 
    the simulation study in Section~\ref{sec:simulations} suggest that one should be cautious when applying adaptive multiple testing procedures
    in the context of limma (or more generally Algorithm~\ref{algo:np_limma}): adaptive procedures can fail to control
    the false discovery rate, e.g., when~\eqref{eq:indep_under_null} is violated, while the BH procedure can still work
    in such settings. Hence BH is to be preferred; we note that the software underlying limma indeed uses BH by default. 
\end{enumerate}

\subsection*{Reproducibility}
All numerical results of this paper are fully third-party reproducible, 
and we provide the code and data in Github under the following link:\\ 
\url{https://github.com/nignatiadis/empirical-partially-bayes-paper}.

\subsection*{Acknowledgments}
We would like to thank Taehyun Kim, Roger Koenker, and Peter McCullagh for helpful feedback on this manuscript. We also would like to thank the editor, Enno Mammen, the associate editor, and three anonymous reviewers for their constructive comments and suggestions. This work was completed in part with resources provided by the University of Chicago’s Research Computing Center. This research was supported by NSF grant DMS-2015376.

\bibliographystyle{abbrvnat}
\bibliography{chisq_eb}

\begin{thebibliography}{135}
\bibitem[Walsh(1949)]{walsh1949information}
J.~E. Walsh.
\newblock On the ``information'' lost by using a t-test when the population variance is known.
\newblock \emph{Journal of the American Statistical Association}, 44\penalty0 (245):\penalty0 122--125, 1949.
\end{thebibliography}

\newpage 

\appendix

\setcounter{page}{1}
\renewcommand{\thepage}{S\arabic{page}} 
\setcounter{table}{0}
\renewcommand{\thetable}{S\arabic{table}}

\setcounter{figure}{0}
\renewcommand{\thefigure}{S\arabic{figure}}

\setcounter{equation}{0}
\renewcommand{\theequation}{S\arabic{equation}}

\setcounter{prop}{0}
\renewcommand{\theprop}{S\arabic{prop}}
\setcounter{footnote}{0}
\renewcommand{\thefootnote}{S\arabic{footnote}}

\renewcommand{\theHprop}{S\arabic{prop}}

\section{General remarks on proofs and notation}
\label{sec:general_proof_remarks}
In the main text for many of the theoretical statements we provided two versions: one that applies to the hierarchical setting with \smash{$\sigma_i^2 \simiid G$}~\eqref{eq:random_sigma},
and a second starred version that applies to the compound setting in which $\boldsigma = (\sigma_1, \dotsc, \sigma_n)$ is fixed. For example, we have
Theorem~\ref{theo:hellinger_dist_convergence} and Theorem~\ref{manualtheorem:hellinger_dist_convergence}. Furthermore, in the main text we clarified whether an expectation
or probability is computed with respect to the hierarchical or compound setting by using the subscript $G$ ($\EE[G]{\cdot}$, $\PP[G]{\cdot}$), resp. $\boldsigma$  
($\EE[\boldsigma]{\cdot}$, $\PP[\boldsigma]{\cdot}$).

Throughout the supplement we often omit the subscript $G$ and $\boldsigma$, in which case an argument is to be understood as going through in both
the hierarchical and compound settings. This allows us to provide unified proofs for the two versions of our theoretical results. We clarify within the proofs
whenever a step requires different arguments for the two cases. For the unified treatment, we also introduce the following 
notation: 
\begin{itemize}
\item We write $P_i^* = P_*(z, s^2)$ for the oracle p-values, where:
\begin{itemize}
\item $P_*(z, s^2) = P_G(z,s^2)$ as defined in~\eqref{eq:conditional_pvalue} 
when $\sigma_i^2 \sim G$ (see~\eqref{eq:random_sigma}), and 
\item $P_*(z,s^2) = P_{\boldsigma}(z, s^2)$ as defined in~\eqref{eq:compound_pvalue} in the compound setting.
\end{itemize}
\item Analogously, we use the notation $f_*$ to denote either the marginal density (under the true prior $G$) or the average marginal density, that is we write
\begin{itemize}
    \item $f_*=f_G$ as defined in~\eqref{eq:marginal_density} when $\sigma_i^2 \sim G$ (see~\eqref{eq:random_sigma}), and
    \item $f_*=f_{\boldsigma}$ as defined in~\eqref{eq:average_marginal_density} in the compound setting.
\end{itemize}
\item We write $G_*=G$ when $\sigma_i^2 \sim G$ (as in~\eqref{eq:random_sigma}) and  $G_*=G(\boldsigma)$ in the compound setting.
\end{itemize}

\noindent \textbf{Further notation:} Throughout the proofs we write $A \lesssim_{a} B$ to denote that there exists a constant $C=C(a)$ that depends only on $a$
such that $A \leq C \cdot B$. Analogously we also write $A \gtrsim_a B$. 

\section{Properties of \texorpdfstring{$\chi^2$-}{chi-squared }likelihood and precision parametrization}
\label{sec:precision_parametrization}

Often in the proofs it will be convenient to reparameterize the likelihood by the precisions $\tau_i^2 = 1/\sigma_i^2$: 

\begin{equation}
\label{eq:chisq_precision_likelihood}
p(s^2 \mid \tau^2) \equiv p(s^2 \mid \tau^2, \nu)  \coloneqq \frac{\nu^{\nu/2} \p{\tau^2}^{\nu/2}}{2^{\nu/2}\Gamma(\nu/2)}  \p{s^2}^{\nu/2-1} \exp\p{-\frac{\nu \tau^2 s^2}{2}}.
\end{equation}

\begin{lemm}
    \label{lemm:chi_square_likelihood}
    The $\chi^2$-likelihood $p(s^2 \cond \tau^2) = p(s^2 \cond \tau^2, \nu)$ has the following properties ($\nu \geq 2$).
    \begin{enumerate}[label=\Alph*.]
        \item $\sup\cb{p(s^2 \cond \tau^2)   \,:\, s^2 > 0,\; 0 < \tau^2 \leq L^{-1} } \lesssim_{\nu} L^{-1} \lesssim_{\nu, L} 1.$
        \item  $\tau^2 \mapsto p(s^2 \cond \tau^2)$ is decreasing in the interval $[1/s^2,\,\infty)$ and
        increasing in $(0, 1/s^2)$.
        \item It holds that:
        $$ p(s^2 \cond \tau^2) \leq p(s^2 \cond 1/U) \lesssim_{\nu, U} \exp\p{- \nu s^2 / (4U)} \text{ for } s^2 \geq U \geq 1/\tau^2.$$
        In particular, suppose $s^2 \geq 4(\kappa/\nu) U$ with $\kappa \geq  \nu/4$ and that $U \geq 1/\tau^2$, then:
        $$ p(s^2 \cond \tau^2)  \leq p(s^2 \cond 1/U) \lesssim_{\nu, U} \exp(-\kappa). $$
        \item $\sup\cb{ \frac{\partial p(s^2 \mid \tau^2)}{\partial (\tau^2)}\,:\, \tau^2 > 0, s^2>0} \lesssim_{\nu} 1.$
    \end{enumerate}
\end{lemm}

\begin{proof}[Proof of Lemma~\ref{lemm:chi_square_likelihood}]
The following inequality will be useful:
\begin{equation}
    \label{eq:basic_gamma_inequality}
    x^k \exp(-x) \leq k^k \exp(-k) \text{ for all } x > 0, k \geq 0,
\end{equation}
where we interpret $0^0$ as $1$.
The inequality follows by observing that the LHS is maximized when $x=k$.

Let us now start with the claim in part A. Note that,
$$ p(s^2 \cond \tau^2) \lesssim_{\nu}  \p{\tau^2}^{\nu/2} \p{\frac{ \nu s^2}{2}}^{\nu/2-1} \exp\p{-\frac{\nu \tau^2 s^2}{2}} = \tau^2 \cdot \cb{\p{\frac{ \nu \tau^2 s^2 }{2}}^{\nu/2-1} \exp\p{-\frac{\nu \tau^2 s^2}{2}}  }.$$
We apply inequality~\eqref{eq:basic_gamma_inequality} with $k=\nu/2-1$ and we conclude that the result in A. holds by noting that we take the supremum over $\tau^2 \leq L^{-1}$.

We next take the derivative of~\eqref{eq:chisq_precision_likelihood} with respect to $\tau^2$:
$$
\frac{\partial p(s^2 \mid \tau^2)}{\partial (\tau^2)} = \frac{\nu^{\nu/2} \p{s^2}^{\nu/2-1}}{2^{\nu/2}\Gamma(\nu/2)}\frac{\nu}{2}\p{\tau^2}^{\nu/2-1}\p{1- \tau^2 s^2} \exp\p{-\frac{\nu \tau^2 s^2}{2}}.
$$
From here we can read off the monotonicity claimed in B. The result in C. follows by the monotonicity property: for $s^2 \geq U \geq 1/\tau^2$ we have that $1/s^2 \leq 1/U \leq \tau^2$ and so:
$$
\begin{aligned}
p(s^2 \cond \tau^2) & \leq p(s^2 \cond 1/U) \\ 
& \lesssim_{\nu} \frac{1}{U^{\nu/2}}  \p{s^2}^{\nu/2-1} \exp\p{-\frac{\nu s^2}{2 U}} \\ 
& = \cb{ \frac{1}{U^{\nu/2}}  \p{s^2}^{\nu/2-1} \exp\p{-\frac{\nu s^2}{4 U}}}\exp\p{-\frac{\nu s^2}{4 U}} \\ 
& \lesssim_{\nu, U} \exp\p{-\frac{\nu s^2}{4 U}}.
\end{aligned}
$$
Thus, if $s^2 \geq 4(\kappa/\nu) U$ with $\kappa \geq \nu/4$ we have that $s^2 \geq U$ and so for all $\tau^2 \geq 1/U$:
$$ p(s^2 \cond \tau^2) \leq p(s^2 \cond 1/U)  \lesssim_{\nu, U} \exp(-\kappa).$$
Finally,
for part D. we apply inequality~\eqref{eq:basic_gamma_inequality} twice: once with $k=\nu/2$ and once with $k=\nu/2-1$.
\end{proof}

\section{Proofs for Section~\ref{sec:conditional_pvalues}}

\subsection{Proof of Proposition~\ref{prop:minimal_sufficient}}
\label{subsec:proof_prop_minimal_sufficient}
\begin{proof}[Proof of Proposition~\ref{prop:minimal_sufficient}]
Ancillarity of $S^2$ follows since under~\eqref{eq:full_hierarchical},
the distribution of $S^2$ does not depend on $\mu$: $S^2$ is 
distributed as the product $\sigma^2 \cdot X$
for independently drawn $\sigma^2 \sim G$,
and $\nu X \sim \chi^2_{\nu}$ (where neither $G$, nor
$\chi^2_{\nu}$ depends on $\mu$).

We continue with minimal sufficiency of $(Z, S^2)$.
Sufficiency of $(Z, S^2)$ is immediate in this setting 
(it is the full data generated from model~\eqref{eq:full_hierarchical}).
Hence it remains to prove minimality; we will use 
the characterization of minimal sufficiency
in terms of likelihood shapes~\citep[Theorem 3.11]{keener2010theoretical}.
Write $p(z, s^2 \cond \mu)$ for 
the joint density of $(z, s^2)$ on $\RR \times \RR_+$ 
under~\eqref{eq:full_hierarchical} with respect to the Lebesgue
measure, where we make the
dependence on $\mu$ explicit. Next, fix $x_1=(z_1, s_1^2)$, $x_2=(z_2, s_2^2)$
with $s_1^2, s_2^2 > 0$ and suppose that:
\begin{equation}
    \label{eq:lik_shape}
p(z_1, s_1^2 \cond \mu) = C(x_1, x_2) p(z_2, s_2^2 \cond \mu) \text{ for all } \mu,
\end{equation}
with $C(\cdot)$ being a function of $x_1, x_2$ only (and not of $\mu$).
To prove minimality,
it suffices to prove that~\eqref{eq:lik_shape} implies that
$(z_1, s_1^2)=(z_2, s_2^2)$.

We will start by showing $z_1 = z_2$. Suppose otherwise, i.e., $z_1 \neq z_2$.
Using the precision parametrization described in Supplement~\ref{sec:precision_parametrization},
~\eqref{eq:lik_shape} is equivalent to:
$$
\begin{aligned}
&\int (s_1^2)^{\nu/2-1} (\tau^2)^{(\nu+1)/2}\exp(-\tau^2(z_1-\mu)^2/2)\exp(-\nu\tau^2s_1^2/2) dH(\tau^2) \\ 
& \;\;\;\;\;\;\;\;\;\;= C(x_1, x_2) \int (s_2^2)^{\nu/2-1} (\tau^2)^{(\nu+1)/2}\exp(-\tau^2(z_2-\mu)^2/2)\exp(-\nu\tau^2s_2^2/2) dH(\tau^2).
\end{aligned}
$$
Now pick $\mu$ such that:
$(z_2-\mu)^2 - (z_1 - \mu)^2 = (s_1^2 - s_2^2)\nu$. This is possible because we 
assumed $z_1 \neq z_2$.
Then, the terms within $\exp(\cdot)$ on the LHS and RHS become identical, i.e., for some $\ell = \ell(x_1,x_2) > 0$:
$$\int (s_1^2)^{\nu/2-1} (\tau^2)^{(\nu+1)/2}\exp(-\ell \tau^2)dH(\tau^2) = C(x_1, x_2)\int (s_2^2)^{\nu/2-1} (\tau^2)^{(\nu+1)/2}\exp(-\ell \tau^2)dH(\tau^2).$$
This in turn implies that $C(x_1, x_2) =  (s_1^2)^{\nu/2-1}/ (s_2^2)^{\nu/2-1}$ and so:
$$
\begin{aligned}
&\int  (\tau^2)^{(\nu+1)/2}\exp(-\tau^2/2\cb{ (z_1-\mu)^2 +\nu s_1^2}) dH(\tau^2) \\ 
& \;\;\;\;\;\;\;\;\;\;= \int  (\tau^2)^{(\nu+1)/2}\exp(-\tau^2/2\cb{ (z_2-\mu)^2 +\nu s_2^2}) dH(\tau^2) \text{ for all } \mu \in \RR.
\end{aligned}
$$
However, the above leads to a contradiction: suppose without loss of 
generality that $s_1^2 \geq s_2^2$ and pick $\mu=z_2$. Then $(z_1-\mu)^2 +\nu s_1^2 > \nu s_2^2 = (z_2-\mu)^2 +\nu s_2^2$, so
that the above equality could not hold. Thus $z_1 = z_2$.

It remains to prove that also $s_1^2 = s_2^2$. We have already shown
that $z_1 = z_2$. By translation, we may assume without loss of generality
that in fact $z_1 = z_2 =0$. Further let $u = \mu^2/2$ and $\tilde{C}(x_1,x_2) = C(x_1,x_2)((s_2/s_1)^2)^{\nu/2-1}$. Then:
$$
\begin{aligned}
&\int  \exp(-u\tau^2) (\tau^2)^{(\nu+1)/2}\exp(-\tau^2\nu s_1^2/2) dH(\tau^2) \\ 
& \;\;\;\;\;\;\;\;\;\;= \tilde{C}(x_1, x_2) \int \exp(-u\tau^2) (\tau^2)^{(\nu+1)/2}\exp(-\tau^2\nu s_2^2/2) dH(\tau^2) \text{ for all } u \geq 0.
\end{aligned}
$$
On the LHS and RHS we have two finite measures supported on $(0, \infty)$
that have identical Laplace transforms.  Hence by uniqueness
of the Laplace transform (e.g., \citet[Theorem 22.2]{billingsley1995probability}) these
two measures must be identical, which is to say:
$$ (\tau^2)^{(\nu+1)/2}\exp(-\tau^2\nu s_1^2/2)  =  \tilde{C}(x_1, x_2) (\tau^2)^{(\nu+1)/2}\exp(-\tau^2\nu s_2^2/2)\;\;\, H\text{-almost surely}.$$
Since $G$, and so $H$ is not degenerate, the above must hold
for at least two values of $\tau \in (0,\;\infty)$, which implies that $s_1^2 = s_2^2$.
\end{proof}

\subsection{Proof of Proposition~\ref{prop:limma_as_conditional}}
\label{sec:proof_prop_limma_as_conditional}

\begin{proof}[Proof of Proposition~\ref{prop:limma_as_conditional}]
    
It will be convenient to parameterize the problem in terms of the precision $\tau^2=1/\sigma^2$
as in Supplement~\ref{sec:precision_parametrization}. In what follows 
we may assume without loss of generality that $\mu=0$; otherwise
one can consider the translated family $Z \mapsto Z-\mu$.

The joint density of $Z, S^2, \tau^2$ with respect to the Lebesgue measure 
on $\RR \times \RR_{>0} \times \RR_{>0}$ is given by (up to a constant):
$$
\begin{aligned}
& p(z, s^2, \tau^2)  \\ 
=\; & C\p{\tau^2}^{1/2}\exp\p{-\frac{z^2\tau^2}{2}} \cdot \p{s^2}^{\nu/2-1} \p{\tau^2}^{\nu/2} \exp\p{-\frac{\nu s^2 \tau^2}{2}} \cdot  \p{\tau^2}^{\nu_0/2-1}\exp\p{-\frac{\nu_0 s_0^2 \tau^2}{2}}\\ 
=\; & C\p{s^2}^{\nu/2-1}  \p{\tau^2}^{(\nu_0 + \nu +1)/2 - 1} \exp\p{- \frac{ (z^2 + \nu s^2 + \nu_0 s_0^2)\tau^2}{2}}.
\end{aligned}
$$
Hence marginalizing over $\tau^2$:
$$
\begin{aligned}
p(z, s^2) &= \int_{0}^{\infty}p(z, s^2, \tau^2) d(\tau^2) \\ 
&= C \p{s^2}^{\nu/2-1}  \int_{0}^{\infty}\p{\tau^2}^{(\nu_0 + \nu +1)/2 - 1} \exp\p{- \frac{ (z^2 + \nu s^2 + \nu_0 s_0^2)\tau^2}{2}}d(\tau^2)\\ 
&= C \p{s^2}^{\nu/2-1} \p{\frac{z^2 + \nu s^2 + \nu_0 s_0^2}{2}}^{-(\nu_0 + \nu +1)/2 } \Gamma((\nu_0+ \nu + 1)/2) \\ 
&= C' \p{s^2}^{\nu/2-1}\p{z^2 + \nu s^2 + \nu_0 s_0^2}^{-(\nu_0 + \nu +1)/2 }.
\end{aligned}
$$
We make the following parenthetical remark: the joint marginal distribution of $(Z, S^2)$ is called the Student-Siegel 
distribution in~\citet[Table 2.2., page 25]{aitchison1975statistical}.

Continuing with our proof, let $\tilde{t} = z  / \sqrt{(\nu_0 s_0^2 + \nu s^2)/(\nu_0 + \nu)}$ and consider the mapping $\psi(z, s^2) = (\tilde{t}, s^2)$. Then,
$$\abs{\text{det}( \nabla \psi(z, s^2))} = \abs{ \partial \psi_1(z, s^2) / \partial z } =  \abs{ \partial \tilde{t} / \partial z } = 1 /\sqrt{(\nu_0 s_0^2 + \nu s^2)/(\nu_0 + \nu)}.$$
Hence:
$$
\begin{aligned}
p(\tilde{t}, s^2) &= p( \psi^{-1}(\tilde{t},s^2))/\abs{\text{det}( \nabla \psi(z, s^2))} \\ 
&=  C' \p{s^2}^{\nu/2-1}\sqrt{(\nu_0 s_0^2 + \nu s^2)/(\nu_0 + \nu)}\p{ \tilde{t}^2\frac{\nu_0 s_0^2 + \nu s^2}{\nu_0 + \nu} + \nu s^2 + \nu_0 s_0^2}^{-(\nu_0 + \nu +1)/2 }\\ 
&= C{''} \p{\tilde{t}^2 + \nu_0 + \nu }^{-(\nu_0 + \nu +1)/2}  \cdot \p{s^2}^{\nu/2-1} (\nu_0 s_0^2 + \nu s^2)^{-(\nu_0 + \nu)/2}.
\end{aligned}
$$
From the above result we can immediately read off the following two
distribution statements: $\widetilde{T}$ and $S^2$ are independent (with respect to their
marginal distribution---they are \emph{not} independent conditional on 
$\tau^2$). Second, $\widetilde{T} \sim t_{\nu_0+\nu}$. We also refer
the reader to~\citet[Section 4]{smyth2004linear} for further
discussion of the above distributional conclusions.

It remains to note the following two distributional equalities to conclude the proof of the first part of the proposition. First,
$$ Z \mid (S^2=s^2) \;\;   \stackrel{\mathcal{D}}{=}\;\;  \tilde{S} \cdot \widetilde{T} \mid (S^2=s^2),$$
since $Z = \widetilde{T} \widetilde{S}$. Second,
$$ \tilde{S} \cdot \widetilde{T} \mid (S^2=s^2) \;\;  \stackrel{\mathcal{D}}{=} \;\; \tilde{s}\cdot T_{\nu_0+\nu},$$
since $\widetilde{S}$ is a deterministic function of $S^2$, $S^2$ and $\widetilde{T}$
are independent, and $\widetilde{T} \sim t_{\nu_0+\nu}$.

For the second part of the proposition, fix $z \in \RR$. 
Also write $\tilde{s}$, $\tilde{t} = z/\tilde{s}$ for the
observed values of $\widetilde{S}$ and $\widetilde{T}$.
It holds that:
$$
\begin{aligned}
P_G(z, s^2) &\stackrel{(i)}{=} \PP[G]{ \abs{Z^{H_0}} \geq \abs{z} \cond S^2=s^2} \\ 
& \stackrel{(ii)}{=} \PP[G]{ \tilde{s}\cdot T_{\nu_0+\nu} \geq \abs{z}} \\ 
& = \PP[G]{ T_{\nu_0+\nu} \geq \abs{z}/\tilde{s}} \\ 
& = \PP[G]{ T_{\nu_0+\nu} \geq \abs{\tilde{t}}}\\ 
& \stackrel{(iii)}{=} P^{\text{limma}}_{\nu_0, s_0^2}(z, s^2).
\end{aligned}
$$
In $(i)$ we used the definition of the conditional p-values in~\eqref{eq:conditional_pvalue}.
In $(ii)$ we used the distributional result of the first part of the proposition.
Finally, in $(iii)$ we used the definition of the limma p-value in~\eqref{eq:moderated_t}. 

Finally, for completeness we also prove the RHS equality of~\eqref{eq:conditional_pvalue}. Note that,
$$
\begin{aligned}
\PP[G]{ \abs{Z^{H_0}} \geq \abs{z} \;\cond\; S^2=s^2} &=\EE[G]{ \PP[G]{\abs{Z} \geq \abs{z} \;\cond\; \sigma^2, S^2=s^2} \;\cond\;  S^2=s^2} \\ 
&= \EE[G]{ \PP[G]{\abs{Z^{H_0}} \geq \abs{z} \;\mid\; \sigma^2} \;\cond\; S^2=s^2} \\ 
&=  \EE[G]{ 2(1-\Phi(|z|/\sigma)) \;\cond\;  S^2=s^2}.
\end{aligned}
$$
\end{proof}

\subsection{Proof of Proposition~\ref{prop:monotonicity}}
\label{sec:proof_prop_monotonicity}
\begin{proof}
As mentioned directly after the proposition, the first result is an
application of the probability integral transform conditional on $S^2$.
The second result follows by iterated expectation:
$$ \PP[G]{P_G(Z, S^2) \leq t} = \EE[G]{\PP[G]{P_G(Z, S^2) \leq t \cond  S^2}} = \EE[G]{t} = t \text{ for all } t \in [0,\,1].$$
It remains to prove the third result regarding the monotonicity.
We use the precision parameterization of Supplement~\ref{sec:precision_parametrization} wherein $\tau^2 \sim H$. 
We also let $\bar{\Phi} = 1 - \Phi$. By definition,
$P_G(z, s^2)/2 = \EE[G]{ \bar{\Phi}(\abs{z}\tau) \mid S^2=s^2}$. The monotonicity in terms of $\abs{z}$ follows because
$ \bar{\Phi}(\abs{z}\tau) \leq \bar{\Phi}(\abs{z'}\tau)$ for any $\abs{z'} \leq \abs{z}$. 

To show monotonicity with respect to $s^2$, we first note that (as $S^2 \mid \tau^2 \sim \chi^2_{\nu}/(\nu \tau^2)$ and $\tau^2 \sim H$):
$$
\begin{aligned}
\frac{1}{2}P_G(z, s^2)=  \frac{ \int \bar{\Phi}(\abs{z}\tau)(\tau^2)^{\nu/2}\exp\p{-\frac{\nu \tau^2 s^2}{2}} dH(\tau^2)}{\int (\tau^2)^{\nu/2}\exp\p{-\frac{\nu \tau^2 s^2}{2}} dH(\tau^2)}.
\end{aligned}
$$
Let $A_s$ be the probability measure that is absolutely continuous with respect to $H$ with Radon-Nikodym derivative:
$$ \frac{dA_s(\tau^2)}{dH} = (\tau^2)^{\nu/2}\exp\p{-\frac{\nu \tau^2 s^2}{2}} \bigg / \int (\tau^2)^{\nu/2}\exp\p{-\frac{\nu \tau^2 s^2}{2}} dH(\tau^2).$$
Then, taking the derivative of $P_G(z, s^2)$ with respect to $s^2$, we find that:
$$
\frac{\partial P_G(z, s^2)}{\partial (s^2)} =  -\nu\p{\int \bar{\Phi}(\abs{z}\tau)\tau^2 dA_s(\tau^2) -  \int \bar{\Phi}(\abs{z}\tau) dA_s(\tau^2)\int \tau^2 dA_s(\tau^2)}.
$$
To show the monotonicity claim, it suffices to show that $\partial P_G(z, s^2) / \partial s^2 \geq 0$, i.e., that $\Cov[A_s]{\bar{\Phi}(\abs{z}\tau),\, \tau^2} \leq 0$ when $\tau^2 \sim A_s$. 
But note that $\tau^2 \mapsto \tau^2$ is non-decreasing and that $\tau^2 \mapsto \bar{\Phi}(\abs{z}\tau)$ is non-increasing. Hence the
conclusion follows by Chebyshev's ``other'' inequality (e.g.,~\citet{fink1984chebyshev}).
\end{proof}

\subsection{Proof of Proposition~\ref{prop:ttest_failure}}
\label{sec:proof_prop_ttest_failure}
\begin{lemm}
    \label{lemm:chisquare_cdf}
Let $X_{\nu}$ be a random variable distributed according to the 
$\chi^2_{\nu}$-distribution ($\nu \geq 2$ degrees of freedom) and denote its distribution function by $F_{\nu}$. Then:
$$F_{\nu}(u)u^{-\nu/2} \to c(\nu) \coloneqq \frac{1}{2^{\nu/2}\Gamma(\nu/2+1)}> 0 \text{ as } u \searrow 0.$$
Furthermore, there exists $C(\nu)$ such that:
$$F_{\nu}(u) \leq C(\nu) u^{\nu/2} \text{ for all } u > 0.$$
\end{lemm}

\begin{proof}
Let us write $f_{\nu} = F'_{\nu}$ for the density of $X_{\nu}$. Then, we have:
$$ f_{\nu}(s) = \frac{1}{2^{\nu/2}\Gamma(\nu/2)} (s)^{\nu/2-1} \exp(-s/2) = \frac{1}{2^{\nu/2}\Gamma(\nu/2)}  \sum_{k=0}^{\infty} \frac{(-1)^k}{2^k k!} s^{k + \nu/2 - 1}.$$
Integrating the above term by term (and noting that all series defined are absolutely convergent), we see that:
$$ F_{\nu}(s) = \frac{1}{2^{\nu/2}\Gamma(\nu/2)}  \sum_{k=0}^{\infty} \frac{(-1)^k}{2^k k! (k+\nu/2)} s^{k + \nu/2}.$$
From here we can immediately read off the first claim about the limit of $F_{\nu}(u)u^{-\nu/2}$ as $u \to 0$.

For the second claim, let us define the function $h(u) \coloneqq  C(\nu) u^{\nu/2} - F_{\nu}(u)$ (where we will pick $C(\nu)$ below). Note
that $h(0) = 0$ and 
$$h'(u) = (\nu/2) C(\nu) u^{\nu/2-1} -  \frac{1}{2^{\nu/2}\Gamma(\nu/2)} u^{\nu/2-1} \exp(-u/2),$$
and so this can be
made $\geq 0$ for all $u>0$ by choosing $C(\nu)$ large enough. 
\end{proof}

\begin{proof}[Proof of Proposition~\ref{prop:ttest_failure}]
The uniformity of the t-test p-values and their monotonicity are elementary. It remains to prove the failure of conditional uniformity,
which is equivalent to the following statement:
$$ \lim_{\delta \searrow 0} \PP[G]{P^{\text{ttest}}(Z, S^2) > \alpha \cond S^2 \leq \delta} = 0 \text{ for all }\alpha \in (0,\,1].$$
We start by observing that the event
$\cb{P^{\text{ttest}}(Z, S^2) > \alpha}$ is identical
to the event $\cb{\abs{Z} < q S}$, where 
$q>0$ is the $(1-\alpha/2)$-quantile of the t-distribution
with $\nu$-degrees of freedom.
Hence:
$$
\begin{aligned}
\PP[G]{P^{\text{ttest}}(Z, S^2) > \alpha \cond S^2 \leq \delta} &= \PP[G]{\abs{Z} < q S  \cond S^2 \leq \delta} \\ 
&= \EE[G]{ \PP[G]{\abs{Z} < q S \cond \sigma^2, S^2 } \cond S^2 \leq \delta} \\ 
&= \EE[G]{ \p{2\Phi(q S/\sigma )-1} \cond S^2 \leq \delta} \\ 
&= \frac{\EE[G]{ \p{2\Phi(q S/\sigma)-1} \ind(S^2 \leq \delta)}}{ \PP[G]{ S^2 \leq \delta} }.
\end{aligned}
$$
As in the statement of Lemma~\ref{lemm:chisquare_cdf}, we let $X_{\nu}$ be a random variable distributed according to the 
$\chi^2_{\nu}$-distribution and denote its distribution function by $F_{\nu}$. Then:
$$
\begin{aligned}
\PP[G]{ S^2 \leq \delta} &= \PP[G]{X_{\nu} \leq \nu \delta/\sigma^2} \\ 
&= \EE[G]{ \PP[G]{X_{\nu}  \leq \nu \delta/\sigma^2   \cond \sigma^2   }} \\ 
& = \EE[G]{ F_{\nu}(\nu \delta/\sigma^2 )}.
\end{aligned}
$$
By Lemma~\ref{lemm:chisquare_cdf}, $F_{\nu}(u)u^{-\nu/2} \to c(\nu) > 0$ as $u \to 0$. 
Hence:
$$
\begin{aligned}
&\liminf_{\delta \to 0} \cb{\PP[G]{ S^2 \leq \delta} \big / \delta^{\nu/2}} \\ 
&\;\;\;\;\;\;\;\;= \liminf_{\delta \to 0} \cb{\EE[G]{ \p{\sigma^2/\nu}^{-\nu/2} \cdot F_{\nu}(\nu \delta/\sigma^2)  \cdot \p{\nu \delta/\sigma^2}^{-\nu/2}}} \\
&\;\;\;\;\;\;\;\;\stackrel{(*)}{\geq} \EE[G]{ \liminf_{\delta \to 0} \cb{\p{\sigma^2/\nu}^{-\nu/2} \cdot F_{\nu}(\nu \delta/\sigma^2)  \cdot \p{\nu \delta/\sigma^2}^{-\nu/2}}} \\
&\;\;\;\;\;\;\;\; = c(\nu)\nu^{\nu/2} \EE[G]{\sigma^{-\nu}} > 0.
\end{aligned}
$$
In $(*)$ we used Fatou's lemma.
Next, we recall also by Lemma~\ref{lemm:chisquare_cdf} that there exists $C(\nu)$ such that:
$$F_{\nu}(u) \leq C(\nu) u^{\nu/2} \text{ for all } u > 0.$$
Hence, we also get that:
$$
\begin{aligned}
&\limsup_{\delta \to 0} \cb{\EE[G]{ \p{2\Phi(q S/\sigma)-1} \ind( S^2 \leq \delta)}  \big / \delta^{\nu/2}} \\ 
&\;\;\;\;\;\;\;\;\leq \limsup_{\delta \to 0} \cb{\EE[G]{ \p{2\Phi(q \sqrt{\delta}/\sigma)-1} \ind( S^2 \leq \delta)}  \big / \delta^{\nu/2}} \\ 
&\;\;\;\;\;\;\;\;=\limsup_{\delta \to 0} \cb{\EE[G]{ \p{2\Phi(q \sqrt{\delta}/\sigma)-1} F_{\nu}(\nu \delta/\sigma^2)  \big / \delta^{\nu/2}}} \\
&\;\;\;\;\;\;\;\;\lesssim_{\nu} \limsup_{\delta \to 0} \cb{\EE[G]{ \p{2\Phi(q \sqrt{\delta}/\sigma)-1} \sigma^{-\nu}}} \\ 
&\;\;\;\;\;\;\;\; = 0,
\end{aligned}
$$
where in the last step we used dominated convergence. To conclude, we write:
$$ \PP[G]{P^{\text{ttest}}(Z, S^2) > \alpha \cond S^2 \leq \delta} = \frac{\EE[G]{ \p{2\Phi(q S/\sigma)-1} \ind(S^2 \leq \delta)} \big / \delta^{\nu/2} }{ \PP[G]{ S^2 \leq \delta} \big / \delta^{\nu/2}},$$
and take limits as $\delta \to 0$.
\end{proof}

\section{Proofs for Section~\ref{subsec:npmle}}

\subsection{Proof of Proposition~\ref{prop:npmle_opt}}
\label{subsec:prop_npmle_opt}
We will need the following result:
\begin{lemm}[\citet{polya1925aufgaben}, Page 48]
    \label{lemm:polyaszego}
Let $P_1(x),\dotsc,P_n(x)$ be polynomials $\neq 0$ of 
degrees $m_1-1,\dotsc,m_n-1$, where $m_i \geq 1$ for all $i$. Also let $a_1, \dotsc, a_n \in \RR$.
Then, the function
$$g(x) = \sum_{i=1}^n P_i(x) \exp(a_i x),$$
has at most $(\sum_{i=1}^n m_i) - 1$ zeros in $\RR$ (counting multiplicity).
\end{lemm}

\begin{proof}[Proof of Proposition~\ref{prop:npmle_opt} ]
We prove these results when the problem is parametrized in terms of precisions $\tau^2$
as in Section~\ref{sec:precision_parametrization}. The results for this parametrization
will then imply the statement of the proposition via the pushforward
$\tau^2 \mapsto \frac{1}{\tau^2} = \sigma^2$.

Let $N$ be the number of unique values among $\cb{S_1^2,\dotsc,S_n^2}$. 
Without loss of generality we assume that these unique values occur in the order $S_1^2,\dotsc,S_{N}^2$. We 
also let $w_i \geq 1$ be the multiplicity of $S_i^2$ (so that, $\sum_{i=1}^{N} w_i = n$ and $1 \leq N \leq n$.)
A NPMLE with respect to the precision parametrization is defined as follows:
\begin{equation*}
\widehat{H} \in \argmax\cb{ \sum_{i=1}^{N} w_i \log\p{ f_H(S_i^2)}\,:\,H \text{ distribution supported on } (0,\infty)},
\end{equation*}
Let us start by defining $\mathcal{H}$ as the class of all measures with mass $\leq 1$ supported 
on $(0,\, \infty)$ and $\mathcal{H}_1 \subset  \mathcal{H}$ the subclass of distributions. 

We define the map,
$$\psi: \mathcal{H} \to \RR^N,\, \text{ where }\, \psi(H)=(\psi_1(H),\dotsc,\psi_N(H)) \,\text{ with } \, \psi_i(H) = \int p(S_i^2 \mid \tau^2) dH(\tau^2).$$
The image $\psi(\mathcal{H}) \subset \RR^N$ is convex. It is also compact by the Helly Bray selection theorem \citep[Theorem 6.8]{breiman1992probability} and
the fact that $\tau^2 \mapsto p(S_i^2 \mid \tau^2)$ is bounded and continuous (and also converges to $0$
as $\tau \searrow 0$ or $\tau \nearrow +\infty$). Furthermore, there must 
exist unique $\hat{\beta} \in \psi(\mathcal{H})$ such that:
\begin{equation}
    \label{eq:unique_hat_beta}
\hat{\beta} \in \argmax\cb{ h(\beta)\,:\, \beta \in  \psi(\mathcal{H})},\;  h(\beta) \coloneqq \sum_{i=1}^N w_i \log( \beta_i).
\end{equation}
This follows by strict concavity of $h(\beta)$ as well as the fact that $h(\beta) \to -\infty$ as any $\beta_i \to 0$. 
In particular there also must exist $\widehat{H} \in \mathcal{H}$ such that $\beta = \psi(\widehat{H})$.
Our next goal is to prove that $\widehat{H}$ is in fact a probability measure. Suppose otherwise, i.e.,
$\widehat{H}((0, \infty)) =  1 - \eta$ for some $\eta >0$. Let $t$ be any point in the interval $ (0, \infty)$
and consider the measure $\widetilde{H} = \widehat{H} + \eta \delta_{t}$. Then $\widetilde{H} \in \mathcal{H}$ 
and $h(\psi(\widetilde{H})) > h(\psi( \widehat{H}))$ which is a contradiction. Thus $\widehat{H} \in \mathcal{H}_1$.

For any two distributions $H_1, H_2 \in \mathcal{H}_1$ we define the directional derivative from $H_2$ to $H_1$ as:
$$ D(H_1, H_2) = \lim_{\lambda \searrow 0} \frac{ h(\psi( \lambda H_1 + (1-\lambda) H_2 )) - h(\psi(H_2))}{\lambda} = \sum_{i=1}^N w_i \frac{\psi_i(H_1) - \psi_i(H_2)}{\psi_i(H_2)}.$$ 
Since $\widehat{H}$ maximizes the log marginal likelihood, it must hold that $D(H_1, \widehat{H}) \leq 0$ for any $H_1 \in \mathcal{H}_1$.
In particular it must also hold for any Dirac measure $H_1 = \delta_{\tau^2}$ with $\tau^2 \in  (0,\, \infty)$.
Hence for such $\tau^2$ we get by rearranging the inequality $D(\delta_{\tau^2}, \widehat{H}) \leq 0$ that:
\begin{equation}
\label{eq:directional_derivative_inequality}
\xi(\tau^2) \coloneqq \sum_{i=1}^N w_i \frac{p(S_i^2 \mid \tau^2)}{f_{\widehat{H}}(S_i^2)} \leq n \text{ for all } \tau^2 \in  (0,\, \infty),
\end{equation}
where above we introduced the function $\xi(\cdot)$.
The integral of the LHS (i.e., of $\xi(\cdot)$) of the inequality with respect to $\widehat{H}$ is equal to $\sum_{i=1}^N w_i = n$, and so
we get an equality of the integrated LHS and RHS of~\eqref{eq:directional_derivative_inequality} $\text{w.r.t.}$ $\widehat{H}$. This implies that
equality must hold $\widehat{H}$ almost surely, i.e.,
\begin{equation}
    \label{eq:directional_derivative_equality}
\xi(\tau^2) = \sum_{i=1}^N w_i \frac{p(S_i^2 \mid \tau^2)}{f_{\widehat{H}}(S_i^2)} = n\;\;\;\; \widehat{H}\text{-almost surely}.
\end{equation}
Absorbing constants with respect to $\tau^2$ for the $i$-th summand into a constant $c_i$, we may write $\xi(\cdot)$ as:
$$\xi(\tau^2) = \sum_{i=1}^N c_i \p{\tau^2}^{\nu/2} \exp\p{-\frac{\nu \tau^2 S_i^2}{2}}.$$
Next, the derivative of $\xi$ with respect to $\tau^2$ is equal to:
\begin{equation}
    \label{eq:xiprime}
\xi'(\tau^2) = \frac{d \xi(\tau^2)}{d(\tau^2)} =  \sum_{i=1}^N c_i \cdot \frac{\nu}{2} \p{\tau^2}^{\nu/2-1} \exp\p{-\frac{\nu \tau^2 S_i^2}{2}} \p{ 1 - \tau^2 S_i^2}.
\end{equation}
We see that $\xi'(\tau^2) =0$ for $\tau^2 \in (0, \infty)$ is equivalent to:
$$
\zeta(\tau^2):=\sum_{i=1}^N c_i  \exp\p{-\frac{\nu \tau^2 S_i^2}{2}} \p{ 1 - \tau^2 S_i^2} = 0.
$$
But $\zeta$ is of the form considered in Lemma~\ref{lemm:polyaszego} with $m_i = 2$ and thus
can have at most $\sum_{i=1}^N m_i - 1 = 2N-1$ zeros. It follows that the number of zeros of the function $\tau^2 \mapsto \xi(\tau^2) - n$ is at most $2N$ (else $\tau^2 \mapsto \zeta(\tau^2)$ would have at least $2N$ zeros) and so the support of $\widehat{H}$ must be finite by~\eqref{eq:directional_derivative_equality}.

From here we can derive our first conclusion regarding bounds for the support of $\widehat{H}$. First we will prove that $\widehat{H}$ has zero mass in the interval $(0,\min_i \cb{1/S_i^2})$. 
By~\eqref{eq:xiprime} it holds that $\xi'(\tau^2) > 0$ in the interval $(0,\,  \min\cb{1/S_i^2})$,
that is, $\xi(\cdot)$ is strictly increasing. Now suppose there is $\tau^2 \in (0,\,  \min_i\cb{1/S_i^2})$ in the support of $\widehat{H}$. Then $\xi(\tau^2) = n$ by~\eqref{eq:directional_derivative_equality} and
$\xi(\tau^2 + \varepsilon) > n$ for sufficiently small $\varepsilon$, in contradiction to~\eqref{eq:directional_derivative_equality}. 
Thus we conclude that $\widehat{H}$ assigns zero mass to $(0,\min_i \cb{1/S_i^2})$. 
Analogously we can argue that $\widehat{H}$ has zero mass in the interval $(\max_i \cb{1/S_i^2}, +\infty)$ since
on the latter interval it holds that $\xi'(\tau^2) < 0$, i.e., $\xi(\cdot)$ is strictly decreasing.

It remains to accomplish the following tasks: to provide an upper bound on the cardinality of the support of $\widehat{H}$
and to prove the uniqueness of $\widehat{H}$.  For the upper bound on the cardinality, suppose the support of $\widehat{H}$ consists of $S$ points. 
Then for each support point the function $\xi(\tau^2)$ 
has a local maximum due to~\eqref{eq:directional_derivative_inequality} and~\eqref{eq:directional_derivative_equality}.
Furthermore, between any consecutive support points, there must be a local minimum. Hence $\xi'$ 
must have at least $2S-1$ zeros. We also argued earlier that $\xi'$ can have at most $2N-1$ zeros. Hence $2N - 1 \geq 2S-1$, i.e., $S \leq N \leq n$.

It remains to prove the uniqueness of $\widehat{H}$. Suppose there is another maximizer $\widetilde{H}$ 
of the log marginal likelihood. From the beginning of our argument, namely~\eqref{eq:unique_hat_beta},
we then know that it must be the case that 
\begin{equation} 
    \label{eq:marginal_density_equality}
f_{\widetilde{H}}(S_i^2) = f_{\widehat{H}}(S_i^2) = \hat{\beta}_i \text{ for } i =1,\dotsc,N.
\end{equation}
Hence~\eqref{eq:directional_derivative_equality} is identical for both $\widetilde{H}$
and $\widehat{H}$, and it must hold that $\widetilde{H}(\mathcal{S}) = \widehat{H}(\mathcal{S})=1$, 
where $\mathcal{S}$ is the set of roots in~\eqref{eq:directional_derivative_equality}. Thus we can write:
$$ \widehat{H} = \sum_{j=1}^S  \widehat{\pi}_j \delta_{\tau_j^2},\;\;\; \widetilde{H} = \sum_{j=1}^S  \widetilde{\pi}_j \delta_{\tau_j^2},$$
where we indexed the elements of $\mathcal{S}$ as $\cb{\tau_j^2}$. By~\eqref{eq:marginal_density_equality} 
we get:
$$
\begin{aligned}
&\sum_{j=1}^S  \p{ \widehat{\pi}_j - \widetilde{\pi}_j} p(S_i^2 \mid \tau_j^2) =   f_{\widehat{H}}(S_i^2) - f_{\widetilde{H}}(S_i^2)=     0\, \text{ for }\, i =1,\dotsc,N, \\ 
\Longrightarrow \;& \sum_{j=1}^S  \p{ \widehat{\pi}_j - \widetilde{\pi}_j} \p{\tau_j^2}^{\nu/2}  \exp\p{-\frac{\nu \tau_j^2 S_i^2}{2}} = 0\, \text{ for }\, i =1,\dotsc,N.
\end{aligned}
$$
Hence the function $s \mapsto \sum_{j=1}^S  \p{ \widehat{\pi}_j - \widetilde{\pi}_j} \p{\tau_j^2}^{\nu/2} \exp\p{-\frac{\nu \tau_j^2 s}{2}}$ has at least 
$N$ zeros. On the other hand, by Lemma~\ref{lemm:polyaszego} with $m_i=1$, it can have at most $S-1$ zeros; otherwise it must be identically equal to $0$. 
Suppose it is not identical to zero. Then $S-1 \geq N$. But we already proved 
that $S \leq N$, which leads to a contradiction. We conclude that the above function is identically equal to zero, which means
that $\widehat{\pi}_j = \widetilde{\pi}_j$ for all $j$, that is, $\widehat{H}=\widetilde{H}$ and $\widehat{G} = \widetilde{G}$.
\end{proof}

\subsection{Proof of Theorems~\ref{theo:hellinger_dist_convergence} and~\ref{manualtheorem:hellinger_dist_convergence}}
\label{sec:proof_of_theo_hellinger_dist_convergence}
We first define the class of marginal densities of $S^2$ when
$\sigma^2 \in [L,\,U]$ almost surely with $0 < L < U < \infty$:
\begin{equation}
\label{eq:constrained_precision}
\mathcal{F} \equiv \mathcal{F}(L, U, \nu) \coloneqq \cb{f_G(\cdot) = f_G(\cdot;\nu)\,:\, G([L,\, U])=1}. 
\end{equation}
For the proof of the theorem we will use the following lemmata (which we prove later). Before stating the first lemma, we introduce the concept of an $(\Norm{\cdot}_{\infty}, \varepsilon)$-cover. A set $\mathcal{C}$ is said to be an $(\Norm{\cdot}_{\infty}, \varepsilon)$-cover of $\mathcal{F}(L,U, \nu)$ if for every $f \in \mathcal{F}(L,U, \nu)$, there exists some $g \in \mathcal{C}$ such that $\Norm{f - g}_{\infty} \leq \varepsilon$. The cardinality of the minimal $(\Norm{\cdot}_{\infty}, \varepsilon)$-cover of $\mathcal{F}(L,U, \nu)$ is denoted by $N(\varepsilon,\, \mathcal{F}(L,U, \nu),\, \Norm{\cdot}_{\infty})$.

\begin{lemm}
\label{lemm:entropy}
For all $\varepsilon >0$ sufficiently small (that is, with $\varepsilon \leq c \equiv c(\nu, L, U)$), it holds that: 
$$\log\p{N(\varepsilon,\, \mathcal{F}(L,U, \nu),\, \Norm{\cdot}_{\infty})} \lesssim_{\nu, L , U }   \abs{\log(\varepsilon)}^2.$$
\end{lemm}

\begin{lemm}
\label{lemm:chisq_tail_second_moment}
Suppose $\sigma_i^2 \leq U$ almost surely for all $i$. Then, for any $C>0, \gamma >0$, $n \in \mathbb N_{\geq 1}$ and any $B \geq 4U$, it holds that:
$$\PP{   \prod_{i: S_i^2 > B}\frac{C\p{ S_i^2}^2}{\eta B^2} \geq \exp(2\gamma)}  \leq \exp\p{-2\gamma + \frac{n C (B^2 + 8 UB + 32 U^2)}{ \eta B^2} \exp(-B/(4U))  }.$$
\end{lemm}

\begin{proof}[Proof of Theorems~\ref{theo:hellinger_dist_convergence} and~\ref{manualtheorem:hellinger_dist_convergence}]

Recall from Supplement~\ref{sec:general_proof_remarks} that we write $f_*=f_G$, resp. $f_*=f_{\boldsigma}$ depending on whether we are studying the hierarchical or compound setting.
Our goal is to obtain bounds on 
\begin{equation}
    \label{eq:tail_event}
\PP{ A_n}, \text{ where } A_n \coloneqq \cb{\Dhel(f_{\hG}, f_*) \geq \varepsilon_n}.
\end{equation}
We will figure out later 
what the choice of $\varepsilon_n$ should be.

Our strategy is as follows. Consider the following class of marginal densities:
\begin{equation}
    \label{eq:constrained_class_large_hellinger}
\mathcal{F}(\varepsilon_n, \ubar{L}, \bar{U}, \nu) = \cb{f_{\tG}(\cdot) = f_{\tG}(\cdot;\nu)\,:\, \tG([\ubar{L},\, \bar{U}])=1, \Dhel(f_{\tG}, f_*) \geq \varepsilon_n}.
\end{equation}
This is essentially the class of mixture densities in~\eqref{eq:constrained_precision} subject to the additional constraint that their Hellinger
distance to $f_*$ is sufficiently large. Furthermore, let 
$\mathcal{S} = \cb{f_{j}\, : \, j \in \mathcal{J}}$, $\mathcal{J}=\cb{1,\dotsc,J}$, $J = \#\mathcal{S}$,
be a \emph{proper} $(\Norm{\cdot}_{\infty}, \eta)$-cover of~\eqref{eq:constrained_class_large_hellinger}.\footnote{By ``proper cover'',
we mean that the centers of the cover are themselves elements of~\eqref{eq:constrained_class_large_hellinger}.}
Lemma~\ref{lemm:entropy} already provides a cover of a larger class of functions. However,
the latter cover is not a proper cover for~\eqref{eq:constrained_class_large_hellinger}, that is the cover will contain elements
that do not necessarily lie in the class~\eqref{eq:constrained_class_large_hellinger}. However, by a standard argument we can turn an improper
cover to a proper cover at the cost of doubling the radius (i.e., an $\eta/2$-cover in Lemma~\ref{lemm:entropy} yields
a $\eta$-cover for~\eqref{eq:constrained_class_large_hellinger}). Hence we get that:
$$\log J \lesssim_{\nu, \ubar{L} , \bar{U}}   \abs{\log(\eta)}^2.$$
On the event $A_n$ defined in~\eqref{eq:tail_event}, there must exist
$\widehat{j} \in \mathcal{J}$ such that $\Norm{f_{\widehat{j}} - f_{\hG}}_{\infty}\leq \eta$. 
This further implies that on $A_n$:
$$f_{\hG}(s^2) \leq f_{\widehat{j}}(s^2) + \eta \leq \max_{j \in \mathcal{J}} f_j(s^2) + \eta\; \text{ for all } s^2>0.$$
From here we proceed as follows. We introduce the following function parameterized by $B$:
$$\eta(z) \coloneqq \eta\ind\cb{\abs{z}\leq B} + \eta \frac{B^2}{z^2}\ind\cb{\abs{z} > B},\, \text{ for }\, z>0.$$
Notice that by construction
$$ \int_0^{\infty} \eta(z) dz = \eta B +  \eta B = 2\eta B,$$
and also
$$ f_{\hG}(s^2) \leq   f_{\widehat{j}}(s^2) + \eta(s^2) \leq \max_{j \in \mathcal{J}} \cb{ f_j(s^2) +  \eta(s^2)} \text{ if } s^2 \leq B,\;  \text{ and } f_{\hG}(s^2)  \leq C(\nu, \ubar{L}) \text{ otherwise}.$$
The boundedness required for the ``otherwise'' part of the branch above follows by Lemma~\ref{lemm:chi_square_likelihood}A with $C(\nu, \ubar{L})>0$ 
being a constant that only depends on $\nu$ and $\ubar{L}$.
For any $f$, write:
$$L_n(f, f_*) = \prod_{i=1}^n \frac{f(S_i^2)}{f_*(S_i^2)}.$$
Since $f_{\hG}$ is the NPMLE, it must hold that $L_n(f_{\hG}, f_*) \geq 1$.
Next, on the event $A_n$:
$$
\begin{aligned}
L_n(f_{\hG}, f_*) &\leq \prod_{i: S_i^2 \leq B} \frac{ f_{\widehat{j}}(S_i^2) +  \eta(S_i^2)}{f_*(S_i^2)} \prod_{i: S_i^2 > B}  \frac{C(\nu, \ubar{L})}{f_*(S_i^2)} \\
&= \prod_{i=1}^n \frac{ f_{\widehat{j}}(S_i^2) +  \eta(S_i^2)}{f_*(S_i^2)} \prod_{i: S_i^2 > B}  \frac{ C(\nu,\ubar{L})}{ f_{\widehat{j}}(S_i^2) +  \eta(S_i^2)}\\
&\leq \prod_{i=1}^n \frac{ f_{\widehat{j}}(S_i^2) +  \eta(S_i^2)}{f_*(S_i^2)} \prod_{i: S_i^2 > B}  \frac{ C(\nu, \ubar{L})}{\eta(S_i^2)}.
\end{aligned}
$$
Pick any $\gamma$. We get:
$$ 
\begin{aligned}
\PP{\Dhel(f_{\hG}, f_*) \geq \varepsilon_n} 
&= \PP{\Dhel(f_{\hG}, f_*) \geq \varepsilon_n,\;    L_n(f_{\widehat{G}_n}, f_*) \geq 1} \\ 
& \leq   \PP{ \cb{\p{\sup_{j \in \mathcal{J}} \prod_{i=1}^n \frac{ f_j(S_i^2) +  \eta(S_i^2)}{f_*(S_i^2)}} \cdot \p{\prod_{i: S_i^2 > B}  \frac{C(\nu, \ubar{L})}{\eta(S_i^2)}}} \geq 1} \\
& \leq   \PP{ \p{\sup_{j \in \mathcal{J}} \prod_{i=1}^n \frac{ f_j(S_i^2) +  \eta(S_i^2)}{f_*(S_i^2)}}\geq \exp(-2\gamma)}  + \PP{   \prod_{i: S_i^2 > B}  \frac{ C(\nu,\ubar{L})}{\eta(S_i^2)}  \geq \exp(2\gamma)}.
\end{aligned}
$$
In Lemma~\ref{lemm:chisq_tail_second_moment} we already bounded the second probability. Hence it remains to bound the first term. 
Observe that,
$$
\begin{aligned}
&\PP{ \p{\sup_{j \in \mathcal{J}} \prod_{i=1}^n \frac{ f_j(S_i^2) +  \eta(S_i^2)}{f_*(S_i^2)}}\geq \exp(-2\gamma)} \\ 
&\;\;\;\;\;\;\;\; \leq J\sup_{j \in \mathcal{J}} \PP{  \prod_{i=1}^n \frac{ f_j(S_i^2) +  \eta(S_i^2)}{f_*(S_i^2)}   \geq \exp(-2\gamma)} \\ 
&\;\;\;\;\;\;\;\; = J\sup_{j \in \mathcal{J}} \PP{  \prod_{i=1}^n \sqrt{\frac{ f_j(S_i^2) +  \eta(S_i^2)}{f_*(S_i^2)}}   \geq \exp(-\gamma)} \\ 
&\;\;\;\;\;\;\;\; \leq J\sup_{j \in \mathcal{J}}\cb{ \exp(\gamma) \prod_{i=1}^n \EE{\sqrt{\frac{ f_j(S_i^2) +  \eta(S_i^2)}{f_*(S_i^2)}}}} \\ 
&\;\;\;\;\;\;\;\; \leq J\sup_{j \in \mathcal{J}}\cb{ \exp(\gamma) \exp\p{\sum_{i=1}^n \cb{\EE{\sqrt{\frac{ f_j(S_i^2) +  \eta(S_i^2)}{f_*(S_i^2)}}}-1}}} \\ 
&\;\;\;\;\;\;\;\; \stackrel{(*)}{=} J\sup_{j \in \mathcal{J}} \cb{ \exp(\gamma) \exp\p{ n \cb{\int_0^{\infty} \sqrt{f_j(z) + \eta(z)}\sqrt{f_*(z)}dz -1}}}\\
&\;\;\;\;\;\;\;\; \stackrel{(**)}{\leq} J\exp(\gamma)\exp\p{ - n\Dhel^2(f_j, f_*) + n\sqrt{2 \eta B}} \\
&\;\;\;\;\;\;\;\;  \leq \exp\cb{-n\varepsilon_n^2 + n \sqrt{2 \eta B} + \gamma + \log J}.
\end{aligned}
$$
In $(*)$ we have used the following argument, which is slightly different in the hierarchical setting compared to the compound setting. In the hierarchical setting:
\begin{equation*}
\sum_{i=1}^n \EE[G]{\sqrt{\frac{ f_j(S_i^2) +  \eta(S_i^2)}{f_G(S_i^2)}}} = \sum_{i=1}^n \int_0^{\infty} \sqrt{\frac{ f_j(z) +  \eta(z)}{f_G(z)}} f_G(z)dz =  n \int_0^{\infty} \sqrt{f_j(z) + \eta(z)}\sqrt{f_G(z)}dz.
\end{equation*}
In the compound setting:
$$
\begin{aligned}
\sum_{i=1}^n \EE[\sigma_i]{\sqrt{\frac{ f_j(S_i^2) +  \eta(S_i^2)}{f_{\boldsigma}(S_i^2)}}} &= \sum_{i=1}^n \int_0^{\infty} \sqrt{\frac{ f_j(z) +  \eta(z)}{f_{\boldsigma}(z)}} p(z \mid \sigma_i^2)dz \\ 
&= n \int_0^{\infty} \sqrt{\frac{ f_j(z) +  \eta(z)}{f_{\boldsigma}(z)}} \p{\frac{1}{n}\sum_{i=1}^n p(z \mid \sigma_i^2)}dz \\
&=  n \int_0^{\infty} \sqrt{f_j(z) + \eta(z)}\sqrt{f_{\boldsigma}(z)} dz.
\end{aligned}
$$
In $(**)$ we used the following argument:
$$
\begin{aligned}
\int_0^{\infty} \sqrt{f_j(z) + \eta(z)}\sqrt{f_*(z)}dz - 1 &\leq \int_0^{\infty} \p{\sqrt{f_j(z) f_*(z)} +  \sqrt{\eta(z) f_*(z)}} dz - 1 \\ 
&\leq - \Dhel^2(f_j, f_*)  + \p{ \int_0^{\infty }\eta(z)dz}^{1/2}\p{ \int_0^{\infty }f_*(z)dz}^{1/2} \\ 
&= - \Dhel^2(f_j, f_*) + \sqrt{2 \eta B}.
\end{aligned}
$$ 
We are now ready to pick all parameters. First we pick $\gamma = (\log n)^2$ and $\eta = 1/n^2$. Then $\log J \lesssim_{\nu , \ubar{L}, \bar{U}} (\log n)^2$. Also let $B=16\bar{U}(1 \lor \log n)$.
Finally, we pick:
$$ \varepsilon_n^2 = C \frac{(\log n)^2}{n},$$
where we pick $C$ sufficiently large so that $n \varepsilon_n^2$ dominates all other terms that are of order $(\log n)^2$.
\end{proof}

\subsection{Proof of Lemmata~\ref{lemm:entropy} and~\ref{lemm:chisq_tail_second_moment}}

\begin{lemm}
\label{lemm:variance_tail_prob}
Take any $G$ supported on $[0,\,U]$ and suppose $\nu \geq 2$. 
Then, for any $\varepsilon \in (0,\,1/e)$ it holds that:
$$ \PP[G]{ S^2 \geq t(\varepsilon)} \leq \varepsilon, \text{ where } t(\varepsilon) = 4U\abs{\log \varepsilon}.$$
\end{lemm}
\begin{proof}[Proof of Lemma~\ref{lemm:variance_tail_prob}]
Let $t' \geq 1$ and define $t = 4Ut'$. Since $\nu \geq 2$, it follows that
$$\nu+ 2\sqrt{\nu t'} + 2t' \leq (3\nu+2)t' \leq 4\nu t' = (\nu / U)t.$$
Write $S^2 = X_{\nu} \sigma^2 / {\nu}$, where $X_{\nu}$ follows the $\chi^2_{\nu}$-distribution (chi-square distribution with $\nu$ degrees of freedom).
Then, as $\sigma^2 \leq U$ with probability $1$,
$$
\begin{aligned}
\PP[G]{ S^2 \geq t} &= \PP[G]{ \frac{\sigma^2}{\nu} X_{\nu} \geq t} 
\leq \PP[G]{X_{\nu} \geq \frac{\nu}{U}t} 
\leq \PP[G]{X_{\nu} \geq \nu + 2\sqrt{\nu t'} + 2 t'}  
 \stackrel{(*)}{\leq} \exp(-t').
\end{aligned}
$$
The inequality in $(*)$ follows by standard concentration results for $\chi^2$-random variables (e.g.,~\citet[Lemma 1]{laurent2000adaptive}).
Hence, for $\varepsilon \in (0,1/e)$ and $t' = \abs{\log \varepsilon} \geq 1$, i.e., 
for $t(\varepsilon) = 4U\abs{\log \varepsilon}$, it follows 
that $\PP[G]{ S^2 \geq t(\varepsilon)} \leq \varepsilon$. 
\end{proof}

\begin{proof}[Proof of Lemma~\ref{lemm:entropy}]
We will use the precision parametrization (Supplement~\ref{sec:precision_parametrization}) of the priors
for most of the proof. Throughout we fix $H$ supported on $[U^{-1},\, L^{-1}]$.\\

\noindent \textbf{Step 1:} Let $\varepsilon >0$ be given. In subsequent steps we will only seek to control the marginal
densities for $s^2 \in (0,\, B]$. We choose this $B$ in such a way that:
$$f_H(s^2) \lesssim_{\nu, L, U}  \varepsilon \text{ for all } s^2 \geq B.$$ 
By Lemma~\ref{lemm:chi_square_likelihood}C, we have that for $\kappa \geq \nu/4$
and $B = 4(\kappa/\nu)U$, then $p(s^2 \mid \tau^2) \lesssim_{\nu, U} \exp(-\kappa)$.
Hence we can pick $\kappa = \max\cb{\nu/4,\, \log(1/\varepsilon)}$ and 
$$ B = (4/\nu) U \max\cb{\nu/4,\, \log(1/\varepsilon)}.$$

\noindent \textbf{Step 2:} In addition to $H$, take another distribution $H'$
supported on $[U^{-1},\, L^{-1}]$. Then, for $s^2 \leq B$ and any $K \in \mathbb N_{\geq 1}$ it holds that:
$$ 
\begin{aligned}
&f_H(s^2) - f_{H'}(s^2)\\
&\;\;\;= C(\nu) (s^2)^{\nu/2-1} \int \p{\tau^2}^{\nu/2} \exp\p{-\frac{\nu \tau^2 s^2}{2}}\p{dH(\tau^2)-dH'(\tau^2)} \\
&\;\;\;\stackrel{(*)}{=} C(\nu) (s^2)^{\nu/2-1} \int \p{\tau^2}^{\nu/2} \p{R_K(\tau, s) + \sum_{k=0}^K \frac{(-1)^k \nu^k (\tau^2)^{k} (s^2)^{k}}{k! 2^k}}\p{dH(\tau^2)-dH'(\tau^2)} \\
&\;\;\;= C(\nu) (s^2)^{\nu/2-1} \int \p{R_K(\tau, s)\p{\tau^2}^{\nu/2}  \, +\,\sum_{k=0}^K \frac{(-1)^k \nu^k (\tau^2)^{k+ \nu/2} (s^2)^{k}}{k! 2^k}}\p{dH(\tau^2)-dH'(\tau^2)}.
\end{aligned}
$$
In $(*)$ we applied Taylor's theorem. The remainder term $R_K$ therein satisfies:
$$\abs{R_K(\tau, s)} \leq \frac{\p{\nu \tau^2 s^2/2}^{K+1}}{(K+1)!} \leq \p{ \frac{e \nu \tau^2 s^2 /2}{K+1}}^{K+1} \leq \p{ \frac{2\nu L^{-1} B}{K+1}}^{K+1},$$ 
where we used that $(K+1)! \geq ((K+1)/e)^{K+1}$ from Stirling's approximation.
We now build upon a moment matching argument.
Suppose $H'$ matches the following (fractional) moments:
\begin{equation}
    \label{eq:fractional_moment_match}
\int (\tau^2)^{k+ \nu/2} dH(\tau^2) = \int (\tau^2)^{k+ \nu/2} dH'(\tau^2) \text{ for } k=0,\dotsc, K.
\end{equation}
Then:
$$ 
\begin{aligned}
\abs{f_H(s^2) - f_{H'}(s^2)} & = C(\nu) (s^2)^{\nu/2-1} \int R_K(\tau, s)\p{\tau^2}^{\nu/2}\p{dH(\tau^2)-dH'(\tau^2)} \\ 
& \lesssim_{\nu}  (B)^{\nu/2-1}  L^{-\nu/2} \p{ \frac{2\nu L^{-1} B}{K+1}}^{K+1}.
\end{aligned}
$$
Now let $K \geq \max\cb{ 2^{\nu/2} e \nu L^{-1} B,\, \log(1/\varepsilon)}$. Then noting that $K^{\nu/2-1} \leq \p{2^{\nu/2-1}}^{K+1}$, we get:
$$
\begin{aligned}
&(B)^{\nu/2-1}  L^{-\nu/2} \p{ \frac{2\nu L^{-1} B}{K+1}}^{K+1}  \\  
&\;\;\;\;\;\;
 \lesssim_{\nu, L} K^{\nu/2-1} \p{ \frac{2\nu L^{-1} B}{K+1}}^{K+1}  
 \leq  \p{ \frac{2^{\nu/2}\nu L^{-1} B}{K+1}}^{K+1} 
 \leq e^{-(K+1)} \leq \varepsilon.
\end{aligned}
$$

\noindent \textbf{Step 3:} By Caratheodory's theorem, there exists
a discrete distribution $H'$ supported on $[U^{-1}, L^{-1}]$ with at most $K' = K+2$ 
points of support that matches the above moments of $H$.

\noindent \textbf{Step 4:}
In Step 3 we constructed $H'$ of the following form:
\begin{equation}
\label{eq:discrete_measure}
H' = \sum_{j=1}^{K'} \pi_j \delta_{x_j},
\end{equation}
where $U^{-1} \leq x_1 \leq \dotsc \leq x_{K'} \leq L^{-1}$ and $\pi \in \text{Simplex}(K')$, the probability simplex embedded in $\RR^{K'}$.
Now take another measure of the form:
$$
H'' = \sum_{k=1}^{K'} p_j \delta_{y_j},
$$
where $U^{-1} \leq y_1 \leq \dotsc \leq y_{K'} \leq L^{-1}$ and $p \in \text{Simplex}(K')$.
Then:
$$
\begin{aligned}
    \abs{f_{H'}(s^2) - f_{H''}(s^2)} &= \abs{ \sum_{j=1}^{K'} \pi_j p(s^2 \cond x_j) - \sum_{j=1}^{K'} p_j p(s^2 \cond y_j)}  \\
    &=  \abs{\sum_{j=1}^{K'} \pi_j \p{p(s^2 \cond x_j) - p(s^2 \cond y_j)} + \sum_{j=1}^{K'} \p{\pi_j - p_j}  p(s^2 \cond y_j)} \\
    &\lesssim_{\nu, L,U} \max_{j=1,\dotsc,K'} \cb{ \abs{x_j - y_j}} \, + \,  \Norm{\pi - p}_1.
\end{aligned}
$$
In the last step we used Lemma~\ref{lemm:chi_square_likelihood}, parts A. and D.\\

\noindent \textbf{Step 5 (putting everything together):} It only remains to provide an explicit construction of the cover.
To this end, first let $\Xi$ be an $(\varepsilon,\,\abs{\cdot})$-cover of $[U^{-1},\, L^{-1}]$. By picking
an equidistant grid we see that $\abs{\Xi} \lesssim_{L,U} 1/\varepsilon$.

Second, let $\mathcal{S}(\text{Simplex}(K'))$ be an $(\varepsilon,\,\Norm{\cdot}_1)$-cover of $\text{Simplex}(K')$.
A volume argument yields $\abs{\mathcal{S}(\text{Simplex}(K'))} \leq (5/\varepsilon)^{K'}$ for $\varepsilon \leq 1$. 

We then define:
$$\mathcal{H} \coloneqq \cb{ H'' = \sum_{j=1}^{K'} p_j \delta_{y_j} \, : \, p \in\mathcal{S}(\text{Simplex}(K')),\; y_j \in \Xi }.$$
Combining the previous steps and the triangle inequality,
we see that for any $f = f_H \in \mathcal{F}$, there exists $H'' \in \mathcal{H}$ such that
$$\Norm{f_H - f_{H''}}_{\infty} \lesssim_{\nu, L, U} \varepsilon,$$
and so our cover of~\eqref{eq:constrained_precision} consists of all elements of the form
$ f_{H''}(\cdot) = \int p(\cdot \cond \tau^2) dH''(\tau^2)$ with $H'' \in \mathcal{H}$. It only
remains to count $\abs{\mathcal{H}}$. We have that:

$$
\begin{aligned}
\abs{\mathcal{H}}  \leq \binom{\Xi}{K'} \cdot \mathcal{S}(\text{Simplex}(K'))) \leq  \p{\frac{C(L,U)}{\varepsilon}}^{K'}   \cdot \p{\frac{5}{\varepsilon}}^{K'}.
\end{aligned}
$$
Hence:
$$ \log(\abs{\mathcal{H}} ) \lesssim_{L, U} K' \log(1/\varepsilon).$$
The above is a $C(\nu, L, U)\cdot \varepsilon$ cover for a constant $C(\nu, L, U)$, and so:
$$\log\p{N(\varepsilon, \mathcal{F}(L,U, \nu), \Norm{\cdot}_{\infty})} \lesssim_{\nu, L, U} K' \log(C(\nu,L,U)/\varepsilon) \lesssim_{\nu, L, U} K' \log(1/\varepsilon).$$
Observe that our construction allows to choose $K' \lesssim_{\nu, L , U} \log(1/\varepsilon)$ to conclude. 
\end{proof}

\begin{proof}[Proof of Lemma~\ref{lemm:chisq_tail_second_moment}]
$$
\begin{aligned}
&   \PP{   \prod_{i: S_i^2 > B}\frac{C\p{S_i^2}^2}{\eta B^2} \geq \exp(2\gamma)} \\ 
& \;\;\;\;\;\;\; \leq \exp(-2\gamma) \prod_{i=1}^n \EE{  \p{\frac{C\p{S_i^2}^2}{\eta B^2}}^{ \ind (S_i^2 > B)}} \\ 
& \;\;\;\;\;\;\; \leq \exp(-2\gamma) \prod_{i=1}^n \EE{ \p{1+\ind (S_i^2 > B) \frac{C\p{S_i^2}^2}{\eta B^2}}} \\ 
& \;\;\;\;\;\;\; \leq \exp(-2\gamma) \exp\p{ \frac{C}{\eta B^2} \sum_{i=1}^n \EE{ \ind (S_i^2 > B) \p{S_i^2}^2}}.
\end{aligned}
$$
Hence it suffices to bound the inner expectation.  We have that:
$$ 
\begin{aligned}
\EE{ \ind (S_i^2 > B) \p{S_i^2}^2} &= 2\int_0^{\infty} t \PP{S_i^2 \ind (S_i^2 > B) \geq t}\, dt  \\ 
& = 2 \p{\int_0^{B} t \PP{S_i^2 \geq B}\,dt \, + \, \int_B^{\infty} t \PP{S_i^2 \geq t}\, dt} \\  
& \stackrel{(*)}{\leq} B^2 \exp(-B/(4U)) + 2\int_B^{\infty} t \exp(-t/(4U)) dt \\ 
& = \p{B^2+ 8 UB + 32 U^2} \exp(-B/(4U)).
\end{aligned}
$$
The tail bound $(*)$ follows precisely as in the proof of Lemma~\ref{lemm:variance_tail_prob} and we use the fact here that $B \geq 4U$. 
\end{proof}

\subsection{Proof of Corollary~\ref{coro:weak_convergence}}
\label{subsec:proof_coro_weak_convergence}

\begin{proof}
Our proof operates on the following event
$$ A \coloneqq  \bigcup_{k=1}^{\infty} \bigcap_{ n=k}^{\infty} A_n,\, \text{ where }\, A_n \coloneqq  \cb{\Dhel(f_{\hG_n}, f_*) < \frac{C \log n}{\sqrt{n}}},$$
with the constant $C$ in the statement of Theorem~\ref{theo:hellinger_dist_convergence}. 
We make the dependence of $\hG$ on $n$ explicit by writing $\hG_n \equiv \hG$.
By Theorem~\ref{theo:hellinger_dist_convergence}  we also have that $\PP[G]{A_n^c} \leq \exp(-c (\log n)^2)$, where 
$A_n^c$ denotes the complement of $A_n$, and so:
$$ \sum_{n=1}^{\infty} \PP[G]{A_n^c} < \infty.$$
Thus, by the first Borel-Cantelli Lemma (e.g., \citep[Theorem 4.3]{billingsley1995probability}), it follows 
that $\PP[G]{A} = 1$. We will prove that $\hG_n \cd G$ on the event $A$.

First by construction it holds that $(\hG_n)_{n \geq 1}$ is tight (since $\hG_n$ is supported on the compact
set $[\ubar{L}, \bar{U}]$). Next take any subsequence $\hG_{n_k}$ such that $\hG_{n_k} \cd \widetilde{G}$
as $k \to \infty$ for some probability measure $\widetilde{G}$. By~\citet[Chapter 25, corollary on page 337]{billingsley1995probability}, 
the weak convergence $\hG_n \cd G$ will follow if we can show that $\widetilde{G} = G$.

To see this, first note that for any $u>0$:
$$ f_{\hG_{n_k}}(u) = \int p(u \mid \sigma^2) d\hG_{n_k}(\sigma^2) \to \int p(u \mid \sigma^2 ) d\widetilde{G}(\sigma^2) =  f_{\widetilde{G}}(u)\; \text{ as } k \to \infty.$$
Here we used the definition of weak convergence along with the facts that $\sigma^2 \mapsto p(u \mid \sigma^2)$ 
is continuous and bounded on $[\ubar{L}, \bar{U}]$. Hence by Scheffe's theorem~\citep[Theorem 16.12.]{billingsley1995probability}
applied to the Lebesgue measure on $(0, \infty)$, we get convergence in total variation distance ($\TV$),
$$\TV( f_{\hG_{n_k}},\, f_{\widetilde{G}}) \coloneqq \frac{1}{2}\int\abs{ f_{\hG_{n_k}}(u)  - f_{\widetilde{G}}(u)}\,du \, \to \, 0\; \text{ as }\; k \to \infty.$$
On the other hand:
$$ \TV(f_{\hG_{n_k}}, \, f_G ) \stackrel{(*)}{\leq} \sqrt{2} \Dhel(f_{\hG_{n_k}}, \, f_G)   \, \stackrel{(**)}{\to} \, 0\; \text{ as }\; k \to \infty.$$
$(**)$ follows because on the event $A$ it holds that:
$$ \limsup_{k \to \infty} \Dhel(f_{\hG_{n_k}}, \, f_G) \, \leq \, \limsup_{k \to \infty} \frac{C \log n_k}{\sqrt{n_k}} \to 0,\; \text{ as }\; k \to \infty.$$
$(*)$ is a standard result that follows
e.g., by the Cauchy-Schwarz inequality. Hence by the triangle inequality, we find that \smash{$\TV(f_{\widetilde{G}}, \; f_G ) =0$}, which implies 
that $f_G(u) = f_{\widetilde{G}}(u)$ for almost all $u>0$, and so by continuity, for all $u>0$.
Let $H$ and \smash{$\widetilde{H}$} be the push-forwards of $G$ and \smash{$\widetilde{G}$} under the map $\sigma^2 \mapsto 1/\sigma^2$. 
In other words, we are switching to the precision parametrization of Supplement~\ref{sec:precision_parametrization}.  
Then it follows by~\citet[Section 4, Example 1]{teicher1961identifiability} that \smash{$H = \widetilde{H}$}.
This in turn implies that \smash{$G = \widetilde{G}$}, and so we conclude with the proof
of this corollary.
\end{proof}

\section{Proofs for Section~\ref{subsec:tweedie}}
\subsection{Proof of Proposition~\ref{prop:pvalue}}
\label{subsec:proof_prop_pvalue}
\begin{proof}
To simplify notation suppose without loss of generality that $z \geq 0$. Then:
$$
\begin{aligned}
&2\EE{(1-\Phi(z/ \sigma)) \mid S^2=s^2} \\ 
=\;\; &  2\frac{1}{f_G(s^2;\nu)}\int_0^{\infty} (1-\Phi(z/ \sigma))  p( s^2 \cond \sigma^2,\nu) dG(\sigma^2) \\
=\;\; & 2\frac{1}{f_G(s^2;\nu)}\int_0^{\infty} \p{\int_z^{\infty} \frac{1}{\sqrt{2\pi \sigma^2}} \exp\p{- \frac{u^2}{2\sigma^2}}\, du} p( s^2 \mid \sigma^2) dG(\sigma^2) \\
=\;\; & \sqrt{\frac{2}{\pi}}\frac{(\nu/2)^{\nu/2}}{\Gamma(\nu/2)}\frac{\p{s^2}^{\nu/2-1}}{f_G(s^2;\nu) }\int_z^{\infty} \int_0^{\infty}  \p{\sigma^2}^{-1/2} \p{\sigma^2}^{-\nu/2}   \exp\p{-\frac{\nu s^2 + u^2}{2\sigma^2}} \,dG(\sigma^2)\,du\\
=\;\; & \sqrt{\frac{2}{\pi}}\frac{(\nu/2)^{\nu/2}}{\Gamma(\nu/2)}\frac{\p{s^2}^{\nu/2-1}}{f_G(s^2;\nu) }\int_z^{\infty} \int_0^{\infty}  \p{\sigma^2}^{-\frac{\nu+1}{2}}   \exp\p{-\frac{(\nu+1)\cb{\p{\nu s^2 + u^2}/(\nu+1)}}{2\sigma^2}} \,dG(\sigma^2)\,du\\
=\;\; & \sqrt{\frac{2}{\pi}}\frac{(\nu/2)^{\nu/2}}{\Gamma(\nu/2)}\frac{\p{s^2}^{\nu/2-1}}{f_G(s^2;\nu) }\int_z^{\infty}  \Gamma\p{\frac{\nu+1}{2}}\p{\frac{2}{\nu+1}}^{\frac{\nu+1}{2}}\p{\frac{\nu s^2 + u^2}{\nu+1}}^{-\frac{\nu-1}{2}} f_G\p{\frac{\nu s^2 + u^2}{\nu+1}; \nu+1}\,du \\
=\;\; & \sqrt{\frac{2}{\pi}}\frac{\sqrt{2} (1+1/\nu)^{-\nu/2}\Gamma((\nu+1)/2)}{(\nu+1)^{1/2}\Gamma(\nu/2)}\frac{\p{s^2}^{\nu/2-1}}{f_G(s^2;\nu) }\int_z^{\infty} \p{\frac{\nu s^2 + u^2}{\nu+1}}^{1/2-\nu/2} f_G\p{\frac{\nu s^2 + u^2}{\nu+1}; \nu+1}\,du \\ 
=\;\; & \frac{(1+1/\nu)^{-\nu/2}\Gamma((\nu+1)/2)}{\sqrt{\pi}(\nu+1)^{-1/2}\Gamma(\nu/2)}\frac{\p{s^2}^{\nu/2-1}}{f_G(s^2;\nu) }\int_0^{\infty} \frac{ (t^2)^{-\frac{\nu-1}{2}}}{\sqrt{ (\nu+1)t^2 - \nu s^2}} f_G(t^2; \nu+1)\ind\p{t^2 \geq \frac{\nu s^2 + z^2}{\nu+1}} \,d(t^2). 
\end{aligned}
$$
\end{proof}
\subsection{Proof of Lemma~\ref{lemm:closeness_of_marginals_plus_1}}
\label{sec:proof_lemm_closeness_of_marginals_plus_1}
As mentioned in the main text, our proof of Lemma~\ref{lemm:closeness_of_marginals_plus_1} builds upon Mellin transform analysis.
For conciseness we do not provide the necessary background and definitions here and instead refer to~\citet{brennermiguel2021spectral};
our notation in the proof matches the notation therein. We do however provide references for any mathematical property of the Mellin transform
that we invoke.

Our proof also requires the following preliminary lemmata.

\begin{lemm}[Theorem 4.4.3. in~\citet{gabcke1979neue}]
    \label{lemm:incomplete_gamma_tail}
Let $\kappa \geq 1$. Then, for any $T \geq \kappa$ it holds that:
$$ \int_T^{\infty} t^{\kappa - 1}\exp(-t)\,dt \leq \kappa T^{\kappa - 1} \exp(-T).$$
\end{lemm}

\begin{proof}
The proof appears in~\citet{gabcke1979neue}. Since the latter is written in German, we provide the argument here as well.
In the case $\kappa=1$, the above in fact holds with equality (by integrating the exponential function). Henceforth we assume that $\kappa > 1$.
By partial integration:
$$ \int_T^{\infty} t^{\kappa - 1}\exp(-t)\,dt = T^{\kappa-1}\exp(-T) + (\kappa-1) \int_T^{\infty} t^{\kappa - 2} \exp(-t)\,dt.$$
Hence we now look at the second summand of the RHS. Consider the transformation $x = 1/t$, then:
$$\int_T^{\infty} t^{\kappa - 2} \exp(-t)\,dt =  \int_0^{1/T} x^{-\kappa} \exp(-1/x) \,dx \stackrel{(*)}{\leq} \frac{1}{T} T^{\kappa}\exp(-T) = T^{\kappa-1}\exp(-T).$$
Combining with the preceding display yields the result of the lemma. 

It remains to justify $(*)$. Define $h(x) = x^{-\kappa} \exp(-1/x)$. This function has derivative $h'(x) = \exp(-1/x)\p{-\kappa x^{-\kappa - 1} +  x^{-\kappa -2}}$.
Hence $h'(x) > 0$ for all $x \in (0, 1/T)$, that is, $h$ is non-decreasing in the interval $(0, 1/T)$.
\end{proof}

\begin{lemm}[Corollary 1.4.4 in~\citet{andrews1999special}]
    \label{lemm:gamma_bounds}
    Fix $a_1 \leq a_2$. Then, for any $a \in [a_1, a_2]$, it holds that as $\abs{b} \to \infty$,
    $$\abs{\Gamma(a+\complexi b)} = \sqrt{2\pi}\abs{b}^{a-1/2}\exp(-\pi \abs{b}/2)\cb{1+O(1/\abs{b})},$$
    and the constant in $O(\cdot)$ only depends on $a_1$ and $a_2$.
\end{lemm}

\begin{proof}[Proof of Lemma~\ref{lemm:closeness_of_marginals_plus_1}]
Below we carry out a slightly more general argument which allows to bound the $L^2$ distance weighted by
the function $x \mapsto x^{2c-1}$. We fix $c$ for now, and suppose that it satisfies the following inequality:\footnote{The result stated in the lemma
corresponds to the choice $c=1/2$. For $c=1/2$, the inequality~\eqref{eq:c_inequality} is satisfied, since we only consider $\nu \geq 2$.}
\begin{equation}
\label{eq:c_inequality}
c \geq \frac{3}{2} - \frac{\nu}{2} > 1 -  \frac{\nu}{2}.
\end{equation}
For the results
in this slightly more general setting, we also consider the more general form of our assumption on $G,G'$, namely that:
\begin{equation}
    \label{eq:sigma_assumption_moment}
\EE[G]{ \p{\sigma^2}^{c-1} } \leq M, \;\; \EE[G']{ \p{\sigma^2}^{c-1} } \leq M,
\end{equation}
which again reduces to the assumption of the lemma when $c=1/2$.

Let $X_{\nu} \sim \chi^2_{\nu}/\nu$. Then, $S^2$ in~\eqref{eq:EB} is equal
in distribution to $\sigma^2 X_{\nu}$.
We write $p_{\nu}(\cdot)$ for the Lebesgue density of $X_{\nu}$, that is
$$ p_{\nu}(x) = \frac{ \nu^{\nu/2} }{2^{\nu/2} \Gamma(\nu/2)} x^{\nu/2-1} \exp\p{-\frac{\nu x}{2}},\;\text{ for }\, x>0.$$
We first verify that the following integrals are finite:
$$ \int_0^{\infty} x^{c-1}p_{\nu}(x)\,dx \; < \;\infty,\;\;\;\;\int_0^{\infty} x^{2c-1}p^2_{\nu}(x)\,dx \; <\; \infty.$$
For the first integral:
$$ \int_0^{\infty} x^{c-1}p_{\nu}(x)\,dx =  \frac{ \nu^{\nu/2} }{2^{\nu/2} \Gamma(\nu/2)} \int_0^{\infty} x^{(c + \nu/2-1) -1} \exp\p{-\frac{\nu x}{2}}dx =  \p{\frac{2}{\nu}}^{c -1}\frac{\Gamma(c + \nu/2 -1)}{\Gamma(\nu/2)} < \infty.$$
Above we used the fact that $c+\nu/2-1 > 0$ by~\eqref{eq:c_inequality}. For the second integral:
$$
\begin{aligned}
\int_0^{\infty} x^{2c-1}p^2_{\nu}(x)\,dx \lesssim_{\nu} \int_0^{\infty} x^{2c-1}  \p{x^{\nu/2-1}   \exp\p{-\frac{\nu x}{2}}}^2\,dx =  \int_0^{\infty} x^{(2c+\nu-2)-1}   \exp\p{-\nu x}\,dx,
\end{aligned}
$$
hence this integral is also finite since $2c + \nu-2 >0$ by~\eqref{eq:c_inequality}. Using the above result,~\eqref{eq:sigma_assumption_moment} and Lemma 2.2a) in~\citet{butzer1999selfcontained}, we also get that:
$$ 
\int_0^{\infty} x^{2c-1} f_G^2(x; \nu) dx \leq M^2 \int_0^{\infty} x^{2c-1}p^2_{\nu}(x)\,dx < \infty.
$$
Furthermore:
$$ \int_0^{\infty} x^{c-1}f_G(x; \nu)\,dx = \EE[G]{(S^2)^{c-1}} = \EE[G]{\p{\sigma^2}^{c-1} X_{\nu}^{c-1} } = \EE[G]{\p{\sigma^2}^{c-1}} \EE{X_{\nu}^{c-1}} < \infty.$$ 
The above steps establish integrability conditions that simplify the application of the Mellin transform theory (and the bounds hold verbatim with $G'$ replacing $G$). 

The next calculation computes the Mellin transform of $p_{\nu}$ 
for $c$ that satisfies~\eqref{eq:c_inequality} and arbitrary $t \in \RR$ (its definition being shown in the first line below). Define
$$
\begin{aligned}
\MelS{p_{\nu}}{c}(t) &:= \EE[X\sim p_{\nu}]{ X^{c-1+\complexi t}} \\ 
&= \int_0^{\infty} x^{c + \nu/2 -1+\complexi t - 1} \frac{\nu^{\nu/2}}{2^{\nu/2}\Gamma(\nu/2)}   \exp\p{-\frac{\nu x}{2}}\,dx\\ 
&= \int_0^{\infty} \p{\frac{2}{\nu}}^{c -1+\complexi t} u^{c + \nu/2 -1+\complexi t - 1}\frac{1}{\Gamma(\nu/2)}   \exp\p{-u}\,du\\ 
&= \p{\frac{2}{\nu}}^{c -1+\complexi t}\frac{\Gamma(c + \nu/2 -1+\complexi t)}{\Gamma(\nu/2)}.
\end{aligned}
$$
Write $r(t) \coloneqq \MelS{p_{\nu+1}}{c}(t)/ \MelS{p_{\nu}}{c}(t)$, which by the preceding calculation has absolute value equal to
$$
\abs{r(t)} = \p{\frac{\nu}{\nu+1}}^{c-1} \frac{ \Gamma(\nu/2)}{\Gamma((\nu+1)/2)} \frac{\abs{\Gamma(c + (\nu+1)/2 -1+\complexi t)}}{\abs{\Gamma(c + \nu/2 -1+\complexi t)}}.
$$
Hence, using Lemma~\ref{lemm:gamma_bounds}, we see that there exist positive constants $k=k(\nu, c)$ and $K = K(\nu, c)$ that depend only on $\nu$ and $c$ such that:
\begin{equation}
\label{eq:mellin_transform_ratio_bound}
\begin{aligned}
    \abs{r(t)} \leq K(\nu, c) \abs{T}^{1/2}\; \text{ for all } \abs{t} \leq \abs{T},\text{ and all}\, \abs{T} \geq k(\nu, c).
\end{aligned}
\end{equation}
Let us also compute the Mellin transform of $f_G(\cdot; \nu)$:
$$\MelS{f_{G}(\cdot;\nu)}{c}(t) =  \EE[G]{(S^2)^{c-1+\complexi t}} = \EE[G]{\p{\sigma^2}^{c-1+\complexi t} X_{\nu}^{c-1+\complexi t} } = \EE[G]{\p{\sigma^2}^{c-1+\complexi t}} \MelS{p_{\nu}}{c}(t).$$
This computation, \eqref{eq:sigma_assumption_moment}, and Lemma~\ref{lemm:gamma_bounds} furnish the existence of positive constants  $\ell=\ell(\nu, c)$ and $L = L(\nu, c, M)$ such that:
\begin{equation}
\label{eq:mellin_upper_bound}
\abs{ \MelS{f_G(\cdot;\nu+1)}{c}(t)} \leq  L(\nu, c, M) \abs{t}^{c + \nu/2 -1}\exp(-\pi \abs{t}/2) \text{ for all } \abs{t} \geq \ell(\nu, c).
\end{equation}
Before the main argument, we still need the following result that relates the Mellin transform of $f_G(\cdot; \nu)$ and $f_G(\cdot; \nu+1)$ via the ratio $r(\cdot)$
defined above.
$$
\begin{aligned}
\MelS{f_{G}(\cdot;\nu+1)}{c}(t) &= \EE[G]{\p{\sigma^2}^{c-1+\complexi t}}\MelS{p_{\nu+1}}{c}(t)  \\ 
&=\EE[G]{\p{\sigma^2}^{c-1+\complexi t}} \MelS{p_{\nu}}{c}(t) r(t) \\ 
&= \MelS{f_{G}(\cdot;\nu)}{c}(t)\cdot r(t).
\end{aligned}
$$
We are ready to proceed with our main argument. In $(*)$ below 
we apply the Plancherel isometry for the 
Mellin transform, see e.g., Lemma 2.3. in~\citet{butzer1999selfcontained}.
$$
\begin{aligned}
&\int_0^{\infty} x^{2c-1}\p{f_G(x; \nu+1) - f_{G'}(x; \nu+1)}^2\,dx \\ 
\stackrel{(*)}{=}\;\;&\frac{1}{2\pi} \int_{\RR} \abs{ \MelS{f_G(\cdot; \nu+1)}{c}(t) - \MelS{f_{G'}(\cdot; \nu+1)}{c}(t)}^2\,dt \\ 
=\;\;&\frac{1}{2\pi} \int_{\RR} \abs{r(t)}^2 \abs{ \MelS{f_G(\cdot; \nu)}{c}(t) - \MelS{f_{G'}(\cdot; \nu)}{c}(t)}^2\,dt \\
=\;\;&\frac{1}{2\pi} \p{ \int_{\abs{t} \geq T} \abs{r(t)}^2 \abs{ \MelS{f_G(\cdot; \nu)}{c}(t) - \MelS{f_{G'}(\cdot; \nu)}{c}(t)}^2\,dt\,+\, \int_{\abs{t} < T} \dotsc \,dt}\\
=\;\;& \text{I} + \text{II},
\end{aligned}
$$
where $\text{I}$, resp. $\text{II}$ refer to the two integrals and we will pick $T \geq \ell(\nu, c)\lor k(\nu, c) \lor (2c + \nu)$ later. 
We analyze the two integrals in turn. By~\eqref{eq:mellin_upper_bound},
$$
\begin{aligned}
\text{I} &= \frac{1}{2\pi} \int_{\abs{t} \geq T} \abs{ \MelS{f_G(\cdot; \nu+1)}{c}(t) - \MelS{f_{G'}(\cdot; \nu+1)}{c}(t)}^2\,dt\\
&\leq 2 L(\nu,c, M) \int_{\abs{t} \geq T}  \abs{t}^{2c + \nu-2}\exp(-\pi \abs{t})\,dt  \\ 
&\stackrel{(**)}{\lesssim_{\nu, c, M}}  T^{2c + \nu-2 }\exp(-\pi T).
\end{aligned}
$$
For $(**)$ we used Lemma~\ref{lemm:incomplete_gamma_tail} also noting that $2c + \nu -1 \geq 2 \geq 1$ by~\eqref{eq:c_inequality}.
For the second integral, by~\eqref{eq:mellin_transform_ratio_bound},
$$
\begin{aligned}
\text{II} &= \frac{1}{2\pi} \int_{\abs{t} < T} \abs{r(t)}^2 \abs{ \MelS{f_G(\cdot; \nu)}{c}(t) - \MelS{f_{G'}(\cdot; \nu)}{c}(t)}^2\,dt\\
&\leq \frac{1}{2\pi} K(\nu, c)^2 T \int_{\RR} \abs{ \MelS{f_G(\cdot; \nu)}{c}(t) - \MelS{f_{G'}(\cdot; \nu)}{c}(t)}^2\,dt \\ 
&= K(\nu, c)^2 \cdot T \cdot \int_0^{\infty} x^{2c-1}\p{f_G(x; \nu) - f_{G'}(x; \nu)}^2\,dx.
\end{aligned}
$$
In the second line we used~\eqref{eq:mellin_transform_ratio_bound} and for the last step we used the Plancherel isometry for the Mellin transform.
Now let us call:
$$\rho(\nu, c) \coloneqq \int_0^{\infty} x^{2c-1}\p{f_G(x; \nu) - f_{G'}(x; \nu)}^2\,dx.$$
We have proved that:
$$ 
\rho(\nu+1, c) \lesssim_{\nu, c, M} T^{2c + \nu-2}\exp(-\pi T) + T\rho(\nu,c).
$$
The above bound holds for any $c$ that satisfies~\eqref{eq:c_inequality}. We now specialize to the case discussed in the lemma,
namely $c=1/2$. 
Then, the above general inequality reduces to:
$$ 
\rho(\nu+1, 1/2) \lesssim_{\nu, M} T^{\nu-1}\exp(-\pi T) + T\rho(\nu, 1/2).
$$
Using e.g.,~\eqref{eq:basic_gamma_inequality}, we get $T^{\nu-1}\exp(-\pi T) \lesssim_{\nu} \exp(-T)$, so that:
$$ 
\rho(\nu+1, 1/2) \lesssim_{\nu, M} \exp(-T) + T\rho(\nu, 1/2).
$$
Let us pick
$$T = \abs{\log(\rho(\nu, 1/2))} \lor \ell(\nu, 1/2)\lor k(\nu, 1/2) \lor (1 + \nu).$$
Notice that
$$ T \; \lesssim_{\nu} \; 1 + \abs{ \log(\rho(\nu, 1/2))},$$
and also that
$$  \exp(- T) \leq \exp(\log(\rho(\nu, 1/2))) = \rho(\nu, 1/2).$$
Hence,
$$ \rho(\nu+1, 1/2) \lesssim_{\nu, M} (1 + \abs{ \log(\rho(\nu, 1/2))}) \rho(\nu, 1/2).$$
The above is the claim made in the lemma, since by definition:
$$\rho(\nu, 1/2) = \Norm{f_G(\cdot;\nu) -  f_{G'}(\cdot;\nu)}^2_{L^2},\;\; \rho(\nu+1, 1/2) = \Norm{f_G(\cdot;\nu+1) -  f_{G'}(\cdot;\nu+1)}^2_{L^2}.$$
\end{proof}

\subsection{Proof of Theorems~\ref{theo:pvalue_quality} and~\ref{manualtheorem:pvalue_quality}}
\label{subsec:proof_theo_pvalue_quality}

\begin{proof}[Proof of Theorems~\ref{theo:pvalue_quality} and~\ref{manualtheorem:pvalue_quality}]
We start by the following observation. For any distribution $\widetilde{G}$ such that
$\sigma_i^2 \geq \ubar{L}$, it holds that 
$P_{\widetilde{G}}(z, s^2) \geq 2\bar{\Phi}(\abs{z}/\ubar{L}^{1/2})$, where $\bar{\Phi} := 1 - \Phi$ is the survival function of the standard normal distribution. Let $z_{1-\zeta/2}$ be the $(1-\zeta/2)$-quantile of the standard
normal distribution. Then for $\abs{z} \leq \ubar{z} := \ubar{L}^{1/2} z_{1-\zeta/2}$, it holds that:
$$P_{\widetilde{G}}(z, s^2) \geq 2\bar{\Phi}(z_{1-\zeta/2}) = \zeta.$$ 
Henceforth (as mentioned in Supplement~\ref{sec:general_proof_remarks}) 
we write $P_*(z, s^2)$ to denote either $P_G(z, s^2)$ (hierarchical setting) or $P_{\boldsigma}(z, s^2) = P_{G(\boldsigma)}(z, s^2)$ (compound setting).

By the above considerations we conclude that $P_*(z, s^2)\land \zeta - P_{\hG}(Z_i, S_i^2)\land \zeta =\zeta-\zeta= 0$ for 
$\abs{z} \leq \ubar{L}^{1/2} z_{1-\zeta/2}$. On the other hand, for all $i$, it holds that
$$\abs{ P_*(Z_i, S_i^2)\land \zeta - P_{\hG}(Z_i, S_i^2)\land \zeta} \leq \abs{P_*(Z_i, S_i^2) - P_{\hG}(Z_i, S_i^2)}.$$
This holds since the projection of a number in $[0,\,1]$ to the set $[0,\,\zeta]$ is a contraction.
Combining the above two observations and letting $\tilde{Z}_i \coloneqq  \text{sign}(Z_i) \cdot \p{\abs{Z_i} \lor \ubar{z}}$, 
we find the following: 
\begin{equation}
    \label{eq:truncated_to_sup}
\begin{aligned}
\abs{ P_*(Z_i, S_i^2)\land \zeta - P_{\hG}(Z_i, S_i^2)\land \zeta} &\leq \abs{P_*(\tilde{Z}_i, S_i^2) - P_{\hG}(\tilde{Z}_i, S_i^2)} \\ 
&\leq \sup_{z: \abs{z} \geq \ubar{z}}\abs{P_*(z, S_i^2) - P_{\hG}(z, S_i^2)}.
\end{aligned}
\end{equation}
We now need one argument that is slightly different for Theorem~\ref{theo:pvalue_quality} versus Theorem~\ref{manualtheorem:pvalue_quality}. 
To bound the LHS in the inequality stated in Theorem~\ref{manualtheorem:pvalue_quality} we build on the inequality
$$ 
\frac{1}{n}\sum_{i=1}^n \EE[\boldsigma]{\abs{ P_*(Z_i, S_i^2)\land \zeta - P_{\hG}(Z_i, S_i^2)\land \zeta}} \leq \frac{1}{n} \sum_{i=1}^n \EE[\boldsigma]{\sup_{z: \abs{z} \geq \ubar{z}}\abs{P_*(z, S_i^2) - P_{\hG}(z, S_i^2)}},
$$
which follows by averaging~\eqref{eq:truncated_to_sup} over $i=1,\dotsc,n$. To bound the LHS of the inequality 
in the hierarchical version of Theorem~\ref{theo:pvalue_quality}, let us first fix $j \in \cb{1, \dotsc, n}$. We have that:
$$ 
\begin{aligned}
\EE[G]{\abs{ P_*(Z_j, S_j^2)\land \zeta - P_{\hG}(Z_j, S_j^2)\land \zeta}} &\leq \EE[G]{\sup_{z: \abs{z} \geq \ubar{z}}\abs{P_*(z, S_j^2) - P_{\hG}(z, S_j^2)}} \\ 
&=  \frac{1}{n} \sum_{i=1}^n \EE[G]{\sup_{z: \abs{z} \geq \ubar{z}}\abs{P_*(z, S_i^2) - P_{\hG}(z, S_i^2)}}.
\end{aligned}
$$
For the last equality we used the fact that under~\eqref{eq:full_hierarchical}, $(S_1^2, \dotsc, S_n^2)$ are all exchangeable. Thus:
$$\max_{i=1,\dotsc,n} \EE[G]{\abs{ P_*(Z_i, S_i^2)\land \zeta - P_{\hG}(Z_i, S_i^2)\land \zeta}} \leq  \frac{1}{n} \sum_{i=1}^n \EE[G]{\sup_{z: \abs{z} \geq \ubar{z}}\abs{P_*(z, S_i^2) - P_{\hG}(z, S_i^2)}}.$$
In either case, the preceding arguments establish that it suffices to bound
$$\frac{1}{n} \sum_{i=1}^n \EE{\sup_{z: \abs{z} \geq \ubar{z}}\abs{P_*(z, S_i^2) - P_{\hG}(z, S_i^2)}},$$
and we devote the rest of this proof to this task. Let $A$ be the event:
\begin{equation}
\label{eq:event_hellinger_distance_close_to_0}
A \coloneqq  \cb{\Dhel(f_{\hG}, f_*) < C \log n /\sqrt{n}},
\end{equation}
with the constant $C$ in the statement of Theorem~\ref{theo:hellinger_dist_convergence}.
We have that:
$$  \EE{\sup_{z: \abs{z} \geq \ubar{z}}\abs{P_*(z, S_i^2) - P_{\hG}(z, S_i^2)}} \leq \PP{A^c} + \EE{\sup_{z: \abs{z} \geq \ubar{z}}\abs{P_*(z, S_i^2) - P_{\hG}(z, S_i^2)}\, \ind(A)}
$$ 
Above we used the fact that $P_*(z, S_i^2),\, P_{\hG}(z, S_i^2) \in [0,\,1]$.

Next, for any distribution $\widetilde{G}$ supported on $(0, \infty)$, and any $z \in \RR$, $s^2 > 0$, let us write:
$$
\begin{aligned}
    & N_{\widetilde{G}}(z, s^2)  &&\coloneqq P_{\widetilde{G}}(z, s^2) \cdot f_{\widetilde{G}}(s^2), \\ 
    & N_i(z, \widetilde{G}) &&\coloneqq N_{\widetilde{G}}(z, S_i^2),\, \text{ and }\\
    & D_i(\widetilde{G}) &&\coloneqq f_{\widetilde{G}}(S_i^2).
\end{aligned}
$$
By these definitions it holds that $P_{\widetilde{G}}(z, S_i^2) = N_i(z, \widetilde{G})/D_i(\widetilde{G})$.
Throughout the rest of the proof (and as announced in Supplement~\ref{sec:general_proof_remarks}), we also write $G_*$ for $G$ in the hierarchical setting or for $G(\boldsigma)$ in the compound setting. We also let 
$\hG_* = (\hG + G_*)/2$.
Then:
$$
\begin{aligned}
&\abs{P_*(z, S_i^2) - P_{\hG}(z, S_i^2)} \\ 
&\;\;\;\; = \abs{\frac{N_i(z, G_*)}{D_i(G_*)} - \frac{N_i(z, \hG)}{D_i(\hG)}} \\ 
&\;\;\;\;= \abs{\frac{N_i(z, G_*)}{D_i(G_*)} - \frac{N_i(z, G_*)}{D_i(\hG_*)}  +  \frac{N_i(z, G_*)}{D_i(\hG_*)}-  \frac{N_i(z, \hG)}{D_i(\hG_*)}+  \frac{N_i(z, \hG)}{D_i(\hG_*)} -  \frac{N_i(z, \hG)}{D_i(\hG)}} \\ 
&\;\;\;\;\leq \frac{N_i(z, G_*)}{D_i(G_*)} \frac{\abs{D_i(G_*) - D_i(\hG_*)}}{ D_i(\hG_*)} + \frac{\abs{N_i(z, G_*) - N_i(z, \hG)}}{D_i(\hG_*)} + \frac{N_i(z, \hG)}{D_i(\hG)} \frac{\abs{D_i(\hG_*) - D_i(\hG)}}{ D_i(\hG_*)} \\ 
&\;\;\;\;\leq \frac{\abs{N_i(z, G_*) - N_i(z, \hG)}}{D_i(\hG_*)} + \frac{\abs{D_i(G_*) - D_i(\hG)}}{ D_i(\hG_*)}.
\end{aligned}
$$
In the last step we used two facts: first, it holds that $N_i(z_i, \tG)/D_i(\tG) \in [0,\,1]$ for all $\tG$ (since they correspond to conditional p-values), and 
second, the map $\tG \mapsto D_i(\tG)$ is linear, which implies that:
$$ D_i(G_*) - D_i(\hG_*) = (D_i(G_*) - D_i(\hG))/2,\;\;   D_i(\hG_*) - D_i(\hG) = (D_i(G_*) - D_i(\hG))/2.$$
Combining all of the above results, we get:
$$
\begin{aligned}
& \frac{1}{n} \sum_{i=1}^n \EE{\sup_{z: \abs{z} \geq \ubar{z}}\abs{P_*(z, S_i^2) - P_{\hG}(z, S_i^2)}} \\ 
&\;\;\;\; \leq  \PP{A^c} \, + \, \frac{1}{n}\sum_{i=1}^{n} \EE{\sup_{z: \abs{z} \geq \ubar{z}}\abs{\frac{N_i(z, G_*) - N_i(z,\hG)}{D_i(\hG_*)}
} \ind(A)} \, + \frac{1}{n}\sum_{i=1}^{n} \EE{\abs{\frac{D_i(G_*) - D_i(\hG)}{D_i(\hG_*)}} \ind(A)} \\ 
&\;\;\;\; = \PP{A^c} \, + \, \text{I} \, + \, \text{II}.
\end{aligned}
$$
By Theorem~\ref{theo:hellinger_dist_convergence}, $\PP[G]{A^c}  \leq \exp(-c (\log n)^2)$. 
Furthermore, in two upcoming Lemmata~\ref{lemm:numerator_control} and~\ref{lemm:denominator_control} we bound terms $\text{I}$
and $\text{II}$ above and so conclude. 
\end{proof}

\begin{lemm}
    \label{lemm:numerator_control}
It holds that:
$$\frac{1}{n} \sum_{i=1}^n \EE{\sup_{z: \abs{z} \geq \ubar{z}}\abs{\frac{N_i(z, G_*) - N_i(z, \hG)}{D_i(\hG_*)}} \ind(A)}  \lesssim_{\ubar{L}, \bar{U}, \nu, \zeta}  \frac{(\log n)^{5/2}}{\sqrt{n}}.$$
\end{lemm}

\begin{proof}
Below we will prove that there exists non-random $\zeta_n > 0$ such that for all $i$:
$$\sup_{z: \abs{z} \geq \ubar{z}}\abs{N_i(z, G_*) - N_i(z,\hG)} \ind(A) \leq \zeta_n.$$
Let $B_n > 0$ be another non-random sequence that we will pick later. Also
recall that $D_i(\hG_*) = D_i(G_*)/2 + D_i(\hG)/2 \geq D_i(G_*)/2 = f_*(S_i^2)/2$. We get:
$$ \abs{\frac{N_i(z,G_*) - N_i(z,\hG)}{D_i(\hG_*)}} \ind(A) \leq 4 \ind(S_i^2 > B_n) \,+\, 2\zeta_n \frac{1}{f_*(S_i^2)}\ind(S_i^2 \leq B_n),$$ 
where we also used the fact that $N_i(z,\tG)/D_i(\tG) \in [0,\,1]$ for any distribution $\tG$. 

Taking $B_n = 4\bar{U}\log n$ for $n \geq 3$, we get via Lemma~\ref{lemm:variance_tail_prob} that $\PP{S_i^2 > B_n} \leq 1/n$. 
Hence:
$$ 
\frac{1}{n}\sum_{i=1}^n \EE{\sup_{z: \abs{z} \geq \ubar{z}} \abs{\frac{N_i(z,G_*) - N_i(z,\hG)}{D_i(\hG_*)}} \ind(A)}  \leq \frac{4}{n} + 2\zeta_n \frac{1}{n}\sum_{i=1}^n \EE{ \frac{\ind(S_i^2 \leq B_n) }{f_*(S_i^2)}}  \stackrel{(*)}{\leq} \frac{4}{n} + 2 \zeta_n  B_n. 
$$
$(*)$ requires a slightly different argument in the hierarchical vs. compound setting. In the hierarchical setting it suffices to note that:
\begin{equation} 
\label{eq:integral_of_1_div_f}
\EE[G]{ \frac{\ind(S_i^2 \leq B_n)}{f_G(S_i^2)}} = \int_{0}^{B_n} \frac{1}{f_G(t)}f_G(t)\,dt = B_n.
\end{equation}
In the compound setting:
$$
\begin{aligned}
\frac{1}{n} \sum_{i=1}^n \EE[\sigma_i]{ \frac{\ind(S_i^2 \leq B_n)}{f_{\boldsigma}(S_i^2)}} &=  \frac{1}{n} \sum_{i=1}^n  \int_{0}^{B_n} \frac{1}{f_{\boldsigma}(t)}p(t \cond \sigma_i^2)  \,dt \\ 
&=\int_{0}^{B_n} \frac{1}{f_{\boldsigma}(t)}\p{ \frac{1}{n} \sum_{i=1}^n p(t \cond \sigma_i^2) }\, dt \\ 
&=\int_{0}^{B_n} \frac{1}{f_{\boldsigma}(t)} f_{\boldsigma}(t) \, dt = B_n.
\end{aligned}
$$
It remains to determine $\zeta_n$. Recall that we had defined (in the proof of Theorem~\ref{theo:pvalue_quality}) 
$ N_i(z, \widetilde{G}) = N_{\widetilde{G}}(z, S_i^2)$ for any distribution $\widetilde{G}$. 
By Proposition~\ref{prop:pvalue}, for any $z \geq \ubar{z}$ (we take $z$ to be $>0$ without loss of generality) and $s^2>0$, it holds that:
$$
\begin{aligned}
&\abs{N_{G_*}(z, s^2) - N_{\hG}(z, s^2)}/ \cb{C(\nu)\p{s^2}^{\nu/2-1} } \\ 
&\;\;\;\; =  \abs{\int_0^{\infty} \ind\p{t^2 \geq \frac{\nu s^2 + z^2}{\nu+1}} \frac{ (t^2)^{-(\nu-1)/2}}{\sqrt{ (\nu+1)t^2 - \nu s^2}}\; \p{ f_{G_*}(t^2; \nu+1) - f_{\hG}(t^2; \nu+1)}\,d(t^2)} \\ 
&\;\;\;\; \leq \cb{\int_{ \frac{\nu s^2 + z^2}{\nu+1}}^{\infty}  \frac{ (t^2)^{-(\nu-1)}}{(\nu+1)t^2 - \nu s^2}d(t^2)}^{1/2} \cb{ \int_0^{\infty}\p{f_{G_*}(t^2; \nu+1) - f_{\hG}(t^2; \nu+1)}^2 \,d(t^2)}^{1/2}.
\end{aligned}
$$
We bound the two terms in the product in turn. We bound the first term separately for $\nu \geq 3$ and $\nu=2$. For $\nu\geq 3$:
$$
\begin{aligned}
\int_{ \frac{\nu s^2 + z^2}{\nu+1}}^{\infty}  \frac{ (t^2)^{-(\nu-1)}}{(\nu+1)t^2 - \nu s^2}d(t^2)  \leq \frac{1}{(\nu-2) z^2} \sqb{ -\p{t^2}^{-\nu+2}}^{\infty}_{\frac{\nu s^2 + z^2}{\nu+1}}  
= \frac{(\nu+1)^{\nu-2}}{(\nu-2)} \frac{1}{z^2 \p{ \nu s^2 + z^2}^{\nu-2}}.
\end{aligned}
$$
For $\nu=2$:
$$
\begin{aligned}
\int_{ \frac{2 s^2 + z^2}{3}}^{\infty}  \frac{ (t^2)^{-1}}{3t^2 - 2s^2}d(t^2) &= \frac{1}{2s^2}\sqb{ \log(3t^2 - 2s^2) - \log(t^2) }^{\infty}_{\frac{2 s^2 + z^2}{3}}  = \frac{1}{2s^2} \log\p{\frac{2s^2}{z^2}+1} \leq \frac{1}{z^2}.
\end{aligned}
$$
Hence for any $\nu \geq 2$:
$$
\int_{ \frac{\nu s^2 + z^2}{\nu+1}}^{\infty}  \frac{ (t^2)^{-(\nu-1)}}{(\nu+1)t^2 - \nu s^2}d(t^2)  \lesssim_{ \nu}  \frac{1}{z^2 \p{ \nu s^2 + z^2}^{\nu-2}}.
$$
For the second term, it will be helpful to first note the following:
$$
\begin{aligned}
&\int_0^{\infty}\p{f_{G_*}(u; \nu) - f_{\hG}(u;\nu)}^2 \,d(u) \\
&\;\;\;\; = \int_0^{\infty}\p{\sqrt{f_{G_*}(u;\nu)} + \sqrt{f_{\hG}(u;\nu)}}^2 \p{\sqrt{f_{G_*}(u;\nu)} - \sqrt{f_{\hG}(u;\nu)}}^2 \,d(u)  \\ 
& \;\;\;\; \lesssim_{\ubar{L}, \nu} \frac{1}{2}\int_0^{\infty} \p{\sqrt{f_{G_*}(u;\nu)} - \sqrt{f_{\hG}(u;\nu)}}^2 \,d(u) \\ 
& \;\;\;\; =  \Dhel^2(f_{G_*}(\cdot;\nu), f_{\hG}(\cdot;\nu)).
\end{aligned}
$$
For the step with the inequality $\lesssim_{\ubar{L}, \nu}$ we used Lemma~\ref{lemm:chi_square_likelihood}A.
Next we will use Lemma~\ref{lemm:closeness_of_marginals_plus_1}, which in combination with the above display yields that
$$
\begin{aligned}
\int_0^{\infty}\p{f_{G_*}(t^2; \nu+1) - f_{\hG}(t^2; \nu+1)}^2 \,d(t^2) \lesssim_{\ubar{L}, \bar{U}, \nu} \abs{\log(\Dhel^2(f_{G_*}, f_{\hG}))}\Dhel^2(f_{G_*}, f_{\hG}),
\end{aligned}
$$
where $f_{G_*}(\cdot) = f_G(\cdot; \nu)$ and $f_{\hG}(\cdot) = f_{\hG}(\cdot; \nu)$. On the event $A$ (see~\eqref{eq:event_hellinger_distance_close_to_0}), the RHS above is $\lesssim_{\ubar{L}, \bar{U}, \nu} (\log n)^3/n$. Combining the above results, we find that on the event $A$
and for any $z$ such that $\abs{z} \geq \ubar{z} > 0$:
$$\abs{N_{G_*}(z, s^2) - N_{\hG}(z, s^2)} \lesssim_{\ubar{L}, \bar{U}, \nu} \frac{\p{s^2}^{\nu/2-1}}{z \p{ \nu s^2 + z^2}^{\nu/2-1}} \cdot \frac{(\log n)^{3/2}}{\sqrt{n}}  \lesssim_{\ubar{z}, \ubar{L}, \bar{U}, \nu} \frac{(\log n)^{3/2}}{\sqrt{n}}.$$
Thus we may take:
$$ \zeta_n  \lesssim_{\ubar{L}, \bar{U}, \nu}  \frac{(\log n)^{3/2}}{\sqrt{n}}.$$
\end{proof}

\begin{lemm}
    \label{lemm:denominator_control}
$$ \frac{1}{n}\sum_{i=1}^{n} \EE{\abs{\frac{D_i(G_*) - D_i(\hG)}{D_i(\hG_*)}}\ind(A)} \lesssim_{\nu, \ubar{L}, \bar{U}} \frac{\log n}{\sqrt{n}}.$$
\end{lemm}
\begin{proof}
Let $\eta = 1/n$.
The starting point of the proof will be similar to the start of the proof of Theorem~\ref{theo:hellinger_dist_convergence} and will
consist of picking a suitable proper $(\eta, \Norm{\cdot}_{\infty})$-cover of the following class:
\begin{equation}
        \label{eq:constrained_class_small_hellinger}
\mathcal{F}^-(\ubar{L}, \bar{U}, \nu) \coloneqq \cb{f_{\tG}(\cdot) = f_{\tG}(\cdot;\nu)\,:\, \tG([\ubar{L},\, \bar{U}])=1,\; \Dhel(f_{\tG}, f_*) < \frac{C \log n}{\sqrt{n}}}.
\end{equation}
The constant $C>0$ is as in Theorem~\ref{theo:hellinger_dist_convergence}. This definition is similar to the one 
in~\eqref{eq:constrained_class_large_hellinger} except we now only consider densities close to $f_*$ (rather than far). 

Let us call the $\eta$-cover $\mathcal{S} = \cb{f_{j}\, : \, j \in \mathcal{J}}$, $\mathcal{J}=\cb{1,\dotsc,J}$, $J = \#\mathcal{S}$.
Furthermore, since this is a proper cover, for each $j$ there exists a distribution $G_j$ such that $f_j = f_{G_j}$.
With an analogous argument as in the beginning of the proof of Theorem~\ref{theo:hellinger_dist_convergence}, we have that:
$$\log J \lesssim_{\nu, \ubar{L} , \bar{U}}   \abs{\log \eta }^2 = (\log n)^2.$$
We are ready to start with the main argument of the lemma.
\begin{align}
&\frac{1}{n}\sum_{i=1}^{n} \EE{\abs{\frac{D_i(G_*) - D_i(\hG)}{D_i(\hG_*)}}\ind(A)} \nonumber \\ 
&\;\;\;\; =  \frac{2}{n}\sum_{i=1}^{n} \EE{\abs{\frac{D_i(G_*) - D_i(\hG)}{D_i(G_*) +D_i(\hG)}}\ind(A)} \nonumber \\ 
&\;\;\;\; \stackrel{(*)}{\leq}  \frac{2}{n} \,+\, 8\eta \bar{U} \log n \,+\,   \EE{\sup_{j \in \mathcal{J}}\cb{\frac{2}{n}\sum_{i=1}^{n} \abs{\frac{D_i(G_*) - D_i(G_j)}{D_i(G_*) +D_i(G_j)}}}}  \nonumber \\ 
&\;\;\;\; \leq   \frac{2}{n} \,+\, 8\eta \bar{U} \log n \,+\,   \EE{\sup_{j \in \mathcal{J}}\cb{\frac{2}{n}\sum_{i=1}^{n} \p{\abs{\frac{D_i(G_*) - D_i(G_j)}{D_i(G_*) +D_i(G_j)}}-\EE{\abs{\frac{D_i(G_*) - D_i(G_j)}{D_i(G_*) +D_i(G_j)}}}}}}  \nonumber\\ 
&\;\;\;\;\;\;\;\;\;\;\;\;\;\;\;\;\;\; +\,  \sup_{j \in \mathcal{J}}\cb{\EE{\frac{2}{n}\sum_{i=1}^{n} \abs{\frac{D_i(G_*) - D_i(G_j)}{D_i(G_*) +D_i(G_j)}}}} \nonumber\\ 
&\;\;\;\; =   \frac{2}{n} \,+\, 8\eta \bar{U} \log n \,+\, \text{I} \,+\, \text{II}.  \label{eq:marginal_density_relative_difference_inequality}
\end{align}
We first justify $(*)$. On the event $A$ (see~\eqref{eq:event_hellinger_distance_close_to_0}), there must exist a (random) index $\widehat{j}$ 
such that \smash{$\NormInline{f_{G_{\widehat{j}}} - f_{\hG}}_{\infty} \leq \eta$}.  Let $B_n \coloneqq 4\bar{U}\log n$, so that
$\PP{S_i^2 \geq B_n} \leq 1/n$ ($n \geq 3$, Lemma~\ref{lemm:variance_tail_prob}). Then:
$$
\begin{aligned}
&\frac{1}{n}\sum_{i=1}^{n} \EE{\abs{\frac{D_i(G_*) - D_i(\hG)}{D_i(G_*) +D_i(\hG)}}\ind(A)} \\  
&\;\;\;\;\;\;\leq \frac{1}{n}\sum_{i=1}^{n} \cb{\EE{\abs{\frac{D_i(G_*) - D_i(\hG)}{D_i(G_*) +D_i(\hG)}}\ind(A)\ind(S_i^2 \leq B_n)} \,+\, \PP{S_i^2 > B_n }} \\ 
&\;\;\;\;\;\; \leq \frac{1}{n} \,+\, \frac{1}{n}\sum_{i=1}^{n} \EE{\abs{\frac{D_i(G_*) - D_i(G_{\widehat{j}})}{D_i(G_*) +D_i(G_{\widehat{j}})}}\ind(A)\ind(S_i^2 \leq B_n)} \\ 
&\;\;\;\;\;\;\;\;\;\;\;\;  + \; \frac{1}{n}\sum_{i=1}^n \EE{\abs{\frac{D_i(G_*) - D_i(\hG)}{D_i(G_*) +D_i(\hG)} - \frac{D_i(G_*) - D_i(G_{\widehat{j}})}{D_i(G_*) +D_i(G_{\widehat{j}})}}\ind(A)\ind(S_i^2 \leq B_n)}.
\end{aligned}
$$
We may bound the second summand in the last display by noting that $\widehat{j} \in \mathcal{J}$ on the event $A$, and then taking the supremum over all $j \in \mathcal{J}$. 
On the other hand, we can show that the last summand is at most $2\eta B_n$:
$$
\begin{aligned}
&\frac{1}{n}\sum_{i=1}^n \EE{\abs{\frac{D_i(G_*) - D_i(\hG)}{D_i(G_*) +D_i(\hG)} - \frac{D_i(G_*) - D_i(G_{\widehat{j}})}{D_i(G_*) +D_i(G_{\widehat{j}})}}\ind(A)\ind(S_i^2 \leq B_n) } \\ 
&\;\;\;\;\;\; \leq \frac{1}{n}\sum_{i=1}^n \EE{\abs{\frac{2 D_i(G_*) (D_i(G_{\widehat{j}}) - D_i(\hG))}{(D_i(G_*) +D_i(G_{\widehat{j}}))(D_i(G_*) +D_i(\hG))}}\ind(A)\ind(S_i^2 \leq B_n)} \\ 
&\;\;\;\;\;\; \leq 2\eta \frac{1}{n}\sum_{i=1}^n \EE{\frac{\ind(S_i^2 \leq B_n)}{D_i(G_*)}} \\ 
&\;\;\;\;\;\; = 2\eta B_n.\\ 
\end{aligned}
$$
The argument for the penultimate line is identical to the argument in equation~\eqref{eq:integral_of_1_div_f} (as well as just below that equation 
for the compound setting).

We still need to bound $\text{I}$ and $\text{II}$ in~\eqref{eq:marginal_density_relative_difference_inequality}. We start with $\text{II}$. Fix any $j \in \mathcal{J}$. Then,
$$
\begin{aligned}
\EE{\frac{2}{n}\sum_{i=1}^{n} \abs{\frac{D_i(G_*) - D_i(G_j)}{D_i(G_*) +D_i(G_j)}}} & \leq \frac{2}{n}\sum_{i=1}^n \EE{\p{\frac{D_i(G_*) - D_i(G_j)}{D_i(G_*) +D_i(G_j)}}^2}^{1/2} \\ 
& \leq 2 \cb{\frac{1}{n}\sum_{i=1}^n \EE{\p{\frac{D_i(G_*) - D_i(G_j)}{D_i(G_*) +D_i(G_j)}}^2}}^{1/2}.
\end{aligned}
$$
For both inequality steps above we used Jensen's inequality. Continuing:
$$
\begin{aligned}
&\frac{1}{n}\sum_{i=1}^n \EE{\p{\frac{D_i(G_*) - D_i(G_j)}{D_i(G_*) +D_i(G_j)}}^2} \\ 
&\;\;\;\; = \frac{1}{n}\sum_{i=1}^n \EE{\p{\frac{f_*(S_i^2) - f_{G_j}(S_i^2)}{f_*(S_i^2) + f_{G_j}(S_i^2)}}^2} \\ 
&\;\;\;\; = \int \p{\frac{f_{*}(u) - f_{G_j}(u)}{f_*(u) + f_{G_j}(u)}}^2f_*(u) \, du \\ 
&\;\;\;\; = \int \p{ \sqrt{f_*(u)} - \sqrt{f_{G_j}(u)}}^2 \frac{ \p{\sqrt{f_*(u)} + \sqrt{f_{G_j}(u)}}^2}{f_*(u) + f_{G_j}(u)   } \frac{f_*(u)}{f_*(u) + f_{G_j}(u)} \, du \\ 
&\;\;\;\; \leq 2 \int \p{ \sqrt{f_*(u)} - \sqrt{f_{G_j}(u)}}^2  du \\ 
&\;\;\;\; = 4 \Dhel^2(f_*, \, f_{G_j}) \; \leq \; 4 C \frac{(\log n)^2}{n}.
\end{aligned}
$$
At last, let us turn our attention to $\text{I}$. We introduce the notation
$$ U_{ij} = \abs{\frac{D_i(G_*) - D_i(G_j)}{D_i(G_*) +D_i(G_j)}},\, \text{ for }\, i \in \cb{1,\dotsc,n},\, j \in \mathcal{J},$$
and note that for any $j$, $U_{ij}$ is a function of $S_i^2$ only. Furthermore, $U_{ij}$ is bounded and takes values in the interval $[0,1]$. Hence,
by Hoeffding's inequality (for fixed $j$) and any $t \geq 0$:
$$ 
\PP{ \abs{\frac{1}{n} \sum_{i=1}^n \p{U_{ij} -\EE{U_{ij}}}} \geq t} \leq \exp\p{ - 2 n t^2}.
$$
By applying the union bound over $j \in \mathcal{J}$, we find that
$$ 
\PP{\sup_{j \in \mathcal{J}}\cb{ \abs{\frac{1}{n} \sum_{i=1}^n \p{U_{ij} -\EE{U_{ij}}}}} \geq t} \leq J\exp\p{ - 2 n t^2}.
$$
Hence:
$$ 
\begin{aligned}
&\EE{\sup_{j \in \mathcal{J}}\cb{ \p{ \frac{1}{n} \sum_{i=1}^n \p{U_{ij} -\EE{U_{ij}}}}}}^2 \\ 
&\;\;\;\;\leq \EE{\sup_{j \in \mathcal{J}}\cb{ \p{ \frac{1}{n} \sum_{i=1}^n \p{U_{ij} -\EE{U_{ij}}}}^2}} \\ 
&\;\;\;\; = \int_0^{\infty} \PP{\sup_{j \in \mathcal{J}} \abs{ \frac{1}{n} \sum_{i=1}^n \p{U_{ij} -\EE{U_{ij}}}} \geq \sqrt{t}}\, dt \\ 
&\;\;\;\; \leq \int_0^{\infty}  \min\cb{1, \;   J\exp\p{ - 2 n t}} \, dt \\ 
&\;\;\;\; = \frac{\log J}{2n} + J \sqb{ -\frac{1}{2n} \exp(-2nt)}_{u^*}^{\infty} \\ 
&\;\;\;\; = \frac{1+\log J}{2n} \lesssim_{\nu, \ubar{L}, \bar{U}} \frac{ (\log n)^2}{n}. 
\end{aligned}
$$
Above $u^*$ was chosen such that, $2nu^*= \log J$, i.e., $u^*= \log J/(2n)$.

Let us put everything together. In~\eqref{eq:marginal_density_relative_difference_inequality}, we showed that:
$$\frac{1}{n}\sum_{i=1}^{n} \EE{\abs{\frac{D_i(G_*) - D_i(\hG)}{D_i(\hG_*)}}\ind(A)} \leq  \frac{2}{n} \,+\, 8\eta \bar{U} \log n \,+\, \text{I} \,+\, \text{II},$$
with $\eta = 1/n$. We also showed that:
$$ \text{I}  \lesssim_{\nu, \ubar{L}, \bar{U}} \frac{\log n}{\sqrt{n}},$$
and that,
$$ \text{II} \leq 4 \sqrt{C} \frac{ \log n}{\sqrt{n}}.$$
Hence:
$$ \frac{1}{n}\sum_{i=1}^{n} \EE{\abs{\frac{D_i(G_*) - D_i(\hG)}{D_i(\hG_*)}}\ind(A)} \lesssim_{\nu, \ubar{L}, \bar{U}} \frac{\log n}{\sqrt{n}}.$$
\end{proof}

\subsection{Proof of Propositions~\ref{prop:asymptotically_uniform} and~\ref{manualprop:asymptotically_uniform}}
\label{subsec:proof_coro_asymptotically_uniform}

As explained after Proposition~\ref{manualprop:asymptotically_uniform}, for this result, our proof techniques and results for the empirical Bayes and compound settings deviate. For this reason, we provide the proofs separately for the two settings.

\begin{proof}[Proof of Proposition~\ref{prop:asymptotically_uniform}]
Let us fix $i \in \Hnull$.  We write $P_i = P_{\hG}(Z_i,S_i^2)$, $P_i^* = P_G(Z_i, S_i^2)$. As in the proof of Theorem~\ref{theo:pvalue_quality} in Supplement~\ref{subsec:proof_theo_pvalue_quality}, we let $\ubar{z} := \ubar{L}^{1/2} z_{1-\zeta/2}$. We also let
\begin{equation}
    \label{eq:sup_deviation_pvalues}
\Delta_i := \sup_{z: \abs{z} \geq \ubar{z}}\abs{P_*(z, S_i^2) - P_{\hG}(z, S_i^2)}.
\end{equation}
If $\abs{Z_i} < \ubar{z}$, then $P_i > 2(1-\Phi(\ubar{z}/\ubar{L}^{1/2}))= \zeta$, and so for any $t \in [0,\zeta]$ it holds that
$$\cb{ P_i \leq t} = \cb{P_i \leq t,\; \abs{Z_i} \geq \ubar{z}}.$$ 
Observe that,
$$
\begin{aligned}
\PP[G]{ P_i \leq t \cond S_1^2,\dotsc,S_n^2} &= \PP[G]{ \cb{P_i \leq t}\cap \cb{\abs{Z_i} \geq \ubar{z}} \cond S_1^2,\dotsc,S_n^2}\\ 
&= \PP[G]{ \cb{P_i^* \leq t + (P_i^* - P_i)}\cap \cb{\abs{Z_i} \geq \ubar{z}} \cond S_1^2,\dotsc,S_n^2} \\ 
&\leq \PP[G]{ P_i^* \leq t + \Delta_i \cond S_1^2,\dotsc,S_n^2} \\ 
& \stackrel{(*)}{\leq} t + \Delta_i.
\end{aligned}
$$
In step $(*)$ we used the fact that $P_i^*$ is uniformly distributed conditional on $S_1^2,\dotsc,S_n^2$ and that $\Delta_i$ is a function of $S_1^2,\dotsc,S_n^2$ only. Arguing analogously, we may also establish that:
$$ \PP[G]{ P_i \leq t \cond S_1^2,\dotsc,S_n^2} \geq t - \Delta_i.$$
Hence:
$$ \sup_{t \in [0, \zeta]} \abs{\PP[G]{ P_i \leq t \cond S_1^2,\dotsc,S_n^2} - t} \leq \Delta_i.$$
It follows that:
$$ \EE{\sup_{t \in [0, \zeta]} \abs{\PP[G]{ P_i \leq t \cond S_1^2,\dotsc,S_n^2} - t} } \leq \EE{\Delta_i}.$$
The proof of Theorem~\ref{theo:pvalue_quality} furnishes that
$$\EE{\Delta_i} \lesssim_{\ubar{L}, \bar{U}, \nu, \zeta}  \frac{(\log n)^{5/2}}{\sqrt{n}}.$$
Taking a maximum over all $i \in \Hnull$, we conclude with the first part of the proposition.
Asymptotic uniformity follows from conditional uniformity shown above as follows:
$$
\begin{aligned}
    &\max_{i \in \Hnull}\sup_{t \in [0,\zeta]} \abs{\PP[G]{P_i \leq t} - t} \\ 
    &\;\;\;\; = \max_{i \in \Hnull}\sup_{t \in [0,\zeta]} \abs{ \EE{ \PP[G]{P_i \leq t \cond S_1^2,\dotsc,S_n^2} - t}} \\ 
    &\;\;\;\; \leq \max_{i \in \Hnull} \EE{\sup_{t \in [0,\zeta]} \abs{  \PP[G]{P_i \leq t \cond S_1^2,\dotsc,S_n^2} - t}},
\end{aligned}
$$
and we already upper bounded the last expression in the previous part of the proposition.
\end{proof}

\begin{proof}[Proof of Proposition~\ref{manualprop:asymptotically_uniform}]

Let us fix $i \in \Hnull$. We write $P_i = P_{\hG}(Z_i,S_i^2)$ and $P_i^* = P_{\boldsigma}(Z_i, S_i^2)$.
Let $\zeta' = (1+\zeta)/2 \in (1/2,1)$ and take any $\delta \in (0, \zeta')$. Then by Lemma~\ref{lemm:deterministic_indicator_to_l1} below, 
it holds for $t \leq \zeta' - \delta$ that:
$$\ind(P_i \leq t)\;\leq \;\ind(P_i^* \leq t + \delta) + \frac{1}{\delta}\abs{P_i \land \zeta' - P_i^* \land \zeta'}.$$
Summing the above inequality over $i \in \Hnull$ and rearranging, we find that:
$$ \frac{1}{n} \sum_{i \in \Hnull} \ind(P_i \leq t) - t \leq \frac{1}{n} \sum_{i \in \Hnull} \ind(P_i^* \leq t+\delta) -(t+\delta) + \delta + \frac{1}{\delta}\frac{1}{n}\sum_{i \in \Hnull}\abs{P_i \land \zeta' - P_i^* \land \zeta'}.$$
By Theorem~\ref{theo:avg_significance_controlling}, it also holds that:
$$\frac{1}{n} \sum_{i \in \Hnull} \PP[\boldsigma]{P_i^* \leq t+\delta} \leq t+\delta.$$
Thus,
$$ 
\begin{aligned}
&\sup_{t \in [0, \zeta'-\delta]}\cb{\frac{1}{n} \sum_{i \in \Hnull} \ind(P_i \leq t) - t }\\ 
&\;\;\;\;\; \leq \sup_{t \in [0, \zeta'-\delta]}\abs{\frac{1}{n} \sum_{i \in \Hnull}\cb{ \ind(P_i^* \leq t+\delta) -\PP[\boldsigma]{P_i^* \leq t+\delta}}} + \delta + \frac{1}{\delta}\frac{1}{n}\sum_{i \in \Hnull}\abs{P_i \land \zeta' - P_i^* \land \zeta'}.
\end{aligned}
$$
In the next step we take expectations. By Lemma~\ref{lemm:dkw} below, it holds that:
$$\EE{\sup_{t}\abs{\frac{1}{n} \sum_{i \in \Hnull}\cb{ \ind(P_i^* \leq t+\delta) -\PP[\boldsigma]{P_i^* \leq t+\delta}}}} \leq \frac{n_0}{n} \sqrt{\frac{2e}{n_0}} \leq \sqrt{\frac{2e}{n}}.$$
Hence,
$$ \EE{\sup_{t \in [0, \zeta'-\delta]}\cb{\frac{1}{n} \sum_{i \in \Hnull} \ind(P_i \leq t) - t }} \leq  \sqrt{\frac{2e}{n}} + \delta + \frac{\varepsilon_n}{\delta},$$
where by
Theorem~\ref{manualtheorem:pvalue_quality}, 
$$\varepsilon_n = C(\nu, \ubar{L}, \bar{U}, \zeta') \frac{(\log n)^{5/2}}{\sqrt{n}},$$ 
and $C(\nu, \ubar{L}, \bar{U}, \zeta')$ is the constant specified in the theorem statement. 

Let $n$ be large enough such that $\varepsilon_n^{1/2} \leq (1-\zeta)/2$. For such $n$, we let $\delta = \varepsilon_n^{1/2}$ and it holds that $\delta < \zeta'$ and $\zeta'-\delta \geq \zeta$. We conclude that for $n$ large enough,
$$\EE{\sup_{t \in [0, \zeta]}\cb{\frac{1}{n} \sum_{i \in \Hnull} \ind(P_i \leq t) - t }} \leq  \sqrt{\frac{2e}{n}} + 2\sqrt{C(\nu, \ubar{L}, \bar{U}, \zeta')} \frac{(\log n)^{5/4}}{n^{1/4}}.
$$ 
By possibly inflating the constant, we have proved the second inequality of the proposition. The first inequality directly follows since:
$$    \sup_{t \in [0, \zeta]}\cb{ \EE{\frac{1}{n} \sum_{i \in \Hnull} \ind(P_i \leq t) - t }} \leq     \EE{\sup_{t \in [0, \zeta]}\cb{\frac{1}{n} \sum_{i \in \Hnull} \ind(P_i \leq t) - t }}. $$ 
\end{proof}

\subsection{Auxiliary lemmata on empirical distributions}
We first present a simple smoothing lemma.
\begin{lemm}
    \label{lemm:deterministic_indicator_to_l1}
Let $\zeta \in (0, 1)$, $\delta \in (0,\, \zeta)$ and $t \in [0,\,\zeta - \delta]$. Then:
$$ \ind(P_i  \leq t) - \ind(P_i^* \leq t+\delta) \leq \frac{1}{\delta}\abs{P_i \land \zeta - P_i^* \land \zeta}.$$
Similarly, for any $t \in (0,\,\zeta]$ and $\delta \in (0,\, t)$:
$$ \ind(P_i  \leq t) - \ind(P_i^* \leq t-\delta) \geq -\frac{1}{\delta}\abs{P_i \land \zeta - P_i^* \land \zeta}.$$
\end{lemm}

\begin{proof}
For the first part, we note:
    $$
    \begin{aligned}
    \ind(P_i  \leq t) - \ind(P_i^* \leq t+\delta) & \leq \ind(P_i \land \zeta \leq t) - \ind(P_i^*\land \zeta < t+\delta) \\ 
    & \leq \ind(P_i \land \zeta \leq t,\; P_i^* \land \zeta \geq t+\delta) \\
    & \leq \frac{1}{\delta}\abs{P_i \land \zeta - P_i^* \land \zeta}.
    \end{aligned}
    $$
For the second part, we argue analogously that:
$$
\begin{aligned}
\ind(P_i  \leq t) - \ind(P_i^* \leq t-\delta) 
& \geq  \ind(P_i \land \zeta < t) - \ind(P_i^*\land \zeta \leq t-\delta) \\ 
& \geq -\ind(P_i\land \zeta \geq t,\; P_i^* \land \zeta \leq t-\delta) \\ 
& \geq -\frac{1}{\delta}\abs{P_i \land \zeta - P_i^* \land \zeta}.
\end{aligned}
$$
\end{proof}

Next, we state a version of the celebrated Dvoretzky–Kiefer–Wolfowitz inequality that applies in the compound setting.

\begin{lemm}[Bretagnolle-Dvoretzky–Kiefer–Wolfowitz]
    \label{lemm:dkw}
Let $U_{i,n} \in [0,\,1]$, $i=1,\dotsc,n$ be independent (but not necessarily identically distributed) random variables. Then:
$$\PP{\sup_{t \in [0,1]} \abs{  \frac{1}{n} \sum_{i=1}^n  \cb{\ind( U_{i,n} \leq t) - \PP{ U_{i,n} \leq t}}} \geq \varepsilon} \leq 2e \exp(-n \varepsilon^2) \text{ for all } \varepsilon \geq 0.$$
Furthermore, 
$$\EE{\sup_{t \in [0,1]} \abs{  \frac{1}{n} \sum_{i=1}^n  \cb{\ind( U_{i,n} \leq t) - \PP{ U_{i,n} \leq t}}}} \leq \sqrt{\frac{2e}{n}}.$$
\end{lemm}
\begin{proof}
The high-probability inequality is due to~\citet{bretagnolle1981} along with Massart's
tight constant~\citep{massart1990tight}, see also 
\citet[Lemma 7.1]{donoho2006asymptotic}.\footnote{The extra factor ``$e$'' is not needed when all $U_{i,n}$ are also identically distributed. In that case, the inequality reduces to the standard Dvoretzky–Kiefer–Wolfowitz inequality with Massart's tight constant.}

For the bound on the expectation, we use the following standard argument. Let,
$$X := \sup_{t \in [0,1]} \abs{  \frac{1}{n} \sum_{i=1}^n  \p{\ind( U_{i,n} \leq t) - \PP{ U_{i,n} \leq t}}}.$$
By the high probability inequality above, it follows that $\PP{X \geq \varepsilon} \leq 2e \exp(-n \varepsilon^2)$ and so:
$$ \EE{X^2} = \int_0^{\infty} \PP{X^2 \geq t} \, dt = \int_0^{\infty} \PP{X \geq \sqrt{t}} \, dt \leq \int_0^{\infty} 2e \exp(-n t) \, dt = \frac{2e}{n}.$$
Hence by Jensen's inequality, $\EE{X} \leq \sqrt{2e/n}$.
\end{proof}

\section{Proofs for Section~\ref{subsec:BH}}

\subsection{Further background on the Benjamini-Hochberg procedure}
\label{sec:bh_empirical_process}
For the proofs it will be helpful to recall the following equivalent characterization of the Benjamini-Hochberg procedure~\eqref{eq:BH} through
an empirical process viewpoint~\citep{storey2004strong}. Let
\begin{equation}
    \label{eq:rejections_process}
    R_n(t) \coloneqq \sum_{i=1}^n \ind(P_i \leq t).
\end{equation}
Also let:
\begin{equation}
    \label{eq:bh_threshold}
\hat{t} \coloneqq \sup\cb{t \in [0,\alpha] : \frac{nt}{R_n(t)\lor 1} \leq \alpha}.
\end{equation}
Then the following holds for the Benjamini-Hochberg procedure applied at level $\alpha$: 
\begin{equation}
H_i \text{ is rejected } \; \text{if and only if} \; P_i \leq \hat{t}.
\end{equation}
In particular, the total number of rejections of BH, $R_n$ is equal to $R_n(\hat{t})$ 
(where we slightly abuse notation). In analogy to the empirical process of total discoveries~\eqref{eq:rejections_process}, 
we may also
define the empirical process of null discoveries
$$V_n(t) \coloneqq \sum_{i \in \Hnull} \ind(P_i \leq t),$$
so that the number of false discoveries of BH is equal to $V_n(\hat{t})$. Finally, with the above notation,
the false discovery rate is equal to $\EE{ V_n(\hat{t}) / (R_n(\hat{t})\lor 1)}$.

\subsection{Proof of Theorems~\ref{theo:limma_bh_controls_the_FDR} and~\ref{manualtheorem:limma_bh_controls_the_FDR}}
\label{subsec:proof_limma_bh_controls_the_FDR}

As in the proofs of Propositions~\ref{prop:asymptotically_uniform} and~\ref{manualprop:asymptotically_uniform}, we will provide separate proofs and arguments for the hierarchical and compound settings.

\begin{proof}[Proof of Theorem~\ref{theo:limma_bh_controls_the_FDR}]
Inspired by the proof of~\citet{roquain2022false}, we start by defining $R_{n,-(i)}$ as the number of rejections of the full empirical partially Bayes procedure (Algorithm~\ref{algo:np_limma}) when applied to
$$\cb{ (Z_1, S_1^2),\dotsc, (Z_{i-1}, S_{i-1}^2),  (Z_i^{\infty}, S_i^2), (Z_{i+1}, S_{i+1}^2), \dotsc ,(Z_n, S_n^2)}.$$
In words, we keep the full dataset the same, and we only replace $Z_i$ of the $i$-th unit by $Z_i^{\infty}$ defined as:
$$Z_i^{\infty} = \begin{cases} +\infty & \text{if } Z_i \geq 0, \\ -\infty & \text{if } Z_i < 0. \end{cases}.$$
We have the following result. On the event that the $i$-th hypothesis is rejected by Algorithm~\ref{algo:np_limma} on the original data, then it must be the case that $R_n = R_{n,-(i)} \geq 1$. This can be shown in two steps. First, we note that \smash{$\widehat{G}$} is the same for Algorithm~\ref{algo:np_limma} applied to the original dataset as well as the leave-one-out modified dataset (because \smash{$\widehat{G}$} is a function of $S_1^2,\dotsc,S_n^2$ only). Second, we adapt  the leave-one-out argument for the Benjamini-Hochberg procedure presented in~\citet{roquain2022false} and~\citet{ferreira2006benjamini}.

To further proceed, we also reintroduce the notation $\Delta_i$ which was previously defined in~\eqref{eq:sup_deviation_pvalues} (where we take $\zeta = \max\cb{3/4, (1+\alpha)/2}$). We get the following for $i \in \Hnull$:
$$
\begin{aligned}
&\EE[G]{ \frac{\ind(P_i \leq \alpha R_n/n)}{R_n \lor 1} \cond S_1^2,\dotsc,S_n^2, R_{n,-(i)}} \\
&\;\;\;\; =  \EE[G]{ \frac{\ind(P_i \leq \alpha R_{n,-(i)}/n)}{R_{n,-(i)}} \cond S_1^2,\dotsc,S_n^2,  R_{n,-(i)} } \\ 
&\;\;\;\;\leq \EE[G]{ \frac{\ind(P_i^* \leq \alpha R_{n,-(i)}/n + \Delta_i)}{R_{n,-(i)}} \cond S_1^2,\dotsc,S_n^2,  R_{n,-(i)} } \\ 
&\;\;\;\;\stackrel{(*)}{\leq} \frac{\alpha R_{n,-(i)}/n + \Delta_i}{R_{n,-(i)}}  \\ 
&\;\;\;\; = \frac{\alpha}{n} \,+\,    \frac{\Delta_i}{R_{n,-(i)}}.
\end{aligned}
$$
The argument in $(*)$ is analogous to the main argument in the proof of Proposition~\ref{prop:asymptotically_uniform}. By summing over $i\in \Hnull$ and rearranging, we get:
$$ \sum_{i \in \Hnull} \EE[G]{ \frac{\ind(P_i \leq \alpha R_n/n)}{R_n \lor 1} \cond S_1^2,\dotsc,S_n^2,  R_{n,-(i)}}\, - \, \frac{n_0}{n}\alpha \leq  \min\cb{\sum_{i \in \Hnull}  \frac{\Delta_i}{R_{n,-(i)}},\,1}.$$
Above we also used the fact that the LHS is $\leq 1$. By the tower property of  conditional expectations and noting that the RHS is $\geq 0$ with probability $1$:
$$
\p{\EE{ \mathrm{FDP}_n \cond S_1^2,\dotsc,S_n^2} - \frac{n_0}{n}\alpha}\lor 0 \; \leq \; \EE[G]{ \min\cb{\sum_{i \in \Hnull}  \frac{\Delta_i}{R_{n,-(i)}},\,1} \cond S_1^2,\dotsc,S_n^2}.
$$
Next note that $R_{n,-(i)} \geq R_n$ and so:
$$\min\cb{\sum_{i \in \Hnull} \frac{\Delta_i}{R_{n,-(i)}},\,1} \, \leq \,   \sum_{i \in \Hnull}\frac{\Delta_i}{n \kappa_n} \,+\, \ind(R_n < n\kappa_n).$$
Taking expectations and applying the bound on $\EE[G]{\Delta_i}$ shown in the proof of Theorem~\ref{theo:pvalue_quality}, it follows that:
$$ \EE{\p{\EE{ \mathrm{FDP}_n \cond S_1^2,\dotsc,S_n^2} - \frac{n_0}{n}\alpha}\lor 0} \; \leq \; \frac{1}{\kappa_n} C \frac{(\log n)^{5/2}}{\sqrt{n}}  \, +\, \eta_n.$$
Finally, by Jensen's inequality applied to the function $u \mapsto u \lor 0$, we also get that:
$$\mathrm{FDR}_n - \alpha \leq \p{\mathrm{FDR}_n - \alpha}\lor 0  \leq \EE{\p{\EE{ \mathrm{FDP}_n \cond S_1^2,\dotsc,S_n^2} - \frac{n_0}{n}\alpha}\lor 0}.$$
\end{proof}

\begin{proof}[Proof of Theorem~\ref{manualtheorem:limma_bh_controls_the_FDR}]

Throughout this proof we will use the empirical process characterization of the BH procedure from Section~\ref{sec:bh_empirical_process}. In the first step, let us note the following elementary inequality:
$$
\frac{V_n(\hat{t})}{R_n(\hat{t}) \lor 1} \; \leq \; \ind(R_n(\hat{t}) < n \kappa_n) \,+\, \frac{V_n(\hat{t})}{R_n(\hat{t})}\ind(R_n(\hat{t}) \geq n \kappa_n).
$$
From the above and the assumptions of the theorem, it follows that
$$ \mathrm{FDR}_n \leq \eta_n \,+\, \EE{\frac{V_n(\hat{t})}{R_n(\hat{t})}\ind(R_n(\hat{t}) \geq n\kappa_n)},$$
and we turn toward bounding the second term.
On the event $\cb{R_n(\hat{t}) \geq n \kappa_n}$, it holds that:
$$
\begin{aligned}
\frac{V_n(\hat{t})}{R_n(\hat{t})} \; = \;\underbrace{\frac{n \hat{t}}{R_n(\hat{t})}}_{\leq \alpha} \,+\,\frac{V_n(\hat{t}) - n \hat{t}}{R_n(\hat{t})} \;\leq \;\alpha \, + \, \frac{n}{n \kappa_n}\sup_{t \in [0,\, \alpha]}\cb{ \frac{1}{n}V_n(t) - t}.
\end{aligned}
$$
It thus follows that:
$$ \mathrm{FDR}_n \leq  \eta_n \,+\,\alpha\,+\, \frac{1}{\kappa_n}\EE{\sup_{t \in [0,\, \alpha]}\cb{ \frac{1}{n}V_n(t) - t}}.$$
The expression on the RHS is precisely the expectation we bounded in Proposition~\ref{manualprop:asymptotically_uniform}. We conclude by plugging in the bound therein.
\end{proof}

\subsection{Proof of Proposition~\ref{prop:dense_fdr_control}}
\label{subsec:proof_dense_fdr_control}

\begin{proof}
By our assumptions, there exists $s^{\infty} \in (0, t^{\infty})$ such that 
$$0<\alpha':= \frac{s^{\infty}}{H^{\infty}(s^{\infty})}  < \frac{u}{H^{\infty}(u)} < \alpha \, \text{ for all } u \in (s^{\infty}, t^{\infty}).$$
We define
\begin{equation}
    \label{eq:epsilon_definition_bh_coro}
\varepsilon := (\alpha - \alpha') H^{\infty}(s^{\infty}),
\end{equation}
and then we let $\delta^{\infty} \in (0, s^{\infty})$ be such that:
\begin{equation}
\label{eq:delta_definition_bh_coro}
\alpha\cb{ H^{\infty}(s^{\infty}) - H^{\infty}(s^{\infty} - \delta^{\infty})} = \frac{\varepsilon}{2}.
\end{equation}
We also let $\zeta = \max\cb{3/4, (1 + \alpha)/2} \in (0,1)$.
Next consider the following decomposition (following the empirical process notation of the BH procedure described in Section~\ref{sec:bh_empirical_process}):
$$
\begin{aligned}
s^{\infty} - \alpha\frac{R_n(s^{\infty})}{n} &= s^{\infty}- \alpha \frac{1}{n}\sum_{i =1}^n \ind(P_i  \leq s^{\infty})  \\ 
&\stackrel{(*)}{\leq} s^{\infty}  - \alpha\frac{1}{n}\sum_{i =1}^n \ind(P_i^*  \leq s^{\infty} - \delta^{\infty})\,+\, \frac{\alpha}{\delta^{\infty}} \frac{1}{n}\sum_{i =1}^n \abs{P_i \land \zeta - P_i^* \land \zeta}  \\ 
&= s^{\infty} - \alpha' H^{\infty}(s^\infty) \\
&\phantom{ = s^{\infty}\,} - \alpha \p{ \frac{1}{n} \sum_{i=1}^n\cb{ \ind(P_i^*  \leq s^{\infty} - \delta^{\infty}) - \PP[G,\mu_i]{P_i^*  \leq s^{\infty} - \delta^{\infty}}}} \\ 
&\phantom{ = s^{\infty}\,} - \alpha \p{ \frac{1}{n} \sum_{i=1}^n\cb{ \PP[G,\mu_i]{P_i^*  \leq s^{\infty} - \delta^{\infty}} - H^{\infty}(s^{\infty} - \delta^{\infty})}} \\ 
&\phantom{ = s^{\infty}\,} +  \alpha\cb{ H^{\infty}(s^{\infty}) - H^{\infty}(s^{\infty} - \delta^{\infty})} \\ 
&\phantom{ = s^{\infty}\,} + (\alpha'-\alpha)H^{\infty}(s^{\infty}) \\ 
&\phantom{ = s^{\infty}\,} + \frac{\alpha}{\delta^{\infty}} \frac{1}{n}\sum_{i =1}^n \abs{P_i \land \zeta - P_i^* \land \zeta}.
\end{aligned}
$$
In $(*)$ we used the second smoothing inequality presented in Lemma~\ref{lemm:deterministic_indicator_to_l1}. To simplify the inequality derived above, we define $n_{\varepsilon} \in \mathbb N$ to be such that:
$$ 
\alpha \abs{ \frac{1}{n} \sum_{i=1}^n\cb{ \PP[G,\mu_i]{P_i^*  \leq s^{\infty} - \delta^{\infty}} - H^{\infty}(s^{\infty} - \delta^{\infty})}} \leq \frac{\varepsilon}{4}\, \text{ for all }\, n \geq n_{\varepsilon}.
$$
Let us also define:
$$ X_n := \alpha\p{\sup_{t \in [0,1]}\abs{\frac{1}{n}\sum_{i=1}^n \cb{ \ind(P_i^*  \leq t) - \PP[G,\mu_i]{P_i^*  \leq t}}} \, + \,\frac{1}{\delta^{\infty}} \frac{1}{n}\sum_{i =1}^n \abs{P_i \land \zeta - P_i^* \land \zeta}}.$$
Then, also noting that $s^{\infty} - \alpha' H^{\infty}(s^\infty) = 0$, as well as~\eqref{eq:epsilon_definition_bh_coro} and~\eqref{eq:delta_definition_bh_coro}, we have:
$$
s^{\infty} - \alpha\frac{R_n(s^{\infty})}{n} \leq -\frac{\varepsilon}{4} +  X_n\, \text{ for all }\, n \geq n_{\varepsilon}.
$$
Define the event $A_n := \cb{s^{\infty} - \alpha\frac{R_n(s^{\infty})}{n} > 0}$.  For $n \geq n_{\varepsilon}$, it holds that:
$$
\PP[G,\boldmu]{A_n} \leq \PP[G,\boldmu]{X_n > \frac{\varepsilon}{4}} \leq \frac{4}{\varepsilon} \EE[G,\boldmu]{X_n}.
$$
On the other hand, by Lemma~\ref{lemm:dkw} and Theorem~\ref{theo:pvalue_quality}, there exists a constant $C>0$ such that,
$$\EE[G,\boldmu]{X_n} \leq C \frac{(\log n)^{5/2}}{\sqrt{n}},$$
and so for another constant $C'>0$:
$$\PP[G,\boldmu]{A_n} \leq C' \frac{(\log n)^{5/2}}{\sqrt{n}}.$$
On the complement $A_n^c$ of the event $A_n$, it holds that $s^{\infty} - \alpha\frac{R_n(s^{\infty})}{n} \leq 0$, i.e.,
$$R_n(s^{\infty}) \geq \frac{n}{\alpha}s^{\infty} =n \kappa_n, \text{ where } \kappa_n := \frac{s^{\infty}}{\alpha}.$$
Applying Theorem~\ref{theo:limma_bh_controls_the_FDR} with $\kappa_n$ as above and $\eta_n = \PP[G,\boldmu]{A_n}$, we find that there exists a constant $C''>0$ such that:
$$ \mathrm{FDR}_n - \frac{n_0}{n}\alpha \leq C'' \frac{(\log n)^{5/2}}{\sqrt{n}}.$$
Multiplying the above by $n^{1/2} (\log n)^{-\xi}$ with $\xi > 5/2$, and taking the limsup as $n \to \infty$, we conclude with the asymptotic guarantee on the false discovery rate of the empirical partially Bayes procedure.
\end{proof}

\subsection{Proof of Proposition~\ref{prop:mimicking_oracle}}
\label{subsec:proof_mimicking_oracle}

\begin{proof}
By the compound Dvoretsky-Kiefer-Wolfowitz inequality in Lemma~\ref{lemm:dkw} (or alternatively, using a simpler Glivenko-Cantelli argument), the smoothing inequality in Lemma~\ref{lemm:deterministic_indicator_to_l1}, Theorem~\ref{theo:pvalue_quality}, and the continuity of $H^{\infty}$, we can show that:
\begin{equation}
    \label{eq:gc_to_H}
    \begin{aligned}
\sup_{t \in [0,\,1]}\abs{\frac{1}{n}\sum_{i=1}^n \ind(P_i \leq t)   - H^{\infty}(t)}  \, &\stackrel{\mathbb P}{\to}\, 0 \; \text{ as } n \to \infty, \\ 
\sup_{t \in [0,\,1]}\abs{\frac{1}{n_0}\sum_{i \in \Hnull} \ind(P_i \leq t) - t}  \,&\stackrel{\mathbb P}{\to}\, 0\; \text{ as } n \to \infty.
    \end{aligned}
\end{equation}
The result in~\eqref{eq:gc_to_H} along with our assumptions on $H^{\infty}(\cdot)$ allows us to directly call upon the results in~\citet[Section 3]{ferreira2006benjamini}. In particular, the stated limits (as $n\to \infty$) of $\mathrm{Pow}_n$ and $\mathrm{FNDR}_n$ follow from~\citet[Corollary 3.3] {ferreira2006benjamini} along with the dominated convergence theorem. We can also show that~\eqref{eq:gc_to_H} holds if we replace $P_i$ by the oracle p-values $P_i^* = P_G(Z_i, S_i^2)$. Hence, $\mathrm{Pow}_n^{\text{or}}$ and $\mathrm{FNDR}_n^{\text{or}}$ of the oracle procedure also converge to the same limits.

\end{proof}

\section{Proof of Proposition~\ref{prop:joint_hierarchical} in Section~\ref{sec:joint_hierarchical}}
\label{sec:proof_joint_hierarchical}

\begin{proof} 
Let $(\mu_i, \sigma_i^2, Z_i, S_i^2)$ be generated as in the statement of the proposition, that is,
suppose $(\mu_i, \sigma_i^2) \sim \Pi$~ as in~\eqref{eq:joint_EB} and $(Z_i, S_i^2)$ are 
generated as in~\eqref{eq:full_sampling} conditional on $\mu_i, \sigma_i$.
It will be helpful to also generate (compare to the notation in~\eqref{eq:conditional_pvalue}):
$$Z_i^{H_0} \cond \mu_i, \sigma_i^2, Z_i, S_i^2 \sim \mathcal{N}(0, \sigma_i^2).$$
Further note that the distribution of $(\sigma_i^2, Z_i^{H_0}, S_i^2)$ may be factorized as follows:
$$ \sigma_i^2 \sim G(\cdot; \Pi),\;\; (Z_i^{H_0}, S_i^2) \mid \sigma_i^2 \sim \mathcal{N}(0, \sigma_i^2) \otimes \frac{\sigma_i^2}{\nu}\chi^2_{\nu},$$
where $G(\cdot; \Pi)$ has been defined in~\eqref{eq:marginalized_G}. The distribution of  $(\sigma_i^2, Z_i^{H_0}, S_i^2)$
depends on $\Pi$ only through $G(\cdot; \Pi)$. 

By~\eqref{eq:indep_under_null}, it holds that $\sigma_i^2 \cond (\mu_i=0) \; \sim \; G(\cdot; \Pi)$. 
Since the distribution of $S_i^2$ conditional on $\sigma_i^2, \mu_i$ only depends on $\sigma_i^2$, it also holds that:
$$ (\sigma_i^2, S_i^2) \cond (\mu_i=0)  \;\; \stackrel{\mathcal{D}}{=}\;\; (\sigma_i^2, S_i^2).$$
This also implies that:
$$ (Z_i, \sigma_i^2, S_i^2) \cond (\mu_i = 0) \;\; \stackrel{\mathcal{D}}{=}\;\; (Z_i^{H_0}, \sigma_i^2, S_i^2).$$
Thus, $S_i^2$-almost surely,
\begin{equation}
    \label{eq:matching_conditional_prop_joint}
Z_i \cond S_i^2, (\mu_i = 0)\;\; \stackrel{\mathcal{D}}{=}\;\; Z_i^{H_0} \cond S_i^2.
\end{equation}
We are ready to conclude. For any $t \in [0,1]$, it almost surely holds that:
$$
\begin{aligned}
\PP[\Pi]{P_{ G(\cdot; \Pi)}(Z_i, S_i^2) \leq t \cond S_i^2, \mu_i=0  } &\stackrel{(*)}{=} \PP[\Pi]{P_{ G(\cdot; \Pi)}(Z_i^{H_0}, S_i^2) \leq t \cond S_i^2} \\ 
&\stackrel{(**)}{=} \PP[G(\cdot; \Pi)]{P_{ G(\cdot; \Pi)}(Z_i^{H_0}, S_i^2) \leq t \cond S_i^2} \stackrel{(***)}{=} t.
\end{aligned}
$$
In $(*)$ we used~\eqref{eq:matching_conditional_prop_joint}, in $(**)$ we recalled the observation
that the distribution of $(Z_i^{H_0}, S_i^2)$ only depends on $G(\cdot; \Pi)$, and in $(***)$ we used Proposition~\ref{prop:monotonicity}.
\end{proof}

\section{Further simulation results}
\label{sec:further_sim}

\subsection{Sensitivity analysis for data-driven support of $\widehat{G}$}
\label{subsec:sensitivity}
As explained in Remark~\ref{rema:computation}, our implementation optimizes the NPMLE objective over the class of priors $G$ supported on $[a , b]$, where $b = \max_i S_i^2$ and $a$ is set to the $1\%$ quantile of the $S_i^2$ by default. In this section, we conduct a sensitivity analysis to assess the impact of the choice of $a$ on the performance of our method in the simulation study of Section~\ref{sec:simulations}. We consider three choices for $a$:

\begin{enumerate}
\item $\mathrm{NPMLE}(\mathrm{min})$: $a = \min S_i^2$.
\item $\mathrm{NPMLE}(\mathrm{0.01})$: $a$ is the $1\%$ quantile of the $S_i^2$ (our default choice).
\item $\mathrm{NPMLE}(\mathrm{0.1})$: $a$ is the $10\%$ quantile of the $S_i^2$.
\end{enumerate}
We assess the performance of these alternatives in terms of false discovery rate (FDR), power, and the conditionality metric $\mathrm{MinSVarFP}(\leq 0.2)$ defined in~\eqref{eq:minsamplevar}.

\paragraph{Sensitivity in terms of false discovery rate:} The false discovery rate was robust to the choice of $a$, with $\mathrm{NPMLE}(\mathrm{min})$ and $\mathrm{NPMLE}(\mathrm{0.1})$ having FDR within $0.005$ of the FDR of $\mathrm{NPMLE}(\mathrm{0.01})$ across all settings in Figs.~\ref{fig:main_simulations} and~\ref{fig:adv_simulations}. Recall that the target FDR level was set to $0.1$.

\paragraph{Sensitivity in terms of power:}  Power was also relatively robust to the choice of $a$, with the exception of a slight decrease in power for $\mathrm{NPMLE}(\mathrm{0.1})$ in some settings with larger degrees of freedom $\nu$. The difference in power between $\mathrm{NPMLE}(\mathrm{0.1})$ compared to the power of $\mathrm{NPMLE}(\mathrm{0.01})$ (resp. $\mathrm{NPMLE}(\mathrm{min})$) was less than $0.005$ in all settings except for the two shown in Table~\ref{tbl:power_sensitivity}.

\setlength{\tabcolsep}{3pt} 
\begin{table}
    \centering
    \caption{\textbf{Sensitivity of power to the lower support point of $\widehat{G}$:} The table shows the power of three variations of the empirical partially Bayes approach, corresponding to different choices for the lower end $a$ of the support of the estimated prior $\widehat{G}$. $\mathrm{NPMLE}(\mathrm{0.01})$ sets $a$ to the $1\%$ quantile of the $S_i^2$ and is used throughout the main text. $\mathrm{NPMLE}(\mathrm{min})$ and $\mathrm{NPMLE}(\mathrm{0.1})$ set $a$ to $\min S_i^2$ and the $10\%$ quantile of the $S_i^2$, respectively. The settings shown in this table are the only ones from Figs.~\ref{fig:main_simulations} and~\ref{fig:adv_simulations} where the difference in power between any two method variations exceeded $0.005$. In these settings with larger $\nu$, $\mathrm{NPMLE}(\mathrm{0.1})$ exhibits a slight decrease in power compared to the other choices.}
    \small 
    \begin{tabular}{lllrrrr}
    \hline
    $G$ & $\nu$ & Ordering &  $\mathrm{Power}\cb{\mathrm{NPMLE}(\mathrm{min})}$ & $\mathrm{Power}\cb{\mathrm{NPMLE}(\mathrm{0.01})}$ & $\mathrm{Power}\cb{\mathrm{NPMLE}(\mathrm{0.1})}$ \\ \hline
    Sc-inv-$\chi^2$ & 64 & adversarial & 0.46 & 0.46 & 0.44 \\
    Sc-inv-$\chi^2$  & 32 & adversarial & 0.43 & 0.43 & 0.42 \\ \hline
    \end{tabular}
    \label{tbl:power_sensitivity}
\end{table}

\paragraph{Sensitivity in terms of conditionality properties:} We assess the conditionality properties of the methods using the metric $\mathrm{MinSVarFP}(\leq 0.2)$ defined in equation~\eqref{eq:minsamplevar}. This metric captures the type-I error rate (at $P_i \leq 0.2$) for the true null hypothesis with the smallest sample variance. If the null p-values are exactly uniform conditional on the sample variances, as is the case for the oracle p-values in Proposition~\ref{prop:monotonicity}, then $\mathrm{MinSVarFP}(\leq 0.2)$ should equal $0.2$. Figure~\ref{fig:sensitivity_conditionality} shows the sensitivity of this metric to the choice of the lower support point $a$ of $\widehat{G}$.

\begin{figure}
    \centering
    \begin{tabular}{@{}l@{}}
        a) setting with iid variances (non-adversarial)\\
        \includegraphics[width=0.99\linewidth]{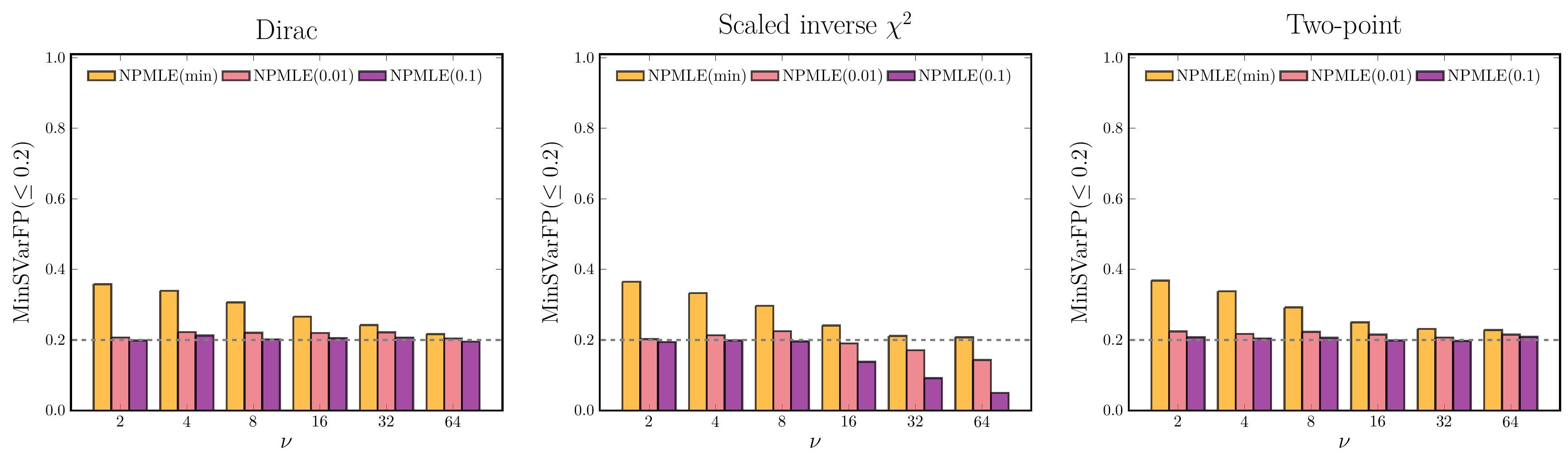} \\[6pt]
        b) setting with adversarially ordered  variances\\
        \makebox[\linewidth][r]{\includegraphics[width=0.97\linewidth]{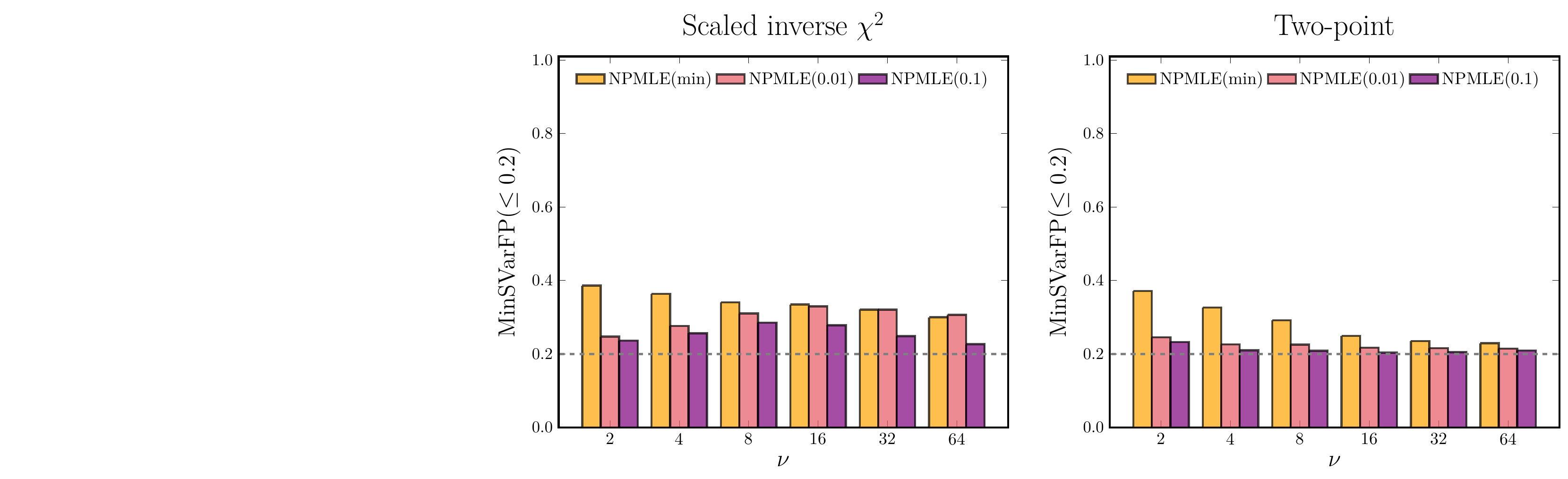}}
    \end{tabular}
    \caption{\textbf{Sensitivity of the conditionality metric $\mathrm{MinSVarFP}(\leq 0.2)$ to the choice of lower support point of $\widehat{G}$:} The figure shows the metric defined in equation~\eqref{eq:minsamplevar}, which measures the probability that the p-value of the true null hypothesis with the smallest sample variance is less than or equal to $0.2$. If the null p-values are exactly uniform conditional on the observed sample variances, this metric should be equal to $0.2$. $\mathrm{NPMLE}(\mathrm{0.01})$, which sets the lower end $a$ of the support of $\widehat{G}$ to the $1\%$ quantile of the $S_i^2$, is the default choice used in the main text. $\mathrm{NPMLE}(\mathrm{min})$ and $\mathrm{NPMLE}(\mathrm{0.1})$ set $a$ to $\min S_i^2$ and the $10\%$ quantile of the $S_i^2$, respectively. Panel a) shows results for settings with iid variances as in Fig.~\ref{fig:main_simulations}, while panel b) shows results for adversarially ordered variances as in Fig.~\ref{fig:adv_simulations}. In the non-adversarial setting, for the scaled inverse $\chi^2$ prior with small $\nu$, $\mathrm{NPMLE}(\mathrm{min})$ yields $\mathrm{MinSVarFP}(\leq 0.2)$ considerably higher than $0.2$, while for $\nu \geq 16$, $\mathrm{NPMLE}(\mathrm{0.1})$ results in $\mathrm{MinSVarFP}(\leq 0.2)$ much lower than $0.2$. $\mathrm{NPMLE}(\mathrm{0.01})$ achieves $\mathrm{MinSVarFP}(\leq 0.2) \approx 0.2$ in most non-adversarial settings. In the adversarial setting, all methods tend to have $\mathrm{MinSVarFP}(\leq 0.2)$ above $0.2$, with $\mathrm{NPMLE}(\mathrm{min})$ being most liberal.}
    \label{fig:sensitivity_conditionality}
\end{figure}

In the non-adversarial setting, for the scaled inverse $\chi^2$ prior with small degrees of freedom $\nu$, $\mathrm{NPMLE}(\mathrm{min})$ yields $\mathrm{MinSVarFP}(\leq 0.2)$ considerably higher than $0.2$, indicating that the p-value for the null hypothesis with the smallest sample variance is not uniformly distributed. Conversely, for $\nu \geq 16$, $\mathrm{NPMLE}(\mathrm{0.1})$ results in $\mathrm{MinSVarFP}(\leq 0.2)$ much lower than $0.2$, suggesting that the  p-value for the null hypothesis with smallest sample variance is overly conservative.  The $1\%$ quantile choice used by $\mathrm{NPMLE}(\mathrm{0.01})$ achieves $\mathrm{MinSVarFP}(\leq 0.2) \approx 0.2$ in most settings. In the adversarial setting, all methods tend to have $\mathrm{MinSVarFP}(\leq 0.2)$ above $0.2$, with $\mathrm{NPMLE}(\mathrm{min})$ being most liberal.

\paragraph{Summary:} To summarize, our sensitivity analysis demonstrates that the $1\%$ quantile choice for the lower end of the support of $\widehat{G}$ performs well in terms of FDR control, power, and conditionality properties. The FDR and power were relatively robust to the choice of $a$, with $\mathrm{NPMLE}(\mathrm{0.01})$ and $\mathrm{NPMLE}(\mathrm{min})$ showing very similar performance, and $\mathrm{NPMLE}(\mathrm{0.1})$ exhibiting only a slight decrease in power in some settings with larger $\nu$. In terms of the conditionality metric $\mathrm{MinSVarFP}(\leq 0.2)$, $\mathrm{NPMLE}(\mathrm{0.01})$ achieved values close to the desired $0.2$ level in most non-adversarial settings, while $\mathrm{NPMLE}(\mathrm{min})$ and $\mathrm{NPMLE}(\mathrm{0.1})$ showed suboptimal behavior in some cases, being overly liberal or conservative, respectively. These results support the use of the $1\%$ quantile as the default choice for $a$ in our implementation.

\subsection{Further simulation settings}
\label{subsec:further_simulations}

In Section~\ref{sec:simulations}, we considered the following setting for the signal strength of non-null hypotheses,
\begin{equation}
\mu_i \mid \sigma_i^2, (\mu_i \neq 0) \sim  \nn(0,\, 16 \sigma_i^2).
\label{eq:conjugate_signal}
\end{equation}
As announced in Section~\ref{sec:simulations}, in this supplement we repeat the simulations leading to Figure~\ref{fig:main_simulations}, replacing the above signal strength distribution (and keeping all other simulation settings identical).

\begin{figure}
    \centering
    \includegraphics[width=0.99\linewidth]{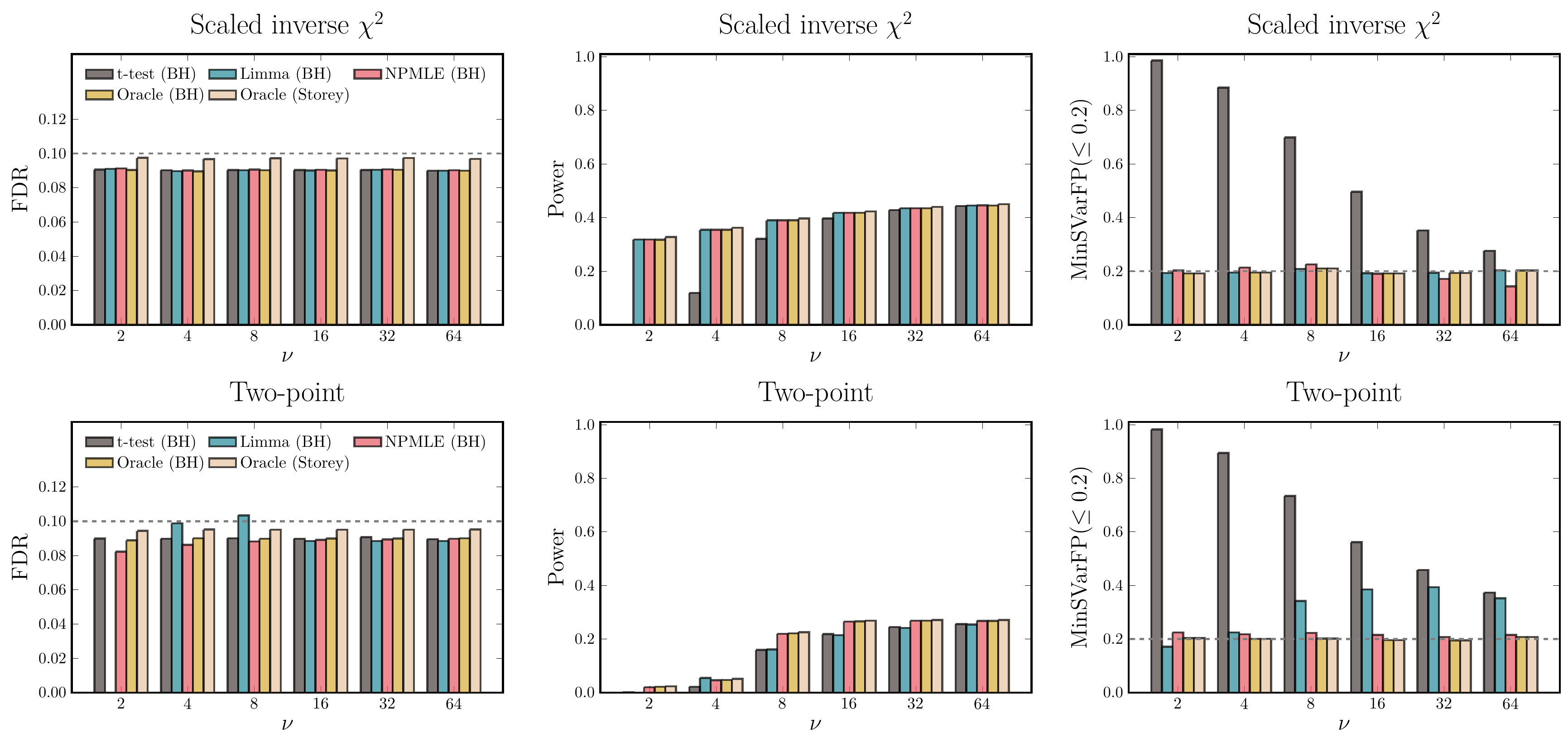}
    \caption{\textbf{Simulation results with homoscedastic normal alternative effect sizes:} {\small 
    The panels in this figure are analogous to the panels of the last two rows of Fig.~\ref{fig:main_simulations}. 
    In contrast to Fig.~\ref{fig:main_simulations}, the alternative signals were generated according to~\eqref{eq:normal_signal} rather than~\eqref{eq:conjugate_signal}.
    }}
    \label{fig:suppl_simulations_normal_signal_strength}
\end{figure}

We first consider the following alternative signal distribution:
\begin{equation}
\mu_i \mid \sigma_i^2, (\mu_i \neq 0) \sim \mathcal{N}(0,\, 16).
\label{eq:normal_signal}
\end{equation}
The results are shown in Figure~\ref{fig:suppl_simulations_normal_signal_strength}.\footnote{The figure omits the first row of Fig.~\ref{fig:main_simulations}. In the setting of that row, it holds that $\sigma_i^2 \simiid \delta_1$, that is, all $\sigma_i^2=1$. 
For that case the alternative signal distributions implied by~\eqref{eq:conjugate_signal} and~\eqref{eq:normal_signal} are identical.
} The results are qualitatively similar to those of Fig.~\ref{fig:main_simulations}. The empirical partially Bayes approach continues to show good control of the false discovery rate, good power, and conditionality properties. In the case of the two-point prior, we do observe some differences, however. In particular, all methods have substantially less power in that setting. The reason is that the signal-to-noise ratio for alternatives  with $\sigma_i^2 = 10$ is dominated by the large variance. Shrinkage on top of $S_i^2$ does not help reject any such alternative. However, the empirical partially Bayes methods still showcase gains in power compared to a standard t-test owing to the shrinkage effects for alternatives with $\sigma_i^2=1$.

\begin{figure}
    \centering
    \includegraphics[width=0.99\linewidth]{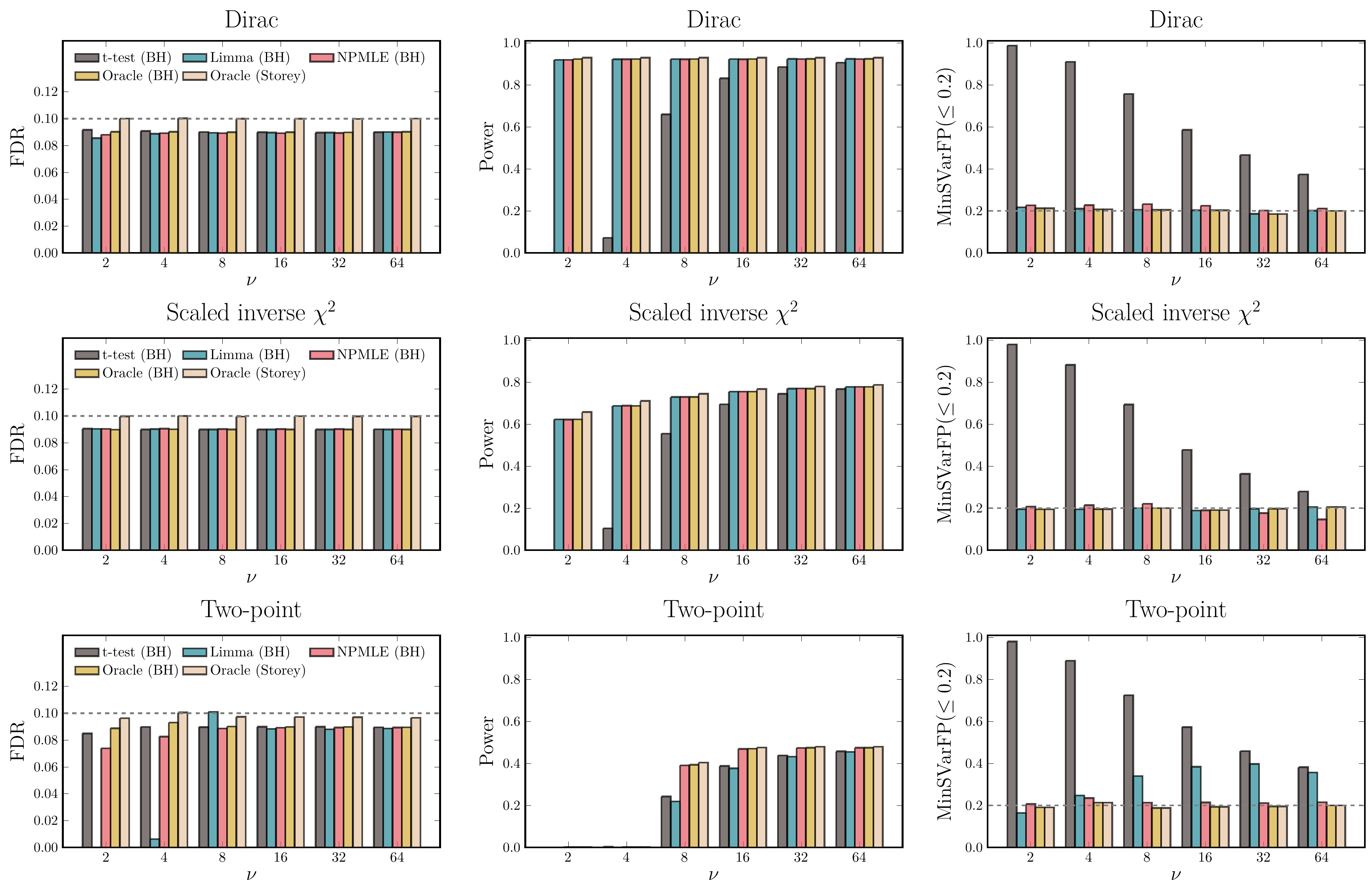}
    \caption{\textbf{Simulation results with constant alternative signal strength:} {\small 
    The panels in this figure are analogous to the panels of Fig.~\ref{fig:main_simulations}. 
    In contrast to Fig.~\ref{fig:main_simulations}, the alternative signals were all set equal to $4$ as specified in~\eqref{eq:dirac_signal}.
    }}
    \label{fig:suppl_simulations_dirac_signal_strength}
\end{figure}

We next consider the Dirac signal distribution:
\begin{equation}
\mu_i \mid \sigma_i^2, (\mu_i \neq 0) \sim \delta_4,
\label{eq:dirac_signal}
\end{equation}
in which $\mu_i=4$ for all non-null signals (and $\mu_i=0$ for null signals). Results are shown in Figure~\ref{fig:suppl_simulations_dirac_signal_strength}. Overall, the main conclusions remain similar to the case of Fig.~\ref{fig:main_simulations}. One notable difference is that for the two-point prior for $\sigma_i^2$ and for $\nu=2,4$, all methods (including the oracle methods) have no power.

\end{document}